\documentclass[11pt,a4paper,reqno]{amsart}
\usepackage{amsfonts,amsthm, amscd, epsfig, amsmath, amssymb,enumerate}
\usepackage{amstext}
\usepackage{xspace}
\usepackage{graphicx}
\usepackage{url}
\usepackage{enumerate}
\usepackage[all,2cell,ps]{xy}
\usepackage{xcolor}
\usepackage{subcaption}
\graphicspath{{./figure/}}
\usepackage{cite}
\usepackage{afterpage} 
\usepackage{faktor}

\newtheorem{itlemma}{Lemma}[section]
\newtheorem{itproposition}[itlemma]{Proposition}
\newtheorem{itfact}[itlemma]{Fact}
\newtheorem{theorem}[itlemma]{Theorem}
\newtheorem{itcorollary}[itlemma]{Corollary}
\newtheorem{itremark}[itlemma]{Remark}
\newtheorem{itremarks}[itlemma]{Remarks}
\newtheorem{itdefinition}[itlemma]{Definition}
\newtheorem{itexample}[itlemma]{Example}

\newenvironment{fact}{\begin{itfact}\rm}{\end{itfact}}
\newenvironment{claim}{\begin{itclaim}\rm}{\end{itclaim}}
\newenvironment{lemma}{\begin{itlemma}}{\end{itlemma}}
\newenvironment{remark}{\begin{itremark}\rm}{\end{itremark}}
\newenvironment{remarks}{\begin{itremarks} \rm}{\end{itremarks}}
\newenvironment{corollary}{\begin{itcorollary}}{\end{itcorollary}}
\newenvironment{proposition}{\begin{itproposition}}{\end{itproposition}}
\newenvironment{definition}{\begin{itdefinition}\rm}{\end{itdefinition}}
\newenvironment{example}{\begin{itexample}\rm}{\end{itexample}}
\newcommand{\be}[1]{\begin{equation}\label{#1}}
\newcommand{\ee}{\end{equation}}
\newcommand{\bl}[1]{\begin{lemma}\label{#1}}
\newcommand{\br}[1]{\begin{remark}\label{#1}}
\newcommand{\brs}[1]{\begin{remarks}\label{#1}}
\newcommand{\bt}[1]{\begin{theorem}\label{#1}}
\newcommand{\bd}[1]{\begin{definition}\label{#1}}
\newcommand{\bp}[1]{\begin{proposition}\label{#1}}
\newcommand{\bfact}[1]{\begin{fact}\label{#1}}
\newcommand{\bc}[1]{\begin{corollary}\label{#1}}
\newcommand{\bex}[1]{\begin{example}\label{#1}}
\newcommand{\ec}{\end{corollary}}
\newcommand{\efact}{\end{fact}}
\newcommand{\eex}{\end{example}}
\newcommand{\el}{\end{lemma}}
\newcommand{\er}{\end{remark}}
\newcommand{\ers}{\end{remarks}}
\newcommand{\et}{\end{theorem}}
\newcommand{\ed}{\end{definition}}
\newcommand{\ep}{\end{proposition}}
\newcommand{\epr}{\end{proof}}
\newcommand{\bpr}{\begin{proof}}
\newcommand{\bcl}[1]{\begin{claim}\label{#1}}
\newcommand{\ecl}{\end{claim}}

\newcommand{\ecs}{\end{corollary}}
\newcommand{\eers}{\end{exercise}}
\newcommand{\eexs}{\end{example}}
\newcommand{\eems}{\end{example}}
\newcommand{\els}{\end{lemma}}
\newcommand{\eles}{\end{lemmaex}}
\newcommand{\ets}{\end{theorem}}
\newcommand{\eds}{\end{definition}}
\newcommand{\eps}{\end{proposition}}

\newcommand{\bi}{\begin{itemize}}
\newcommand{\ei}{\end{itemize}}
\newcommand{\ben}{\begin{enumerate}}
\newcommand{\een}{\end{enumerate}}

\def\vbar{\mathchoice{\vrule height6.3ptdepth-.5ptwidth.8pt\kern-.8pt}
   {\vrule height6.3ptdepth-.5ptwidth.8pt\kern-.8pt}
   {\vrule height4.1ptdepth-.35ptwidth.6pt\kern-.6pt}
   {\vrule height3.1ptdepth-.25ptwidth.5pt\kern-.5pt}}
\def\fudge{\mathchoice{}{}{\mkern.5mu}{\mkern.8mu}}
\def\bbc#1#2{{\rm \mkern#2mu\vbar\mkern-#2mu#1}}
\def\bbb#1{{\rm I\mkern-3.5mu #1}}
\def\bba#1#2{{\rm #1\mkern-#2mu\fudge #1}}
\def\bb#1{{\count4=`#1 \advance\count4by-64 \ifcase\count4\or\bba A{11.5}\or
   \bbb B\or\bbc C{5}\or\bbb D\or\bbb E\or\bbb F \or\bbc G{5}\or\bbb H\or
   \bbb I\or\bbc J{3}\or\bbb K\or\bbb L \or\bbb M\or\bbb N\or\bbc O{5} \or
   \bbb P\or\bbc Q{5}\or\bbb R\or\bbc S{4.2}\or\bba T{10.5}\or\bbc U{5}\or
   \bba V{12}\or\bba W{16.5}\or\bba X{11}\or\bba Y{11.7}\or\bba Z{7.5}\fi}}

\def \R {{\mathbb R}}

\def\F{{\mathcal{F}}}

\def\1{{\bf 1}}

\def\ec{\`e }

\begin{document}

\title[Interacting systems]{Interacting non-linear reinforced
  stochastic processes: synchronization and
  no-synchronization}

\begin{abstract}
\textit{Rich get richer rule} comforts previously often chosen actions. What is happening to the evolution of  individual inclinations to choose an action when agents do interact ? Interaction tends to homogenize while each individual dynamics tends to reinforce its own position.
Interacting stochastic systems of reinforced processes were recently considered in many papers, where the asymptotic behavior was proven to exhibit a.s. synchronization. We consider in this paper models where, even if interaction among agents is present, absence of synchronization may happen due to the choice of an individual non-linear reinforcement. We show how these systems can naturally be considered as models for coordination games~\cite{mas2016behavioral}, technological or opinion dynamics.
\end{abstract}

\maketitle

    \centerline{Irene Crimaldi\footnote{IMT School for Advanced
        Studies Lucca, Piazza San Ponziano 6, 55100 Lucca, Italy,
        \url{irene.crimaldi@imtlucca.it}}, 
         Pierre-Yves~Louis\footnote{Laboratoire de Math\'ematiques et
        Applications UMR 7348, Universit\'e de Poitiers et CNRS,
          B\^at. H3, SP2MI - Site du Futuroscope, 
11 boulevard Marie et Pierre Curie, 
TSA 61125, 
86073 Poitiers Cedex 9 France,
        \url{pierre-yves.louis@math.cnrs.fr}}, Ida
      G.~Minelli\footnote{Dipartimento di Ingegneria e Scienze
        dell'Informazione e Matematica, Universit\`a degli Studi
        dell'Aquila, Via Vetoio (Coppito 1), 67100 L'Aquila, Italy,
        \url{idagermana.minelli@univaq.it}}}




\date{\today}

\section{Introduction}\label{introduction}

The stochastic evolution of systems composed by agents which interact
among each other has always been of great interest in several
scientific fields. For example, economic and social sciences deal with
agents that make decisions under the influence of other
agents or some external media. Moreover, preferences and beliefs are partly transmitted by
means of various forms of social interaction and opinions are driven
by \emph{social influence} \textit{i.e.} the tendency of individuals to become more similar when they
interact ({\it e.g.}~\cite{alos, arenas, blume, lewis}).
\medskip

\indent A natural description of such systems is provided by Agent Based Modeling~\cite{silverman2018methodological,axelrod}, where they are modeled as a collection of decision-making agents with a set of rules (defined at a microscopic level) that includes several issues, for instance learning and adaptation, environmental constraints and so on~\cite{axelrod, Bon02,Duffy}. The character of interactions and influence among agents (or among groups of agents) are crucial in these models
and 
can give rise to emergent phenomena 
  observed in the systems~\cite{arthur2015process,Arthur107}. 
  Agent-based models
  abound in a variety of disciplines, including economics, game
  theory, sociology and political science (\textit{e.g.}~\cite{BL03,
    bottazzi2007modeling, dosi2006evolutionary, fagiolo2004matching,
    fagiolo2003exploitation, fagiolo2005endogenous, KMR93, Ell93,
    SP00, EL04, OMH}).
However, although they are often effective in describing real situations, these models are mainly computational. Due to the many variables involved,   
it is indeed usually hard to prove analytic results in a rigorous way. 
On the other hand 
mathematical literature can be a source of inspiration to improve these models,
 since theoretical results may shed light on 
aspects that are difficult to be captured with only a numerical approach, and help to guess 
emergent behaviors that are unexpected in a computational perspective.  
 For example, many mathematical results in the context of urn models have been 
used to design and study agent-based models both analytically and computationally.

 \indent From a mathematical
point of view, there exists a growing interest in systems of
{\em interacting reinforced} stochastic processes of different kinds
({\it e.g.}~\cite{ale-cri-ghi, ale-cri-ghi-MEAN,
  ale-cri-ghi-WEIGHT-MEAN, ale-ghi, ben, che-luc, cir,
  cri-dai-lou-min, cri-dai-min, dai-lou-min, fortini, ieee-paper,
  lima, pag-sec, sah, LouisMirebrahimi,LouisMinelli}). Our work is
placed in the stream of this scientific literature. Generally
speaking, by reinforcement in a stochastic dynamics we mean any
mechanism for which the probability that a given event occurs increases with the number of times the same event
occurred in the past. This {\em ``reinforcement mechanism''}, also
known as  {\em ``Rich get richer rule''} or {\em ``Matthew effect''}, is a key feature
governing the dynamics of many biological, economic and social systems
(see, {\it e.g.}~\cite{pem}).  The best known example of reinforced
stochastic process is the standard Eggenberger-P\'{o}lya urn
(see~\cite{egg-pol, mah,TRACY1992278}), which has been widely studied and
generalized (some recent variants can be found in
\cite{ale-cri-rescaled, ale-ghi-ros, ale-ghi-vid,
  ber-cri-pra-rig-barriere, chen-kuba, collevecchio, cri-ipergeom,
  ghi-vid-ros, laru-page}).  \\ \indent Precisely, in this work we
consider a system of $N\geq 2$ interacting stochastic processes
$\{I^h=(I_{n,h})_{n\geq 1}:\, 1 \leq h\leq N\}$ such that each one of them
takes values in $\{0,1\}$ and their evolution is modelled as
follows: for any $n\geq 0$, the random variables $\{I_{n+1,h}:\,1\leq
h\leq N\}$ 
 are conditionally independent given the past information
${\mathcal F}_{n}$ with
\begin{equation}\label{prob-model-intro}
P_{n,h}=P(I_{n+1,h}=1\,|\F_n)=
\alpha Z_n + \beta q + (1-\alpha-\beta) f(Z_{n,h}), 
\end{equation}
where $\alpha,\,\beta\in [0,1)$, $\alpha+\beta\in (0,1)$, $q\in (0,1]$, 
\begin{equation}\label{dynamics-model-intro}
Z_{n+1,h}=
\left(1-r_n \right)Z_{n,h}+r_n I_{n+1,h}
\qquad \mbox{with } r_n\sim \frac{1}{n}\,,
\end{equation}
$Z_n=\sum_{i=1}^N Z_{n,i}/N$,
 the function~$f$ is a strictly increasing $[0,1]$ valued function
belonging to $\mathcal{C}^1([0,1])$  and
  $r_n\sim 1/n$ 
  means $\lim_{n \to \infty} nr_n=1$.
The
starting point for the dynamics~\eqref{dynamics-model-intro} is a
random variable $Z_{0,h}$ with values in $[0,1]$ and  the past
information ${\mathcal F}_n$ formally corresponds to the
$\sigma$-field $\sigma(Z_{0,h}:\, 1\leq h\leq N)\vee \sigma(I_{k,h}:\,
1\leq k\leq n,\,1\leq h\leq N)= \sigma(Z_{k,h}:\, 0\leq k\leq
n,\,1\leq h\leq N)$.  \\

The system represents a population of $N$ interacting units 
 whose state at time $n$ is synthesized by the variable $Z_{n, h}$ and whose individual evolution has a reinforcement mechanism driven by a function $f$. 
 Depending on the choice of the function $f$, we can give different meanings to such individual evolution. \\

\indent As a first possible interpretation, let us assume that we are modeling a system of $N$ agents,
who at each time-step $n$ have to have to choose an action
  $s\in
\{0,1\}$. Suppose that $s=1$ represents the ``right'' choice, that is
the one that gives the greater pay-off, and $0$ represents the
``wrong'' one. The processes $I^h$ describes the sequence of actions 
along the time-steps, that is $I_{n+1,h}$ is the indicator function of
the event ``agent $h$ makes the right choice at time $n$''. 
The
processes $\{Z^h=(Z_{n,h})_{n\geq 0}:\, 1 \leq h\leq N\}$ can be interpreted as the ``personal inclination'' of
the agent~$h$ in adopting the right choice along time, that is
$Z_{n,h}$ is the inclination at time $n$ of agent~$h$ toward the right 
choice. 
Therefore the above model includes three issues:
\begin{itemize}
\item Conditional independence of the agents given the past: Given the
  past information until time $n$, the agent $h$ makes a choice at
  time $n+1$ independently of the other agents' choices at time $n+1$.
\item At each time $n+1$, the probability $P_{n,h}$ that agent $h$
  makes the right choice is a convex combination of the average value
  $Z_n$ of all the current agents' inclinations, an external ``forcing
  input'' $q$ and a function of her own current inclination
  $Z_{n,h}$. In the sequel, we will refer to this last factor as the
  ``personal inclination component'' of $P_{n,h}$. The term $Z_n$
  provides a {\em mean-field interaction} among the agents. Note that,
  when $\alpha=0$, we have not a proper interaction: indeed, the
  agents are subject to the same forcing input $q$, but they evolve
  idependently of each other.  We exclude the case $\alpha=\beta=0$
  because it corresponds to $N$ independent agents who evolve only according to
  the personal inclination component.  We also exclude the case
  $\alpha+\beta=1$ because it means that there is not the personal
  inclination component.
\item Since $f$ is strictly increasing, there is a reinforcement 
  mechanism on the personal inclination component: if $I_{n,h}=1$,
  then $Z_{n,h}>Z_{n-1,h}$ (provided $Z_{n-1,h}<1$) and so
  $f(Z_{n,h})>f(Z_{n-1,h})$. In other words, the fact that agent $h$
  makes the right choice at time $n$ implies a positive increment of
  her/his inclination toward the adoption of the right choice in the
  future. As a consequence, the larger the number of times in which an
  agent has made the right choice until time $n$, the higher her/his
  personal inclination component in the probability of a future right
  choice at time $n+1$. The justification of this mechanism is
  twofold: first, higher pay-offs can be related to better
  physiological conditions and so individuals that are better fed and
  healthier are less likely to make mistakes in the future choices;
  second, if the choice is always related to the same action, agents
  that earn higher pay-offs are not inclined to change their action (see~\cite{bilancini} and references therein). 

\item The forcing input $q$ models the presence of an
  external force ({\it e.g.} a political constraint, or an advertising
  campaign) that leads agents toward the right choice with
  probability~$q$.
\end{itemize} 
 
The considered model also fits well 
in a different context, where there is not a "right" choice, but agents have to choose between two brands 
$s\in\{0,1\}$, that are related to a loyalty program: the more they
select the same brand, the more loyalty points they gain. This fact
motivates the reinforcement mechanism on the personal inclination
component and, similarly as before, the forcing input can be
interpreted as the possible presence of an external force that leads
agents toward the brand $1$ with probability $q$.  
\medskip 

Other interpretations can be given in the context of coordination games or opinion dynamics and they will be described in more detail in sections 
\ref{second-inter} and~\ref{third-inter} below, where we will focus on specific choices of the function $f$.  
\medskip

The main object of our study is to check if the system has long-run
equilibrium configurations as $n\to +\infty$, that is, if, for $h=1,\ldots, N$, the
stochastic process $Z^h=(Z_{n,h})_n$ converge almost surely, as $n$ tends to $+\infty$, to some
random variable $Z_{\infty,h}$. Second, we want to analyse the limit 
configurations  $[Z_{\infty,1},\dots,Z_{\infty, N}]$, characterizing the support of their probability distribution. In
particular, we are interested in the phenomenon of \emph{synchronization} of the stochastic processes $Z^h$, which occurs when all the stochastic processes~$Z^h$  converge almost surely toward \emph{the same} random variable. Regarding 
this question, we point out that the above model with~$f$ equal to the
identity function is essentially included in the models considered in
\cite{ale-cri-ghi, cri-dai-lou-min} and, in this case, the almost sure
asymptotic synchronization always take place (precisely, almost sure
synchronization toward a random variable when $\alpha>0$ and $\beta=0$
and toward the constant $q$ when $\beta>0$).
\medskip 

Synchronization phenomena are ubiquitous in Nature and have been observed in a wide variety of models based on randomly interacting units  (see the literature cited above and the references therein). 
Synchronization comes as a result of the interaction
and can be enhanced if a reinforcement mechanism is present in the dynamics: for example, in~\cite{cri-dai-lou-min} it has been shown that, if reinforcement is sufficiently 
strong, agents coordinate each other and synchronize in a time scale smaller than the one needed to reach their common (random) limit, 
giving rise to \emph{synchronized fluctuations}.

Note that the emergence of collective self-organized behaviors    
in social communities has been frequently described in models based on a Statistical Physics approach (see, \textit{e.g.}, \cite{collet2015collective,collet2016rhythmic,aleandri2019opinion,
AleMinelli2020,challet2000statistical}) as a result of a large scale limit. 
However, we emphasize that synchronization is not a large scale phenomenon in the models studied in this paper. Indeed, for suitable values of the parameters,  it occurs \emph{for any} value of~$N$. 
In particular, we will prove that for the models under consideration a \emph{phase transition} occurs, depending on the parameter~$\alpha$ that tunes the strength of interaction. When~$\alpha$ is close enough to~$1$, synchronization occurs for any $N$ (even in absence of the external input). While, if $\alpha$ is below  a threshold,  ''fragmentation''  appears in the population and several limit configurations, where agents are divided into two separated groups with two different  inclinations, are possible.  
In this last scenario, the strength of interaction, even if too weak to produce synchronization, still continues to affect the dynamics, through the number of possible limit configurations and the localization of the limit values for the inclinations. 
\medskip 

Remark
that some of the systems of interacting reinforced processes previously studied~\cite{dai-lou-min,cri-dai-lou-min,cri-dai-min,sah,LouisMirebrahimi,ale-cri-ghi, ale-ghi} can be obtained from the  model introduced in this paper taking $f$ equal to the identity function and substituting~$Z_n$ 
with a weighted average of the agents' inclinations. For such models,   
 cases of no-synchronization may occur only in absence of interaction, that is, when agents are divided into 
two or more groups and at least two of such groups behave independently, \textit{i.e.},  when $\alpha=0$ or when the matrix describing the strengths of interaction between the various agents is not irreducible.  
 Instead, in the present
work, cases of no-synchronization may occur also when $\alpha>0$, \textit{i.e.}, when all the agents interact with each other.  As we will see, the synchronization or no-synchronization of the system is related to the
properties of the function~$f$. In particular, in order to have a
strictly positive probability of no-synchronization, a necessary
condition is~$f$ {\em not linear}.
\medskip 

\indent Finally, it is worthwhile to note that the asymptotic
behaviour of the stochastic process $Z^h$ is strictly related to the
one of the stochastic process
$\{\bar{I}_n^h=\sum_{k=1}^nI_{k,h}/n\}$ (see also
\cite{ale-cri-ghi-WEIGHT-MEAN, ale-cri-ghi-MEAN}), that is, according
  to the previous interpretation, the average of times in which agent
  $h$ adopts the right choice. Therefore, the synchronization or
  no-synchronization of the inclinations of the agents corresponds to
  the synchronization or no-synchronization of the average of times in
  which the agents make the right choice.

\subsection{Interacting systems of coordination games}
\label{second-inter}
In this section we illustrate a possible interpretation of our model in the context of Game Theory. Following the approach of~\cite{Fagiolo}, each interacting unit~$h$ represents a time evolving ''economy'' \textit{i.e.}  a community of agents that grows in time and play a cooperative game. 
The whole system describes a population of~$N$ communities subject to a \emph{mean-field} interaction and to the influence of an external input.  
The individual evolution of a given community is defined as follows.
At time~$n = 0$, there are $N_0 > 0$ agents in the community. Each agent
is fully described by a binary pure strategy $s \in S = \{-1,
+1\}$. Thus, at any $n$, the state of a given community can be characterized
by the current share $X_n \in [0, 1]$ of agents playing strategy $+1$.
The system evolves as follows. Given some initial share~$X_O$, at any
$n > 0$ a new agent enters the community, observes current strategy share
and irreversibly chooses a strategy on the basis of expected
pay-offs. More precisely, call $\pi_n(s)$ the expected pay-off
associated to strategy $s\in S$ at time $n$ and set
$\pi_n=\{\pi_n(s):\, s\in S\}$. We assume that the probability, say
$P_n$, that the agent $n$ chooses $s=+1$ is a function of $\pi_n$.
Moreover, we assume that the expected pay-offs $\pi_n(s)$ are related
to a symmetric $2\times 2$ coordination game, that is we assume that
the agent entering at time $n$ plays a symmetric $2\times 2$
coordination game against all the present agents according to a
standard stage-game pay-off matrix as in Table~\ref{tab:pay-offs}.

\begin{table}[h]
\centering
\caption{pay-offs Matrix. Left: Original pay-offs. Right: Standardized pay-offs.}
\label{tab:pay-offs}
\begin{tabular}{lll}
   & +1 & -1\\
\hline
+1 & $a_{+1,+1}$ & $a_{+1,-1}$ \\
-1 & $a_{-1,+1}$ & $a_{-1,-1}$ \\
\hline
\end{tabular}
\qquad\qquad
\begin{tabular}{lll}
   & +1 & -1\\
\hline
+1 & 1 & 0 \\
-1 & A & B \\
\hline
\end{tabular}
\end{table}

We assume $a_{+1,+1} > a_{-1,+1}$ and $a_{+1,-1} < a_{-1,-1}$ because
the game is a coordination one. We also assume $a_{+1,+1} \geq
a_{-1,-1}$ and $a_{+1,-1}\leq a_{-1,+1}$.  In what follows, we shall
focus on the standardized version of the pay-off matrix, obtained
from the former (without loosing in generality) by letting $A =
(a_{-1,+1} - a_{+1,-1})/(a_{+1,+1} - a_{+1,-1})\in [0,1)$ and $B =
  (a_{-1,-1} - a_{+1,-1})/(a_{+1,+1} - a_{+1,-1})\in (0,1]$.  \\
\indent Expected pay-offs for the agent entering at time $n+1$
associated to any given choice $s\in S$ are given by
\begin{equation*}
 \pi_{n}(s)=
 \begin{cases}
 X_n\qquad &\mbox{if } s=+1\\
 A X_n+ B(1-X_n)\qquad &\mbox{if } s=-1.
 \end{cases}
\end{equation*}
Therefore, since $P_n$ is a function of $\pi_n$, we get that $P_n$ is
a function of $X_n$, \textit{i.e.} $P_n=f(X_n)$.  The dynamics of $\mathcal{X}=(X_n)_n$
is easily given by
\begin{equation}\label{dynamics-fag}
X_{n+1}=\left(1-\frac{1}{N_0+n+1}\right)X_n+\frac{1}{N_0+n+1}I_{n+1}, 
\end{equation}
where $I_{n+1}$ is the indicator function of the event ``agent
entering the community at time $n+1$ chooses strategy $+1$'' and so
$P(I_{n+1}=1\,|\,\mathcal{Z}_k,\, k\leq n)=P_n=f(Z_n)$. Different individual
decision rules give different functions $f$. Two examples are the
following:
\begin{equation*}
 \mbox{Linear Probability (LP):}\quad 
 P_n=\frac{\pi_n(+1)}{\pi_n(+1)+\pi_n(-1)}
\end{equation*}
which gives 
\begin{equation}\label{f-LP-game}
f(x)=\begin{cases}
x\qquad&\mbox{if } A=0\;\mbox{and } B=1
\\
\frac{x}{\theta(x+x^*)}\qquad&\mbox{if } \theta=(1+A-B), 
\end{cases}
\end{equation}
with $x^*=B/\theta=B/(1+A-B)$ and so $\theta x^*\in (0,1]$ and $\theta
  x^* \geq 1-\theta$.

\begin{equation*}
 \mbox{Logit Probability (LogP):}\quad 
 P_n=\frac{\exp(K\pi_n(+1))}{\exp(K\pi_n(+1))+\exp(K\pi_n(-1))},
\end{equation*}
with $K>0$, which gives 
\begin{equation}\label{f-LogP-game}
f(x)=\frac{1}{1+\exp(-\theta(x-x^*))},
\end{equation}
with $\theta=K(1-A+B)>0$ and $x^*=KB/\theta=B/(1-A+B)\in (0,1)$.
\\

Under the above individual decision rules, long-run equilibria 
for one community have been studied in~\cite{Fagiolo}: 
\begin{itemize}
\item With LP rule, and if the game is not a pure-coordination one
  (that is $A=0$ and $B=1$), the long-run behavior of the system
  becomes predictable (see definition in     Section~\ref{Sec-general-results} below): the share of agents playing
  $+1$ in the limit converges a.s. to the constant
  $z_\infty=(1-B)/(1+A-B)$. Note that, when $B=1$, we have
  $z_\infty=0$ and, when $A=0$, we have $z_\infty=1$. In all the other
  cases (that is $B<1$ or $A>0$), coexistence of strategies
  characterizes equilibrium configuration and we have $z_\infty>1/2$,
  or $=1/2$, or $<1/2$ if and only if $A+B<1$, or $A+B=1$, or $A+B>1$,
  respectively.\\ With LP rule, if the game is a pure-coordination
  one, then $X_n$ follows the dynamics of the standard P\'olya urn
  model and so it converges a.s. to a random variable $\mathcal{Z}_\infty$ with
  beta distribution.
\item With LogP, it has been proven that the long-run behavior of the
  community with $x^*=1/2$ is predictable if $KB=\theta x^*=\theta/2\leq
  2$: the share of agents playing $+1$ in the limit converges a.s. to
  $1/2$, that means coexistence of the two strategies in the
  proportion $1:1$. Moreover, When $x^*\neq 1/2$, some numerical
  analysis have been performed pointing out the coexistence of
  strategies in the limit configuration and the fact that the dynamics 
  is again predictable when $KB=\theta x^*$ is small.
\end{itemize}

 We are interested in analyzing the long-run behavior of a system of
 $N\geq 2$ interacting games of the above type. More precisely,
 for each $h \in \{1,\dots, N\}$, let $Z_{n,h}$ be the share of agents playing
 strategy $+1$ in the community $h$. We assume that the dynamics for
 each $Z_{n,h}$ is of the form 
\begin{equation}\label{inter-dynamics-fag}
Z_{n+1,h}=
\left(1-\frac{1}{N_{0,h}+n+1}\right)Z_{n,h}+\frac{1}{N_{0,h}+n+1}I_{n+1,h}, 
\end{equation}
where $I_{n+1,h}$ is the indicator function of the event ``agent
entering community $h$ at time $n+1$ chooses strategy $+1$'', and we
assume that $\{I_{n+1,h}:\, h=1,\dots, N\}$ are conditionally
independent given the past information $\mathcal{F}_n$ with
\begin{equation}\label{inter-dynamics-fag-prob}
P_{n,h}=P(I_{n+1,h}=1\,|\F_n)=
\alpha Z_n + \beta q + (1-\alpha-\beta) f(Z_{n,h}), 
\end{equation}
where $Z_n=\sum_{i=1}^N Z_{n,i}/N$, $q\in (0,1]$ and
  $\alpha,\,\beta\in [0,1)$ and $\alpha+\beta\in (0,1)$.
\\

\indent This corresponds to assume that the agent entering a given
community at the future time~$n+1$ will choose (independently of the
choices of the agents entering the other communities at time~$n+1$) the
strategy $+1$ with a probability~$P_{n,h}$, which is a convex
combination of three factors: the present share~$Z_n$ of players
playing~$+1$ in the entire system, an external forcing
input~$q$ and the expected pay-offs related to the present 
share~$Z_{n,h}$ of players playing~$+1$ in the specific game where the agent enters.  
\\

\indent It is worthwhile to notice that, although our focus is on the
case $N\geq 2$, we are also going to completely describe the
asymptotic behaviour of the system when $N=1$. We point out that, in
\cite{Fagiolo}, the case $x^*\neq 1/2$ is studied only by means of
simulations, while here we provide analytic  results.

\subsection{Technological and Opinion dynamics}\label{third-inter}

By technological dynamics we mean models which describe the diffusion 
of some technological assets in a given community. Such diffusion may depend on several factors, such as communication between agents, the influence of an external media and a form of self-reinforcement due to agents' loyalty. 
On the other hand, opinion dynamics deals with the study of 
formation and evolution of opinions in a population, which is governed by similar factors;  in particular, self-reinforcement can be interpreted in this context  as a mechanism for which the agents' personal inclination, after being verbalized through the choice of one out of two (or more) possible actions,  is subject to reinforcement in the direction of the expressed choice. 
Therefore in what follows we will refer to the first context with the implicit assumption that everything can be translated in the language of the second. 
In the above setting, an interacting unit~$h$ of our model may be interpreted either as a single agent, to whom is associated an inclination or opinion~$Z_{n, h}$ to adopt one of two different assets (or actions), or as a whole community of agents which  has an internal evolution, driven by the function $f$, and interacts with other similar communities, eventually under the influence of an external media.
Below, in order to motivate specific choices of the function~$f$, we describe in details a model based on this last interpretation, where each unit~$h$ is  
modelled as a generalized  P\'olya urn .
In the context of opinion dynamics, our models belong to the class of recently studied CODA models (Continuous Opinions, Discrete Actions) \cite{martins2008continuous,martins2013trust}. 
\\

The generalized P\'olya urn model~\cite{arthur-1983, Arthur-urn} has been
used in order to model the competitive process among new technologies,
which is a fundamental phenomenon in Economics~\cite{dosi}.  This urn
model is included in the class of reinforced stochastic processes
defined by~\eqref{prob-model-intro} and~\eqref{dynamics-model-intro}
when the urn function belongs to $\mathcal{C}^1$ and it is strictly
increasing. Taking $f$ strictly increasing means that the considered
technologies show increasing returns to adoption: the more they are
adopted, the more is learned about them and, consequently the more
they are improved, and the more attractive they become
\cite{Arthur-urn}. The dynamics for a single ``market'' of potential
adopters is as follows: at each time-step $n$ an agent enters the
system and decides to adopt one of two possible technologies $s\in
\{0,1\}$ according to the dynamics~\eqref{dynamics-fag} with a given
urn function $f$. The present work is related to the study of the
long-run behavior of a system of $N\geq 2$ interacting markets of
potential adopters of this kind and so described by
\eqref{inter-dynamics-fag} and~\eqref{inter-dynamics-fag-prob}.  \\

An example of function $f$ used in this framework is 
\begin{equation}\label{f-Tech}
f(x)=(1-\theta)+(2\theta-1)(3x^2-2x^3)\qquad\mbox{with } \theta\in [0,1],
\end{equation}
which belongs to $\mathcal{C}^1$ and is strictly increasing when
$\theta\in (1/2, 1]$.  The applicative justification behind this
  function is as follows. See~\cite{arthur-1983, dosi}. Suppose we have
  two competing technologies, say $s\in \{0,1\}$, and represent the
  population of adopters who are already using one of the two
  technologies as an urn containing balls of two different colors, say
  red for technology~$1$ and black for the other. The composition of
  the urn evolves along time according to the following agents'
  decision making rule: at each time-step, the agent extracts with
  replacement a random sample of $r=3$ balls from the urn (this means
  that the agent asks to $3$ previous agents which technology they are
  using) and then the agent selects with probability $\theta$ the
  technology used by the majority of the extracted sample (and an
  additional ball of the corresponding color is put into the urn) and
  with probability $(1-\theta)$ the technology used by the minority of
  them (and an additional ball of the corresponding color is put into
  the urn). Notice that, rephrasing the above description in the language of opinion dynamics, we get a variant of the celebrated   
Galam’s majority-rule model~\cite{galam1986majority}, with the introduction of a reinforcement mechanism in the dynamics.\\
  According to this dynamics, denoting by $T$ and by $T_1$
  the total number of balls and the total number of red balls,
  respectively, into the urn at time-step $n$, we have
\begin{equation*}
\begin{split}
P_n&=P(I_{n+1}=1|\mathcal{F}_n)=
\theta p(T,T_1)+(1-\theta)\left(1-p(T,T_1)\right)
=(1-\theta)+(2\theta-1)p(T,T_1)
\qquad\mbox{with}\\
p(T,T_1)&=
\sum_{k=2}^3
\frac{ {{T_1}\choose{k}} {{T-T_1}\choose{3-k}} }{ {{T}\choose{3}} }
\sim 
\sum_{k=2}^3 
{{3}\choose{k}} \left(\frac{T_1}{T}\right)^k\left(1-\frac{T_1}{T}\right)^{3-k}
\quad\mbox{for } n\to +\infty.
\end{split}
\end{equation*}
(The above approximation follows from the property of the Gamma
function: $\Gamma(n+1)=n!$ and $\Gamma(n+a)\sim n^a\Gamma(n)$ for
$n\to +\infty$.)  In other terms, setting $X_n=T_1/T$, that is 
the proportion of red balls into the urn at time-step $n$, we have 
$$
P_n\sim (1-\theta)+(2\theta-1)\sum_{k=2}^3 {{3}\choose{k}} X_n^k(1-X_n)^{3-k}=
f(X_n)\quad\mbox{for } n\to+\infty,
$$ where $f$ is the function given in~\eqref{f-Tech}.  \\ 

\indent For
a single market, the authors~\cite{dosi}
find a threshold~$1/2$ below which the 
limit market is shared by the two
technologies in the proportion $1:1$
that is, if
$\theta\leq 1/2$, $Z_n$ converges almost
surely to~$1/2$, otherwise, two limit market configurations are
possible. Although our focus is on the case~$N\geq 2$, we are also
going to completely describe the asymptotic behaviour of the system
when~$N=1$. In particular, we will correct the above mentioned threshold. Indeed, we will
prove that, for $N=1$, when $1/2 < \theta\leq 5/6$, the system has
$1/2$ as the unique limit configuration, while, when $5/6<\theta<1$,
two limit configurations are possible and in both of them the two
technologies coexist, with  the proportion $z:(1-z)$ or $(1-z):z$
with $z\in (0, 1/2)$. Therefore the threshold is, not at
$1/2$, but at $5/6$.  \\

\indent The rest of the paper is organized as follows. In Section
\ref{Sec-general-results} we provide some general results regarding
the asymptotic behavior of the considered systems. More precisely, we
give sufficient conditions for the almost sure convergence of the
stochastic processes $Z^h=(Z_{n,h})$ to some random variable
$Z_{\infty,h}$ and for the almost sure asymptotic synchronization of
the system. Moreover, we give some results concerning the possible
values that the limit random vector
$[Z_{\infty,1},\dots,Z_{\infty,N}]$ can take.  In Section
\ref{Sec-examples} we analyze the systems associated to the functions
introduced in Subsec.~\ref{second-inter} 
and~\ref{third-inter}. Specifically, we show sufficient conditions on the
parameters in order to have the almost sure asymptotic synchronization
of the system. Moreover, in the case when the interaction parameter
and the forcing input are not so strong in order to assure the almost
sure asymptotic synchronization, we characterize the possible limit
configurations of the system. Furthermore, we point out when the
system is predictable (that is, when there exists a unique possible
limit configuration) and when the system may almost surely
asymptotically synchronize toward the value $1/2$, meaning that in the
limit, according to the different applicative frameworks, the two
inclinations, or the two strategies in all the games, or the two
technologies in all the markets, coexist in the proportion $1:1$.  We
discuss also if we may have $P(Z_{\infty,h}=0)>0$ or
$P(Z_{\infty,h}=1)>0$, that is if we may have the case when one of the
two inclinations (or strategies, or technologies) is asymptotically
predominant with respect to the other and it only survives in the
limit. Finally, as said before, we also take into account the case
$N=1$. The paper is also enriched by simulations and figures, all
collected in Section~\ref{Sec-simulations}, and by an appendix,
containing some recalls about Stochastic Approximation theory and some
technical linear algebra results.


\section{General results}\label{Sec-general-results}
By means of~\eqref{prob-model-intro} and~\eqref{dynamics-model-intro},
the recursive equation for $Z_{n,h}$ can be rewritten as
\begin{equation}\label{eqZnh}
Z_{n+1,h}=Z_{n,h}+r_n\left[\alpha Z_n+ \beta q+ (1-\alpha-\beta) 
f(Z_{n,h})-Z_{n,h}\right] + r_n\Delta M_{n+1,h}\,,  
\end{equation} 
where $\Delta M_{n+1,h}=I_{n+1,h}-P_{n,h}$ is a martingale difference
with respect to $\F=(\F_n)_n$. Moreover, summing over $h$, we get the
equation for $Z_n$
\begin{eqnarray}\label{eqZn-new}
Z_{n+1}=
Z_n+r_n\left[\alpha Z_n+ \beta q+ 
(1-\alpha-\beta)\frac{1}{N}\sum_{h=1}^N f(Z_{n,h})-Z_n\right] + 
r_n\left(\frac{1}{N}\sum_{h=1}^N\Delta M_{n+1,h}\right).
\end{eqnarray}  
Let us set $\mathbf{Z}_n=(Z_{n,1},\ldots, Z_{n,N})^{\top}$,  
$\Delta \mathbf{M}_{n+1}=(\Delta M_{n+1,1},\ldots, \Delta M_{n+1,N})^{\top}$
and 
\begin{equation}\label{eqG}
\mathbf{F}(\mathbf{z})=
(F_1(\mathbf{z}),\ldots, F_N(\mathbf{z}))^{\top}
\;\mbox{ with } 
F_h(\mathbf{z})=
\alpha \frac{1}{N}\sum_{i=1}^N z_i+\beta q+ (1-\alpha-\beta)f(z_h)-z_h
\;\forall\mathbf{z}\in [0,1]^N\,.
\end{equation}
Using the above notation, we can write~\eqref{eqZnh} in the vectorial form
\begin{equation}\label{eq-vector-Z}
\mathbf{Z}_{n+1}=\mathbf{Z}_n+r_n\mathbf{F}(\mathbf{Z}_n)+
r_n\Delta \mathbf{M}_{n+1}\,.
\end{equation} 

We are interested in proving that 
\begin{equation}\label{limit}
\mathbf{Z}_n\stackrel{a.s.}\longrightarrow \mathbf{Z}_\infty\,,
\end{equation}
where $\mathbf{Z}_\infty$ is a suitable random variable with values in
$[0,1]$, and in the characterization of the support of its
distribution.\\ \indent Throughout this paper we use the symbols
$\mathbf{0}$ and $\mathbf{1}$ to denote the vector with all the
components equal to zero and the vector with all the components equal
to one. When $P(\mathbf{Z}_\infty=\mathbf{z}_\infty)>0$ with
$\mathbf{z}_\infty$ of the form $\mathbf{z}_\infty=z_\infty
\mathbf{1}$, with $z_\infty\in [0,1]$, we call $\mathbf{z}_\infty$ a
       {\em synchronization point} for the system. If all the possible
       values for $\mathbf{Z}_\infty$ are synchronization points, that
       is $\mathbf{Z}_\infty$ is of the form $Z_\infty\mathbf{1}$ with
       $Z_\infty$ a suitable random variable taking values in $[0,1]$,
       we say that the system {\em almost surely asymptotically
         synchronize}. Moreover, we say that the system is {\em
         predictable} when there exists a point $\mathbf{z}_\infty$
       such that $\mathbf{Z}_\infty= \mathbf{z}_\infty$ almost surely.
       \\ \indent Finally, it is worthwhile to note that the almost
       sure convergence of $\mathbf{Z}_n$ toward a 
       random variable~$\mathbf{Z}_\infty$ implies the almost sure convergence of the
       empirical means
       $\overline{\mathbf{I}}_n=\frac{1}{n}\sum_{k=1}^n\mathbf{I}_k$
       (where $\mathbf{I}_k$ is the random vector with components
       $I_{k,h}$, for $h=1,\dots, N$) toward the same limit.  \\

\indent Let $\mathbf{F}$ and $(\mathbf{Z}_n)_{n\geq 0}$ be defined as
in~\eqref{eqG} and~\eqref{eq-vector-Z} and let
$\mathcal{Z}(\mathbf{F})=\{\mathbf{z}\in [0,1]^N:
\mathbf{F}(\mathbf{z})=\mathbf{0}\}$ be the zero-set of the function
$\mathbf{F}$. Using the Stochastic Approximation methodology (see
Appendix~\ref{app-sto-approx}), we obtain the following results.  The
first one concerns the almost sure convergence of the 
process~$(\mathbf{Z}_n)$.

\begin{theorem}\label{general-as-conv} (Almost sure convergence)\\
 The set $\mathcal{Z}(\mathbf{F})$ contains at least one
 synchronization point and, if $\mathcal{Z}(\mathbf{F})$ is finite and
 $f$ admits a primitive function, we have
$$
\mathbf{Z}_n\stackrel{a.s.}\longrightarrow \mathbf{Z}_{\infty},
$$
where $\mathbf{Z}_\infty$ is a suitable random variable with values in 
 $\mathcal{Z}(\mathbf{F})$. Moreover, we also have
$$
\overline{\mathbf{I}}_n=\frac{1}{n}\sum_{k=1}^n\mathbf{I}_k
\stackrel{a.s.}\longrightarrow \mathbf{Z}_{\infty}.
$$ 
\end{theorem}

\begin{proof}
We firstly show that $\mathcal{Z}(\mathbf{F})$ is non empty, since it
contains at least one synchronization point.  Indeed, points in
$\mathcal{Z}(\mathbf{F})$ are the solutions in $[0,1]^N$ of the system
of equalities
\begin{equation}\label{system}
\alpha \frac{1}{N}\sum_{i=1}^N z_i+\beta q+
(1-\alpha-\beta)f(z_h)-z_h=0\qquad\forall h=1,\dots,N\,.
\end{equation}
In particular, for the synchronization zero points, that is for the
zero points of the form $\mathbf{z}=z\mathbf{1}$, the above system of
equalities reduces to the equation
\begin{equation}\label{system-sincro}
\varphi(z)=f(z)-\frac{(1-\alpha)}{(1-\alpha-\beta)} z+ 
\frac{\beta q}{(1-\alpha-\beta)}=0\,.
\end{equation}
See Fig.~\ref{fig:f3:graph} and Fig.~\ref{fig:f4:graph} for examples and illustrations.
Therefore, since $f$ takes values in $[0,1]$, we have $\varphi(0)\geq
0$ and $\varphi(1)\leq 0$. This fact implies that $\varphi$ always has
at least one zero point in $[0,1]$.  \\ Hence, under the above
assumptions, the almost sure convergence of $(\mathbf{Z}_n)$
immediately follows from Theorem~\ref{cond-suff-lyapunov}, because we
have
\begin{equation*}
\begin{split}
\mathbf{F}&=-\nabla\mathbf{V}\qquad\mbox{with}\\ 
\mathbf{V}(\mathbf{z})&=
- \frac{\alpha}{2N}\left(\sum_{h=1}^{N}z_h\right)^2 
- \beta q\sum_{h=1}^N z_h 
- (1-\alpha-\beta)\sum_{h=1}^N \phi(z_h)
+ \frac{1}{2} \sum_{h=1}^N z_h^2,
\end{split}
\end{equation*}
where $\phi$ is a primitive function of $f$.  \\ \indent Finally,
since $E[\mathbf{I}_{n}|{\mathcal
    F}_{n-1}]=\mathbf{Z}_n\stackrel{a.s.}\longrightarrow
\mathbf{Z}_\infty$, applying Lemma~B.1 in~\cite{ale-cri-ghi-MEAN}
(with $c_k=k$, $v_{n,k}=k/n$ and $\eta=1$), we get that
$\frac{1}{n}\sum_{k=1}^n\mathbf{I}_k \stackrel{a.s.}\longrightarrow
\mathbf{Z}_\infty$.
\end{proof}

The following theorem provides a sufficient condition for the
almost sure synchronization of the system.

\begin{theorem}\label{general-synchro} (Almost sure asymptotic synchronization)
\\ If $\mathcal{Z}(\mathbf{F})$ contains a finite number of
synchronization points, $f$ admits a primitive function and, for each
fixed constant $c\in
\left(-\frac{\alpha+\beta}{1-\alpha-\beta},0\right)$, the function
\begin{equation}\label{cond-phi-tilde}
\widetilde{\varphi}(z)=f(z)-\frac{1}{1-\alpha-\beta}z-c 
\end{equation}
has at most one zero point in $[0,1]$, then we have the almost sure
asymptotic synchronization of the system and the limit random variable
$\mathbf{Z}_\infty$ is of the form $Z_\infty\mathbf{1}$, where
$Z_\infty$ satisfies Equation~\eqref{system-sincro}.
\end{theorem}

\begin{remark} ({\em Linear case})\\ \rm 
Note that, since $c$ belongs to
$\left(-\frac{\alpha+\beta}{1-\alpha-\beta},0\right)$, we have
$\widetilde{\varphi}(0)>0$ and $\widetilde{\varphi}(1)<0$ and so
 the
equation~$\widetilde{\varphi}=0$ always has a solution. The above
result requires that this solution is unique. A particular case in
which this condition is satisfied is when $f$ is {\em linear}. Indeed,
if $f:[0,1]\to [0,1]$ is linear and strictly increasing, then
$f'=\delta\in (0,1]$ and hence $\delta\neq 1/(1-\alpha-\beta)$. It is
  worthwhile to observe that, when $f$ is linear, Equation
 ~\eqref{system-sincro} has infinite solutions (and so Theorem
 ~\ref{general-synchro} does not apply) only when $f$ is the identity
  function and $\beta=0$. However, this case is included in
 ~\cite{ale-cri-ghi, cri-dai-lou-min}, where the almost sure
  asymptotic synchronization is proven also in this case.
\end{remark}

\begin{proof} We first prove that the assumptions of Theorem
\ref{general-synchro} imply that $\mathcal{Z}(\mathbf{F})$ does not
contain ``no-synchronization'' points, that is points that are not
synchronization points.  To this purpose, we recall that the set
$\mathcal{Z}(\mathbf{F})$ is described by the system of equalities
\eqref{system}. In particular, if $\mathbf{z}^*$ is a solution of the
system~\eqref{system} with $z_h^*\neq z_j^*$ for at least a pair of
indexes, equation~\eqref{system} implies
\begin{equation}\label{diseg-1}
(1-\alpha-\beta)f(z_h^*)<z_h^*\qquad\forall h = 1,\dots, N
\end{equation}
and
\begin{equation}\label{diseg-2}
(1-\alpha-\beta)f(z_h^*)>z_h^*-\alpha-\beta\qquad\forall h = 1,\dots, N.
\end{equation}
Moreover,~\eqref{system} (written for $h$ and $j$) also implies 
\begin{equation}\label{no-sincro-diff}
(z_h^*-z_j^*)= (1-\alpha-\beta)(f(z_h^*)-f(z_j^*))\qquad\forall h,j=1,\dots, N\,.
\end{equation}
Therefore, fixed $h$, $z_j^*$ is a solution (different from $z_h^*$) of
the equation
$$
f(z)=\frac{1}{1-\alpha-\beta}z+c 
$$ where $c=f(z_h^*)-\frac{z_h^*}{1-\alpha-\beta}\in
\left(-\frac{\alpha+\beta}{1-\alpha-\beta},0\right)$ (by
\eqref{diseg-1} and~\eqref{diseg-2}).  In other terms, a necessary
condition for the existence of no-synchronization zero points is that
there exists $c\in
Im(f-(1-\alpha-\beta)^{-1}id)\cap(-\frac{\alpha+\beta}{1-\alpha-\beta},0)$
such that the function~\eqref{cond-phi-tilde} has more than one zero
points in $[0,1]$. Hence, we can conclude that the assumptions of
Theorem~\ref{general-synchro} imply that $\mathcal{Z}(\mathbf{F})$
contains only synchronization points. Therefore, this set is not empty
(see Theorem~\ref{general-as-conv}) and, by assumption, it is
finite. Applying Theorem~\ref{general-as-conv}, we obtain the almost
sure convergence of $\mathbf{Z}_n$ toward a random variable
$\mathbf{Z}_\infty$ taking values in the set
$\mathcal{Z}(\mathbf{F})$, and so of the form
$\mathbf{Z}_\infty=Z_\infty\mathbf{1}$, where $Z_\infty$ satisfies
Equation~\eqref{system-sincro}.
\end{proof}

\begin{remark}\label{remark-no-synchro-points} 
({\em Existence and characterization of the no-synchronization zero
    points}) \\ \rm It is worthwhile to underline that, from the above
  proof, we obtain that a necessary condition for the existence of
  no-synchronization zero points of $\mathbf{F}$ is that there exists
  $c\in
  Im(f-(1-\alpha-\beta)^{-1}id)\cap(-\frac{\alpha+\beta}{1-\alpha-\beta},0)$
  such that the corresponding function~\eqref{cond-phi-tilde} has more
  than one zero point in $[0,1]$. Moreover, if $\mathbf{z}^*$ is a
  no-synchronization zero point, then, for any fixed component
  $z_h^*$, each other component is a solution of $\widetilde{\varphi}
  = 0$, with $c=f(z_h^*)-z_h^*/(1-\alpha-\beta)\in
  Im(f-(1-\alpha-\beta)^{-1}id)\cap(-\frac{\alpha+\beta}{1-\alpha-\beta},0)$.
  Conversely, when $\mathbf{z}^*$ is a point with the above property,
  it is a zero point of $\mathbf{F}$ if and only if (since 
 ~\eqref{system}) we have 
\begin{equation}\label{system-c}
\alpha \frac{1}{N}\sum_{i=1}^N z^*_i+\beta q+(1-\alpha-\beta)c=0.
\end{equation}
\end{remark}

We conclude this section providing a very simple condition that allows
us to exclude the linearly unstable zero points (see Appendix
\ref{app-sto-approx}) from the set of possible limit points for the
process $(\mathbf{Z}_n)$.

\begin{theorem}\label{no-unstable} 
(No-convergence toward linearly unstable zero points)\\ If $f(0)>0$
  and $f(1)<1$, then, for each $ \mathbf{z}\in
  \mathcal{Z}(\mathbf{F})$ which is linearly unstable, we have
$$
P(\mathbf{Z}_n\to\mathbf{z})=0.
$$
\end{theorem} 

\begin{proof} 
We can apply Theorem~\ref{no-conv} in Appendix~\ref{app-sto-approx}. 
Fixed $v\in\mathbb{R}^N$ with $|v|=\sum_{h=1}^Nv_h=1$ and
$n\in\mathbb{N}$, consider the random variable
$$
X_{n+1}=\sum_{h=1}^Nv_h\Delta M_{n+1,h}=\sum_{h=1}^Nv_h(I_{n+1,h}-P_{n,h}),
$$ 
where $P_{n,h}=\alpha Z_n+\beta q+(1-\alpha-\beta)f(Z_{n,h})$. We
note that a partition of the sample space is given by the events of
the form
$$
E_{n+1,H}=\{I_{n+1,h}=1\;\forall h\in H,\;I_{n+1,h}=0\;\forall h\in H^c\},
$$ 
where $H$ is a subset of $\{1,\dots,N\}$ (the empty
set included). Therefore, we can write
$$
X_{n+1}=\sum_{H}
\left(\sum_{h\in H} v_h(1-P_{n,h})-\sum_{h\in H^c} v_hP_{n,h}\right)
I_{E_{n+1,H}}=\sum_{H} A_{n,h}I_{E_{n+1,H}},
$$ 
where the first sum is over all the possible subsets of
$\{1,\dots,N\}$ (the empty set included).  It follows that
$$
X_{n+1}^+=\sum_{H} A_{n,H}^+I_{E_{n+1,H}}
$$
and so 
$$
E[X_{n+1}^+\,|\,\mathcal{F}_n]=
\sum_H A_{n,H}^+E[I_{E_{n+1,H}}\,|\,\mathcal{F}_n]=
\sum_H A_{n,H}^+\prod_{h\in H}P_{n,h}\prod_{h\in H^c}(1-P_{n,h})
$$ (where we use the convention $\prod=1$ if $H$ or $H^c$ is empty).
Now, by assumption, $f$ has on $[0,1]$ a minimum value $m=f(0)>0$  
and a maximum value $M=f(1)<1$. Hence, we have
$$
0<(1-\alpha-\beta)m\leq P_{n,h}\leq \alpha+\beta+(1-\alpha-\beta)M<1
$$ 
and this fact implies $\prod_{h\in H}P_{n,h}\prod_{h\in
  H^c}(1-P_{n,h})\geq p>0$ for a suitable constant $p>0$. Moreover,
among the possible $H$, there is $H_*=\{h\in \{1,\dots,N\}:\, v_h\geq 0\}$
(possibly equal to the empty set) and, correspondingly, we have
\begin{equation*}
\begin{split}
A_{n,H_*}^+&=A_{n,H_*}\geq 
(1-\alpha-\beta)\min\{m,1-M\}\sum_{h\in H_*}v_h+\sum_{h\in H_*^c}(-v_h)
\\
&=
(1-\alpha-\beta)\min\{m,1-M\}\sum_{h=1}^N|v_h|
\\
&=(1-\alpha-\beta)\min\{m,1-M\}>0.
\end{split}
\end{equation*}
Thus, condition~\eqref{cond-no-conv} of Theorem~\ref{no-conv} is
satisfied with $C=(1-\alpha-\beta)\min\{m,1-M\}p>0$ and so
$P(\mathbf{Z}_\infty=\mathbf{z})=0$ for all the zero points
$\mathbf{z}$ of $\mathbf{F}$ that are linearly unstable.
\end{proof}


\section{Specific models}\label{Sec-examples} 
In this section, by means of the above general results, we analyze the
asymptotic behaviour of the systems related to the functions $f_{LP}$,
$f_{LogP}$ and $f_{\textrm{Tech}}$, introduced in Section~\ref{introduction}
(Subsec.~\ref{second-inter} and~\ref{third-inter}).
In section~\ref{Sec-simulations} some 	associated numerical illustrations will be presented.

\subsection{Case $f=f_{LP}$}
In this subsection we consider the function 
\begin{equation}\label{f-LP}
f(x)=f_{LP}(x)=\frac{x}{\theta(x+x^*)}
\qquad\mbox{with } \theta>0,\;\theta x^*\in (0,1],\;\theta x^* \geq 1-\theta\,.
\end{equation}
Note that we exclude the case $f_{LP}$ defined in~\eqref{f-LP-game}
with $\theta=0$, because it coincides with the case of a system of
interacting P\'olya urns with mean-field interaction and with or
without a ``forcing input'' $q$ and this model has been already
analyzed in~\cite{ale-cri-ghi, ale-cri-ghi-MEAN,
  ale-cri-ghi-WEIGHT-MEAN, cri-dai-lou-min, cri-dai-min, dai-lou-min}.
\\

The following result states that, provided that
$\mathbf{Z}_0\neq\mathbf{0}$ (note that, in applications, we generally
have $P(\mathbf{Z}_0\neq \mathbf{0})=1$), we always have the almost
sure asymptotic synchronization of the system and, moreover, it is
predictable.

\begin{theorem}\label{th-LP}
Let $f=f_{LP}$. Set
\begin{equation*}
\widehat{P}=
\begin{cases}
P\qquad
&\mbox{when } \beta>0,\\ 
P\qquad
&\mbox{when } \beta=0\;\mbox{and }\;\theta x^*=1,\\
P(\cdot|\mathbf{Z}_0\neq\mathbf{0})\qquad
&\mbox{when } \beta=0\;\mbox{and }\;\theta x^*<1.
\end{cases}
\end{equation*}
and 
\begin{equation}\label{zinfty}
z_\infty=
\begin{cases}
\hat{z}\qquad
&\mbox{when } \beta>0,\\ 
\frac{1-\theta x^*}{\theta}\qquad
&\mbox{when } \beta=0.\\
\end{cases}
\end{equation}
where $\hat{z}\in (0,1)$ depends on the model parameters and it is
defined as in~\eqref{point-sincro}.  Then, under $\widehat{P}$, the
system almost surely asymptotically synchronizes and it is
predictable: indeed, we have
$$
\mathbf{Z}_n \stackrel{a.s.}\longrightarrow
  z_\infty\mathbf{1}
$$
and 
$$
\overline{\mathbf{I}}_n=
\frac{1}{n}\sum_{k=1}^n\mathbf{I}_k \stackrel{a.s.}\longrightarrow
  z_\infty\mathbf{1}.
$$ 
\end{theorem}

Observe that, when $\beta>0$, the limit point $z_\infty \in (0,1)$. In
the first interpretation, this means that in the limit configuration
the $N$ agents keep a positive inclination for both actions; while in
the interpretation regarding the games (see
Subsec.~\ref{second-inter}), this means that in the limit
configuration, both strategies coexist in all the $N$ games. When
$\beta=0$, the limit point $z_\infty$ belongs to the entire interval
$[0,1]$, including the extremes: precisely, it is equal to $0$ when
$\theta x^*=1$ and equal to $1$ when $\theta x^*=1-\theta$. Therefore,
there is the possibility that, in the limit configuration, only one
inclination (or strategy) survives.  Furthermore, we note that the
limit value depends only on $\theta$ and $x^*$, but not on the
parameter $\alpha$, that rules the interaction.

\begin{proof} Let us look for the solutions of the
equation $\mathbf{F}(\mathbf{z})=\mathbf{0}$ in $[0,1]^N$, that is of
the system~\eqref{system}.  
\\ 
\noindent{\em Synchronization zero points.}  We
start looking for the solutions of~\eqref{system} of the type
$\mathbf{z}=z\mathbf{1}$, that is for the solution of
\eqref{system-sincro}. Taking into account that $f=f_{LP}$, we obtain
the second-order equation
\begin{equation}\label{eq-sincro}
\widehat{\varphi}(z)=(1-\alpha)\theta z^2
+[(1-\alpha)\theta x^*-\beta \theta q-(1-\alpha-\beta)]z
- \beta q\theta x^*=0. 
\end{equation}
We recall that, since The discriminant associated to this equation is
$$
\Delta=[(1-\alpha)\theta x^*-\beta \theta q-(1-\alpha-\beta)]^2
+4(1-\alpha)\beta \theta^2 q x^*.
$$ When $\beta=0$ and $\theta x^*<1$, we have two distinct solutions
in $[0,1]$, that is $0$ and $\frac{1-\theta x^*}{\theta}$, while, if
$\beta=0$ and $\theta x^*=1$, we have only one solution $z^*=0$.  When
$\beta>0$, we have $\Delta>0$ and so there are two distinct solutions
of~\eqref{eq-sincro}. However, we are interested only in solutions
belonging to $[0,1]$. Since $\varphi(0)>0$ and $\varphi(1)<0$,
there is at least one solution in $(0,1)$. Moreover, since in $\Delta$
we have the term $4(1-\alpha)\beta\theta^2 q x^*>0$, one of the
solutions is obviously strictly negative. Therefore, there is a unique
solution in $(0,1)$ given by
\begin{equation}\label{point-sincro}
\hat{z}=\frac{-[(1-\alpha)\theta x^*-\beta \theta
    q-(1-\alpha-\beta)]+\sqrt{\Delta}} {2(1-\alpha)\theta}.
\end{equation}
Summing up, synchronization zero points are of the form
$\mathbf{z}^*=z^*\mathbf{1}$ with
\begin{equation}\label{zetastar}
z^*\begin{cases}
\in \{0, \frac{1-\theta x^*}{\theta}\}\qquad
&\mbox{if } \beta=0,\  \theta x^*<1.\\
=0\qquad
&\mbox{if } \beta=0,\  \theta x^*=1.\\
=\hat{z}\qquad
&\mbox{if } \beta>0,\\ 
\end{cases}
\end{equation}
\noindent{\em No-syncronization zero points.}  Such zero points do not
exist: indeed, writing equation~\eqref{cond-phi-tilde} of Theorem
\ref{general-synchro} for $f=f_{LP}$ we obtain
$$
\theta z^2 
+[c(1-\alpha-\beta)\theta +\theta x^*-(1-\alpha-\beta)]z 
+c(1-\alpha-\beta)\theta x^*=0, 
$$ which, since $c<0$, admits at most one solution in $[0,1]$. 
\\ 
\noindent{\em Almost sure asymptotic synchronization.} We have proven
above that the set $\mathcal{Z}(\mathbf{F})$ contains only a finite
number of points. Moreover, $f$ admits the primitive function
$$
\phi(x)=\frac{1}{\theta}\left[x-x^*\ln(x+x^*)\right]+const.
$$ Then, by Theorem~\ref{general-as-conv} and Theorem
\ref{general-synchro}, we can conclude that the system almost surely
asymptotically synchronizes:
$$
\mathbf{Z}_n\stackrel{a.s.}\longrightarrow \mathbf{Z}_\infty=Z_\infty\mathbf{1}
$$
and 
$$
\overline{\mathbf{I}}_n=
\frac{1}{n}\sum_{k=1}^n\mathbf{I}_k \stackrel{a.s.}\longrightarrow
  \mathbf{Z}_\infty=Z_\infty\mathbf{1},
$$ where $Z_\infty$ can take the values $z^*$ specified above. In
  particular, when we are in the case $\beta>0$ or in the case
  $\beta=0$ and $\theta x^*=1$, we have a unique possible value for
  $z^*$ and so the system is predictable. It remains to prove that,
  under $\widehat{P}=P(\cdot|\mathbf{Z}_0\neq\mathbf{0})$, the system
  is predictable with the unique limit point $\frac{1-\theta
    x^*}{\theta}\mathbf{1}$. The following step provides the proof of
  this fact.
\\
\noindent{\em Case $\beta=0$ and $\theta x^*<1$: predicibility under
  $\widehat{P}$.}  Let us consider the case $\beta=0$ and $\theta
x^*<1$, for which we have $\mathcal{Z}(\mathbf{F})=\{\mathbf{0},
(\frac{1-\theta x^*}{\theta})\mathbf{1}\}$.  For
$\mathbf{z}^*=z^*\mathbf{1}$, Corollary~\ref{cor-autovalori-sincro}
provides the eigenvalues of $J(\mathbf{F})(\mathbf{z}^*)$, that is
$$
(1-\alpha-\beta)f'(z^*)-1\qquad\mbox{and}
\qquad (1-\alpha-\beta)f'(z^*)-1+\alpha.
$$ Now, the eigenvalues for $\mathbf{z}^*=(\frac{1-\theta
  x^*}{\theta})\mathbf{1}$ are $(1-\alpha)\theta x^*-1<0$ and
$-(1-\alpha)(1-\theta x^*)<0$, and so $\mathbf{z}^*$ is strictly
stable; while the eigenvalues for $\mathbf{0}$ are $(1-\alpha)(\theta
x^*)^{-1}-1$, that can be positive or negative, and
$-(1-\alpha)(1-1/\theta x^*)>0$, so that $\mathbf{0}$ is linearly
unstable. However, we cannot exclude convergence toward $\mathbf{0}$
by means of Theorem~\ref{no-unstable}, because $f(0)=0$. Anyway, we
observe that, if $\mathbf{Z}_0\neq \mathbf{0}$, then $\mathbf{Z}_n\neq
\mathbf{0}$ for all $n$. Hence, if we prove for
$\mathbf{z}^*=\frac{1-\theta x^*}{\theta}\mathbf{1}$, that
\begin{equation}\label{cond-suf-LP}
\langle \mathbf{F}(\mathbf{z}), \mathbf{z}-\mathbf{z}^*\rangle=
\langle \mathbf{F}(\mathbf{z})- \mathbf{F}(\mathbf{z}^*), 
\mathbf{z}-\mathbf{z}^*\rangle<0
\end{equation}
for all $\mathbf{z}=(z_1,\ldots, z_N)^T
\in[0,1]^N\setminus\mathcal{Z}(\mathbf{F})$, then we can conclude by
Theorem~\ref{cond-suff} that, under
$\widehat{P}=P(\cdot|\mathbf{Z}_0\neq\mathbf{0})$, the system is
predictable. In order to prove~\eqref{cond-suf-LP}, we observe that
$f'$ is positive and strictly decreasing on $[0,1]$ and
$f'(z^*)=\theta x^*<1$ by hypothesis. Then, recalling that
$f(z)-f(z^*)<0$ for $z<z^*$ and, using the mean value theorem for
$z>z^*$, we obtain that $f(z)-f(z^*)<|z-z^*|$ for all $z\in [0,1], z\neq
z^*$. Then, since $z_h\neq z^*$ for at least one $h\in\{1,\ldots,
N\}$, we have
\begin{equation}\label{conti}
\begin{split}
&\langle \mathbf{F}(\mathbf{z})-\mathbf{F}(\mathbf{z}^*), 
\mathbf{z}-\mathbf{z}^*\rangle
=\frac{\alpha}{N}\left[\sum_{h=1}^N(z_h-z^*)\right]^2-\sum_{h=1}^N(z_h-z^*)^2
+(1-\alpha)\sum_{h=1}^N(f(z_h)-f(z^*))(z_h-z^*)
\\
&< -(1-\alpha)\sum_{h=1}^N(z_h-z^*)^2
+(1-\alpha)
\sum_{h=1}^N (z_h-z^*)^2=0.
\end{split}
\end{equation}
Finally, regarding the almost sure convergence of the empirical means
under $\widehat{P}$, we observe that the proof given for Theorem
\ref{general-as-conv} also works with
$\widehat{P}=P(\cdot|\mathbf{Z}_0\neq 0)$, because
$\{\mathbf{Z}_0\neq\mathbf{0}\}\in\mathcal{F}_0$.
\end{proof}

\begin{remark}
({\em Possible asymptotic synchronization toward $1/2$})\\ \rm We
  recall that the almost sure asymptotic synchronization of the system
  toward the value $1/2$ means that in the limit the two inclinations
  (in the first interpretation) or the two strategies in all the games
  (in the second interpretation) coexist in the proportion $1:1$. To
  this regard, we observe that $1/2\mathbf{1}$ is a synchronization
  zero point for the case $f=f_{LP}$ if and only if we have
\begin{equation}\label{value-1mezzo-LP}
\begin{split}
f_{LP}(1/2)-\frac{1-\alpha}{2(1-\alpha-\beta)}+\frac{\beta q}{1-\alpha-\beta}
&=0,\qquad\hbox{that is}\\
\frac{(\theta+2\theta x^*)}{2}\left(1-\alpha-2\beta q\right)
=1-\alpha-\beta, 
\end{split}
\end{equation}
that, in particular, implies $(1-\alpha)>2\beta[q\vee (1-q)]$ (because
$f_{LP}(1/2)\in (0,1)$).  Therefore, only when condition
\eqref{value-1mezzo-LP} is satisfied, the system almost surely
asymptotically synchronizes toward $1/2$. Note that, in the special
case when $\beta=0$ (which includes the case $N=1,\,\alpha=\beta=0$
that corresponds to the one studied in~\cite{Fagiolo}), condition
\eqref{value-1mezzo-LP} simply becomes $\theta x^*=1-\theta/2$.
\end{remark}

Applying Theorem~\ref{clt}, we can provide also the rate of
convergence of $(\mathbf{Z}_n)$. More precisely, we have the following
result:

\begin{remark}\label{remark-tlc-LP} ({\em Rate of convergence})\\\rm 
With the same assumptions and notation as in Theorem~\ref{th-LP}, we
have
 $$ \Delta M_{n+1,h}\Delta
  M_{n+1,j}=(I_{n+1,h}-P_{n,h})(I_{n+1,j}-P_{n,j}),
$$ where $P_{n,h}$ is defined in~\eqref{prob-model-intro}, and so, for
  $h\neq j$, by conditional independence, we get $E[\Delta
    M_{n+1,h}\Delta M_{n+1,j}\,|\,\mathcal{F}_n]=0$, and, for $h=j$,
  taking into account that $\mathbf{F}(\mathbf{z}_\infty)=\mathbf{0}$
  for $\mathbf{z}_\infty=z_\infty\mathbf{1}$,
$$
E[(\Delta M_{n+1,h})^2\,|\,\mathcal{F}_n]=P_{n,h}-P_{n,h}^2
\stackrel{a.s.}\longrightarrow z_{\infty} -z_{\infty}^2
\qquad\mbox{w.r.t. } \widehat{P}.
$$ Moreover, by Corollary \ref{cor-autovalori-sincro}, the smallest
eigenvalue of $-J(\mathbf{F})(z_\infty\mathbf{1})$ is
$\lambda=(1-\alpha)-(1-\alpha-\beta)f'(z_\infty)$. Therefore, applying
Theorem~\ref{clt} when $x_\infty\in (0,1)$, we obtain, under
$\widehat{P}$:
\begin{itemize} 
\item if $\lambda>1/2$, then $$
\sqrt{n}(\mathbf{Z}_n-z_\infty\mathbf{1})\stackrel{d}\longrightarrow
\mathcal{N}\left(\mathbf{0}, \Sigma\right),
$$ 
where $\Sigma$ is a suitable matrix of the form
  $z_\infty(1-z_\infty)\widehat{\Sigma}$;
\item if $\lambda=1/2$, then 
$$
\sqrt{\frac{n}{\ln(n)}}
(\mathbf{Z}_n-z_\infty\mathbf{1})\stackrel{d}\longrightarrow
\mathcal{N}\left(\mathbf{0}, 
\Sigma\right),
$$ 
where $\Sigma$ is a suitable matrix of the form
$z_\infty(1-z_\infty)\widehat{\Sigma}$;
\item if $0<\lambda<1/2$, then 
$$
n^\lambda(\mathbf{Z}_n-z_\infty\mathbf{1})\stackrel{a.s.}\longrightarrow V,
$$
where $V$ is a suitable finite random variable.
\end{itemize}
In particular, when $\lambda>1/2$, using Remark
\ref{clt-alternative-expression}, we obtain
$\Sigma=z_\infty(1-z_\infty)
\left(-2J(\mathbf{F})(\mathbf{z}_\infty)-Id\right)^{-1}$.
\end{remark}


\subsection{Case $f=f_{LogP}$} 

In this subsection, we consider the function 
\begin{equation}\label{f-LogP}
f(x)=f_{LogP}(x)=\frac{1}{1+\exp(-\theta(x-x^*))}
\qquad\mbox{with } x^*\in (0,1),\; \theta>0\,.
\end{equation}
It is a sigmoid function, \textit{i.e.}~its first derivative is a strictly
positive function, which is strictly increasing on $[0, x^*)$ and
  strictly decreasing on $(x^*, 1]$ with a maximum given by
$f'(x^*)=\theta/4$.  Furthermore, we have
$f'(x)=f'(2x^*-x)$ for all $x\in [0,1]$.  \\

\indent The following lemma provides a description of the subset of
$\mathcal{Z}(\mathbf{F})$ containing all the zero points of
$\mathbf{F}$ that are synchronization points (more briefly,
``synchronization zero points'').

\begin{lemma}\label{lemma-synchro-points} (Synchronization zero points)\\
Let $f=f_{LogP}$. Then, accordingly to the values of the parameters,
$\mathcal{Z}(\mathbf{F})$ contains at least three synchronization zero
points. Moreover, at most two of them are stable.  In particular, if
one of the following conditions is satisfied, $\mathbf{F}$ has a
unique stable synchronization zero point:
\begin{itemize}
\item[U1)] $\theta/4\leq (1-\alpha)/(1-
\alpha-\beta)$ or 
\item[U2)] $f'(0)\vee f'(1)\geq
(1-\alpha)/(1-\alpha-\beta)$ or 
\item[U3)] $f'(0)\vee f'(1) < (1-\alpha)/(1-\alpha-\beta) < \theta /4$
  and either
  $f(\widehat{x}_1)>(1-\alpha)\widehat{x}_1/(1-\alpha-\beta)-\beta
  q/(1-\alpha-\beta)$ or
  $f(\widehat{x}_2)<(1-\alpha)\widehat{x}_2/(1-\alpha-\beta)-\beta
  q/(1-\alpha-\beta)$, where $\widehat{x}_1\in (0,x^*)$ and
  $\widehat{x}_2=2x^*-\widehat{x}_1\in (x^*,1)$ are the solutions of
  $f'=(1-\alpha)/(1-\alpha-\beta)$.
\end{itemize}
Otherwise, $\mathbf{F}$ has two stable synchronization zero points
belonging to $(0,\widehat{x}_1]\cup [\widehat{x}_2,1)$ (more
    precisely, one in each of these two intervals).
\end{lemma}

\begin{proof}
We recall (see Theorem~\ref{general-as-conv}) that there exists at
least one synchronization zero point of $\mathbf{F}$ and points of
this type are of the form $\mathbf{z}=z\mathbf{1}$ with
$\varphi(z)=0$, where
$$\varphi(z)=f(z)-(1-\alpha)/(1-\alpha-\beta)z+\beta
q/(1-\alpha-\beta).$$ Note that $\varphi(z)$ is of the form
$f(z)-\delta z+cost$ with $\delta=(1-\alpha)/(1-\alpha-\beta)$ and a
suitable constant $cost$ such that $\varphi(0)>0$ and $\varphi(1)<0$
(note that $f(0)>0$ and $f(1)<1$). Hence, we have $\varphi'=f'-\delta$
and $\varphi''=f''$.  Therefore, recalling that $f$ is a sigmoid
function with $\max_{[0,1]} f'=f'(x^*)=\theta/4$ and the symmetry of
$f'$, we get that equation $\varphi'(x)=0$, \textit{i.e.}~$f'(x)=\delta$, has
at most two solutions on $[0,1]$ and we can have the following cases:
\begin{itemize}
\item[1)] $\theta/4\leq \delta$ or
\item[2)] $f'(0)\vee f'(1)\geq \delta$ (and so $\theta/4>\delta$) or
\item[3)] $f'(0)\vee f'(1)<\delta<\theta/4$ and, letting
  $\widehat{x}_1\in (0, x^*)$ and $\widehat{x}_2=2x^*-\widehat{x}_1\in
  (x^*,1)$ be the solutions of $\varphi'=0$, we have either
  $\varphi(\widehat{x}_1)>0$ or $\varphi(\widehat{x}_2)<0$, or
\item[4)] $f'(0)\vee f'(1)<\delta<\theta/4$ and, letting
  $\widehat{x}_1,\,\widehat{x}_2$ as in c), we have either
  $\varphi(\widehat{x}_1)=0$ or $\varphi(\widehat{x}_2)=0$,
\item[5)] $f'(0)\vee f'(1)<\delta<\theta/4$ and, letting $\widehat{x}_1,\, 
  \widehat{x}_2$ as in c), we have $\varphi(\widehat{x}_1)<0$ and
  $\varphi(\widehat{x}_2)>0$.
\end{itemize}
(Note that 1), 2) and 3) coincide, respectively, with conditions U1),
U2) and U3) in the statement.) In case 1), $\varphi'$ is strictly
negative on $[0,1]\setminus\{x^*\}$, that is $\varphi$ is strictly
decreasing on $[0,1]$, and, since $\varphi(0)>0$ and $\varphi(1)<0$,
this fact implies that $\varphi=0$ has a unique solution in $(0,1)$.
Now, assume to be in case 2). Observe that, since $\varphi(0)>0$ and
$\varphi(1)<0$, it holds $\varphi'(0)\wedge \varphi'(1)<0$ (otherwise
$\varphi$ would be increasing on $[0,1]$, yielding a contradiction).
Now, when ${\varphi}'(0)\geq 0$, then $\varphi'(x)> 0$ for all $x\in
(0,x^*]$, \textit{i.e.}, $\varphi$ is strictly increasing on $(0, x^*]$; this
    fact implies that $\varphi(x)>0$ for all $x\in [0, x^*]$.
    Consequently, $\varphi$ has at most one zero point $z^*\in (x^*,
    1)$, because $\varphi(1)<0$ and $\varphi'$ is strictly decreasing
    on $(x^*,1]$.  Analogously, if $\varphi'(1)\geq 0$, then $\varphi$
      has at most one zero point $z^*\in (0, x^*)$. In case 3),
      $\widehat{x}_1<x^*$ and $\widehat{x}_2=2x^*-\widehat{x}_1>x^*$
      are respectively points of local minimum and local maximum of
      $\varphi$ in $(0,1)$. Now, if $\varphi(\widehat{x}_1)>0$, then
      $\varphi$ has a unique zero $z^*\in (\widehat{x}_2, 1)$.
      Analogously, if $\varphi(\widehat{x}_2)<0$, then $\varphi$ has a
      unique zero $z^*\in (0, \widehat{x}_1)$. In case 4), the
      function $\varphi$ has two zero points: more precisely, if
      $\widehat{x}_1$ is a zero point of $\varphi$, then
      $\varphi(\widehat{x}_2)>0$ and the other zero point of $\varphi$
      belongs to $(\widehat{x}_2, 1)$; if $\widehat{x}_2$ is a zero
      point of $\varphi$, then $\varphi(\widehat{x}_1)<0$ and the
      other zero point of $\varphi$ belongs to
      $(0,\widehat{x}_1)$. Finally, in case 5), then $\varphi$ has
      three zero points: one in $(0,\widehat{x}_1)$, one in
      $(\widehat{x}_1, \widehat{x}_2)$ and the last in $(\widehat{x}_2,
      1)$.\\ \indent Regarding the stability of the synchronization
      zero points of $\mathbf{F}$, we observe that, when
      $\mathbf{z}=z\mathbf{1}$, where $z$ is a zero point of
      $\varphi$, by Corollary~\ref{cor-autovalori-sincro}, the
      eigenvalues of $J(\mathbf{F})(\mathbf{z})$ are given by
$$
(1-\alpha-\beta)f'(z)-1\qquad\mbox{and}
\qquad (1-\alpha-\beta)f'(z)-1+\alpha\,,
$$  
that is 
$$
(1-\alpha-\beta)\varphi'(z)-\alpha\qquad\mbox{and}
\qquad (1-\alpha-\beta)\varphi'(z).
$$ Therefore, in cases 1), 2) and 3), the unique synchronization zero
point is stable, because the corresponding eigenvalues are both
negative. In case 4), recalling that $\varphi'$ is strictly negative
on $(0,\widehat{x}_1)$ and on $(\widehat{x}_1,1)$ and strictly
positive on $(\widehat{x}_1,\widehat{x}_2)$, both synchronization zero
points are stable. For the same reason, in case 5), the
synchronization zero point strictly smaller than $\widehat{x}_1$ and
the one strictly bigger than $\widehat{x}_2$ are stable, while the one
in $(\widehat{x}_1,\widehat{x}_2)$ is linearly unstable.
\end{proof}

\begin{remark}\label{remark-x^*}
\rm Note that if $x^*\mathbf{1}$ is a synchronization zero point of
$\mathbf{F}$, that is $(1-\alpha)(1-2x^*)-\beta(1-2q)=0$, then U2) is
not possible, because, as shown in the above proof, in that case the
unique zero point of $\varphi$ is necessarily different from $x^*$.
Moreover, cases U3) and 4) are also not possible. Indeed,
$f-(1-\alpha)id/(1-\alpha-\beta)$ is strictly increasing on
$(\widehat{x}_1,\widehat{x}_2)$ and so we have
$f(\widehat{x}_1)-(1-\alpha)\widehat{x}_1/(1-\alpha-\beta)
<f(x^*)-(1-\alpha)x^*/(1-\alpha-\beta)= -\beta
q/(1-\alpha-\beta)<f(\widehat{x}_2)-(1-\alpha)\widehat{x}_2/(1-\alpha-\beta)$. 
Therefore, when $x^*\mathbf{1}$ is a synchronization zero
point of $\mathbf{F}$, it is stable if and only if U1) is
satisfied. Otherwise, there are three synchronization zero points:
$x^*$ (linearly unstable) and two stable, say
$\mathbf{z}_1^*=z_1^*\mathbf{1}$ and $\mathbf{z}_2^*=z_2^*\mathbf{1}$,
with $0<z_1^*<\widehat{x}_1<x^*<\widehat{x}_2<z_2^*=2x^*-z_1^*<1$.
\end{remark}

As an immediate consequence of Lemma~\ref{lemma-synchro-points}, we
get that, if the system almost surely asymptotically synchronizes and
one of the conditions U1), U2) and U3) holds true, then it is
predictable.  In the next results (see Theorems~\ref{th-LogP-sincro}
and~\ref{th-LogP-nosincro}) we give sufficient conditions for the
almost sure asymptotic synchronization of the system. Moreover, we
provide a characterization of the possible limit points, that are not
synchronization points (see Theorem~\ref{th-LogP-nosincro}).

\begin{theorem}\label{th-LogP-sincro}
Let $f=f_{LogP}$.  Assume that one of the following conditions hold:
\begin{itemize}
\item[S1)] $\theta/4\leq 1/(1-\alpha-\beta)$ or
\item[S2)] $f'(0)\vee f'(1)\geq 1/(1-\alpha-\beta)$ (and so
  $\theta/4>1/(1-\alpha-\beta)$) or
\item[S3)] $f'(0)\vee f'(1)<1/(1-\alpha-\beta)<\theta/4 $ and, letting
  $x_1^*\in (0,x^*)$ and $x_2^*=2x^*-x_1^*\in (x^*,1)$ be the
  solutions of $f'=1/(1-\alpha-\beta)$, we have either $f(x_1^*)\geq
  x_1^*/(1-\alpha-\beta)$ or $f(x_2^*)\leq x_2^*/(1-\alpha-\beta)
  -(\alpha+\beta)/(1-\alpha-\beta)$.
\end{itemize}
Then, we have the almost sure asymptotic synchronization of the
system, \textit{i.e.}
$$
\mathbf{Z}_n\stackrel{a.s.}\longrightarrow\mathbf{Z}_\infty
$$ 
and 
$$
\overline{\mathbf{I}}_n=\frac{1}{n}\sum_{k=1}^n\mathbf{I}_k
\stackrel{a.s.}\longrightarrow \mathbf{Z}_\infty,
$$ where $\mathbf{Z}_\infty$ is a random variable of the form
$\mathbf{Z}_\infty=Z_\infty\mathbf{1}$. Moreover, the random variable
$\mathbf{Z}_\infty$ takes values in the set of the stable zero points
of $\mathbf{F}$, which is contained in $(0, 1)^N$ and consists of at
most two different points.
\end{theorem}
\begin{proof}
We want to apply Theorem~\ref{general-synchro} and Theorem
\ref{no-unstable}.  Observe first that $f=f_{LogP}$ admits the
primitive function
$$ \phi(x)= x + \frac{1}{\theta} \ln\left( 1+e^{-\theta(x-x^*)}
\right)+const
$$ and, by Lemma~\ref{lemma-synchro-points}, the set of the
synchronization zero points of $\mathbf{F}$ is finite.  Now, consider
the function $\widetilde\varphi$ defined in Theorem
\ref{general-synchro}. Observe that this function has the same form of
$\varphi$: indeed, we have $\widetilde{\varphi}(z)=f(z)-\delta
z+cost$, with $\delta=1/(1-\alpha-\beta)$ and $cost=-c\in
(0,\frac{\alpha+\beta}{1-\alpha-\beta})$ and so such that
$\widetilde\varphi(0)>0$ and $\widetilde\varphi(1)<0$.  Therefore,
arguing exactly as in the proof of Lemma~\ref{lemma-synchro-points},
with $\widetilde\varphi$ in place of $\varphi$ and
$\delta=1/(1-\alpha-\beta)$, we obtain that each of the above
conditions S1), S2) and S3) implies that, for all $c\in
(-\frac{\alpha+\beta}{1-\alpha-\beta}, 0)$, the function
$\widetilde\varphi$ has exactly one zero point in $[0,1]$.  Indeed,
S1) and S2) correspond to condition 1) and 2) in the proof of
Lemma~\ref{lemma-synchro-points}, while condition S3) implies, for all $c\in
(-\frac{\alpha+\beta}{1-\alpha-\beta}, 0)$, that $\widetilde\varphi$
satisfies condition 3) in the proof of Lemma~\ref{lemma-synchro-points}.
Applying Theorem~\ref{general-synchro} we obtain the almost sure
asymptotic synchronization of the system, that is
$\mathbf{Z}_n\stackrel{a.s.}\longrightarrow
\mathbf{Z}_\infty=Z_\infty\mathbf{1}$, where $Z_\infty$ takes values
in the set of the zero points of $\varphi$. Moreover, recalling that
$f(0)>0$ and $f(1)<1$, we can also apply Theorem~\ref{no-unstable} and
conclude that the support of the limit random variable
$\mathbf{Z}_\infty$ only consists of the zero points of $\varphi$ that
give rise to a stable synchronization zero point of $\mathbf{F}$.  By
Lemma~\ref{lemma-synchro-points}, such points belong to $(0,1)$ and
they are at most two.
\end{proof}

Next theorem deal with the case not covered by Theorem
\ref{th-LogP-sincro}. In particular, analyzing the stability of
eventual ``no-synchronization zero point'' of~$\mathbf{F}$, we provide
another condition under which we have the almost sure asymptotic
synchronization of the system (see condition S4) below). Moreover, we
characterize the possible ``no-synchronization limit configurations'' 
for the system.

\begin{theorem}\label{th-LogP-nosincro}
Let $f=f_{LogP}$ and suppose 
\begin{equation}\label{complementary}
f'(0)\vee f'(1)<\frac{1}{(1-\alpha-\beta)}<\theta/4,\quad 
f(x^*_1)< \frac{x_1^*}{(1-\alpha-\beta)}, \quad 
f(x^*_2)>
\frac{x_2^*-(\alpha+\beta)}{(1-\alpha-\beta)}, 
\end{equation}
where $x^*_1\in (0,x^*)$ and $x_2^*=2x^*-x^*_1\in (x^*,1)$ are the
solutions of $f'=1/(1-\alpha-\beta)$.  Moreover, assume that
$\mathcal{Z}(\mathbf{F})$ is finite.  Then
$$
\mathbf{Z}_n\stackrel{a.s.}\longrightarrow\mathbf{Z}_\infty
$$ 
and 
$$
\overline{\mathbf{I}}_n=\frac{1}{n}\sum_{k=1}^n\mathbf{I}_k
\stackrel{a.s.}\longrightarrow \mathbf{Z}_\infty,
$$ where $\mathbf{Z}_\infty$ takes values in the set
$\mathcal{SZ}(\mathbf{F})$ of the stable zero points of $\mathbf{F}$,
which is contained in $(0,1)^N$. Such set always contains at most two
synchronization zero points. Moreover, if $(1-\alpha-\beta)
\left[f'(0)+f'(1)\right]<1+(1-\alpha)$, any 
$\mathbf{z}_\infty\in \mathcal{SZ}(\mathbf{F})$ which is not a
synchronization point, has the form, up to permutations,
$\mathbf{z}_\infty=(z_{\infty,1}, \ldots, z_{\infty,N})^T$ with
\begin{equation}\label{final-no-sincro-points}
z_{\infty,h}=\begin{cases}
\tilde{z}_{\infty,1}\in (0, x^*_1) \qquad &\mbox{for } h=1,\dots, N_1\\
\tilde{z}_{\infty,2}\in (x^*_2, 1)\qquad &\mbox{for } h=N_1+1,\dots, N,
\end{cases}
\end{equation} 
and $N_1\in\{1,\dots, N-1\}$. On the other hand, if 
\begin{equation}\label{cond-CS}
(1-\alpha-\beta)
\left[f'(0)+f'(1)\right] \geq 1+(1-\alpha), 
\end{equation}  
$\mathcal{SZ}(\mathbf{F})$  contains only synchronization points (and
so we have the almost sure asymptotic synchronization of the system).
\end{theorem}

Summing up, taking $f=f_{LogP}$, if we are in the case S1) or S2) or
S3) or S4) $\mathcal{Z}(\mathbf{F})$ finite and~\eqref{complementary} and~\eqref{cond-CS} are satisfied,
then, we have the almost sure asymptotic synchronization of the
system. Otherwise, we may have a non-zero probability that the system
does not synchronize as time goes to $+\infty$. More precisely,
provided that $\mathcal{Z}(\mathbf{F})$ is finite, the system almost
surely converges, and we have a non-zero probability of asymptotic
synchronization, but we may also have a non-zero probability of
observing the system splitting into two groups of components that
converge towards two different values.  We also point out that the
above results state that the random limit $\mathbf{Z}_\infty$ always
belongs to $(0,1)^N$. In the first interpretation, this fact means
that in the limit configuration the $N$ agents always keep a strictly
positive inclination for both actions; while in the interpretation
regarding the games, this fact means that in the limit configuration,
both strategies coexist in all the $N$ games.  Moreover, regarding the
possible ``no-synchronization limit configurations'' we have that,
independently from the value of~$N$, we always have at most two groups
of agents (or games) that approach in the limit two different
values. We never have a more complicated asymptotic fragmentation of
the whole system. Furthermore, we are able to localize the two limit
values: one is strictly smaller than $x_1^*<x^*$ and the other
strictly bigger than $x_2^*>x^*$, where the points $x_i^*$ only depend
on $x^*,\,\theta$ and on $(1-\alpha-\beta)$, which is the ``weight''
of the personal inclination component of $P_{n,h}$ in
\eqref{prob-model-intro}. In addition, Remark~\ref{remark-restriction}
(after the proof) may provide some information on the sizes of the two
groups.
In section~\ref{Sec-simulations} some numerical illustrations are presented. In Fig.~\ref{f4-setparam1:figures} the chosen set of parameters is such that there is only one (stable) synchronization zero point.
 In Fig.~\ref{f4-setparam2:figures} ($\beta \neq 0$) and  Fig.~\ref{f4-setparam3:figures} ($\beta=0$) there are two stable synchronization  zero points and  there are stable no-synchronization zero points.

\begin{proof}
Almost sure convergence follows from the fact that $f$ admits a
primitive function (see the proof of Theorem~\ref{th-LogP-sincro}
above) and $\mathcal{Z}(\mathbf{F})$ is finite, so that we can apply
Theorem~\ref{general-as-conv}. Moreover, since $f(0)>0$ and $f(1)<1$,
by Theorem~\ref{no-unstable} we get that the random variable
$\mathbf{Z}_\infty$ takes values in the set of the stable points of
$\mathcal{Z}(\mathbf{F})$. By Lemma~\ref{lemma-synchro-points}, this
set always contains one or two synchronization points. Let us now
investigate about the existence of stable no-synchronization zero
points of $\mathbf{F}$.  According to Remark
\ref{remark-no-synchro-points}, a necessary condition for the
existence of a solution $\mathbf{z}^*=(z_1^*,\ldots, z_N^*)^T$ of
\eqref{system} with $z_h^*\neq z_j^*$ for at least one pair of indexes
$h,j$, is that there exists $c\in Im(f-(1-\alpha-\beta)^{-1}id) \cap
\left(-\frac{\alpha+\beta}{1-\alpha-\beta},0\right)$ such that the
corresponding function $\widetilde\varphi$ defined in
\eqref{cond-phi-tilde} has more than one zero point in $[0,1]$.  To
this regard, we observe that the assumptions~\eqref{complementary} and
the fact that $f-(1-\alpha-\beta)^{-1}id$ is continuous and strictly
increasing on $(x_1^*,x_2^*)$ imply that
$Im(f-(1-\alpha-\beta)^{-1}id) \cap
\left(-\frac{\alpha+\beta}{1-\alpha-\beta},0\right)$ coincides with
$$
I=\left( 
f(x_1^*)-\frac{x_1^*}{(1-\alpha-\beta)}\vee 
-\frac{(\alpha+\beta)}{1-\alpha-\beta},\,
f(x_2^*)-\frac{x_2^*}{(1-\alpha-\beta)}\wedge 0
\right)
$$ and it is not empty. Moreover, for each $c$ belonging to this set,
the corresponding function $\widetilde\varphi$ has the same form of
$\varphi$: indeed, we have $\widetilde{\varphi}(z)=f(z)-\delta
z+cost$, with $\delta=1/(1-\alpha-\beta)$ and $cost=-c$ such that
$\widetilde\varphi(0)>0$ and $\widetilde\varphi(1)<0$.  Therefore,
arguing exactly as in the proof of Lemma~\ref{lemma-synchro-points},
with $\widetilde\varphi$ in place of $\varphi$ and
$\delta=1/(1-\alpha-\beta)$, we obtain that the assumptions
\eqref{complementary} imply that equation $\widetilde{\varphi}=0$ has
two or three distinct solutions in $(0,1)$ (see cases 4) and 5) in the
proof of Lemma~\ref{lemma-synchro-points}). More precisely, the
equation $\widetilde{\varphi}'=0$, that is $f'=\delta$, has exactly
two solutions $x_1^*, x_2^*\in (0, 1) $ (with
$x^*_1<x^*<x^*_2=2x^*-x^*_1$), which are respectively points of local
minimum and local maximum of $\widetilde{\varphi}$ in $(0,1)$;
moreover, since $c\in I$, we have $\widetilde{\varphi}(x_1^*)\leq 0$
and $\widetilde{\varphi}(x_2^*)\geq 0$. Therefore, $\widetilde\varphi$
has two zero points (case 4)), one in $\{x_1^*,x_2^*\}$ and the other
in $(0,x_1^*)\cup(x_2^*,1)$, or it has three zero points, one in
$(x_1^*,x_2^*)$ and the other two in $(0,x_1^*)\cup(x_2^*,1)$.  Hence,
recalling again Remark~\ref{remark-no-synchro-points}, if
$\mathbf{z}^*$ is a no-synchronization zero point of $\mathbf{F}$,
then, fixed a component $z^*_h$, each other component is a solution of
$\widetilde{\varphi}=0$ with $c=f(z^*_h)-(1-\alpha-\beta)^{-1}z^*_h\in
I$, and so its components belong to $(0,1)$ and are, up to
permutations, of the following form
\begin{equation}\label{type-1}
z_h^*=\begin{cases}
\tilde{z}_1=\zeta_1(\tilde{z}_2)\qquad&\mbox{for } h=1,\dots,N_1\\
\tilde{z}_2 \qquad&\mbox{for } h=N_1+1,\dots,N_2\\
\tilde{z}_3=\zeta_3(\tilde{z}_2)\qquad&\mbox{for } h=N_2+1,\dots, N_3,
\end{cases}
\end{equation}
where $N_i\in\{0,\dots, N-1\}$, $\tilde{z}_1\leq\tilde{z}_2\leq
\tilde{z}_3$, $\tilde{z}_2\in [x^*_1, x^*_2]$, 
\begin{equation*}
\zeta_1(\tilde{z}_2)
\begin{cases}
=\tilde{z}_2=x^*_1\quad\mbox{if } \tilde{z}_2=x^*_1\\
<x^*_1\quad\mbox{if } \tilde{z}_2\in (x^*_1, x^*_2]
\end{cases}
\qquad
\zeta_3(\tilde{z}_2)
\begin{cases}
=\tilde{z}_2=x^*_2\quad\mbox{if } \tilde{z}_2=x^*_2\\
>x^*_2=2x^*-x^*_1\quad\mbox{if } \tilde{z}_2\in [x^*_1 , x^*_2),
\end{cases}
\end{equation*}
and (see~\eqref{system-c} in Remark~\ref{remark-no-synchro-points}) 
\begin{equation}\label{eq-z_2tilde}
\alpha \frac{1}{N}
\left(N_1\zeta_1(\tilde{z}_2)+N_2\tilde{z}_2+N_3\zeta_3(\tilde{z}_2)\right)
+\beta q+
(1-\alpha-\beta)c=0\,.
\end{equation}
Finally, let us study the stability of such a point.  Note that, since
$\tilde{z}_2\in [x^*_1, x^*_2]$, we have
$\widetilde{\varphi}'(\tilde{z}_2)=f'(\tilde{z}_2)-
1/(1-\alpha-\beta)\geq 0$. Moreover, for $\tilde{z}_2\in (x^*_1,
x^*_2)$, we have
$\widetilde{\varphi}'(\tilde{z}_i)=f'(\tilde{z}_i)-1/(1-\alpha-\beta)<0$,
for $i=1,\,3$, while if $\tilde{z}_2=x^*_1=\tilde{z}_1$
(respectively, $\tilde{z}_2=x^*_2=\tilde{z}_3$), we have
necessarily
$\widetilde{\varphi}'(\tilde{z}_3)=f'(\tilde{z}_3)-1/(1-\alpha-\beta)<0$
(respectively,
$\widetilde{\varphi}'(\tilde{z}_1)=f'(\tilde{z}_1)-1/(1-\alpha-\beta)<0$).
Therefore, if $N_2\neq 0$, we have
$f'(z^*_h)>(1-\frac{\alpha}{N})/(1-\alpha-\beta)$ for all $h\in
\{N_1+1,\ldots, N_2\}$.  Now, for $\mathbf{w}=(w_1,\ldots, w_N)^T\in
   [0,1]^N$, consider
\begin{equation}\label{scalar-product-no-sincro}
\begin{split}
\langle \mathbf{F}(\mathbf{w}), \mathbf{w}-\mathbf{z}^*\rangle 
&=
\langle \mathbf{F}(\mathbf{w})-\mathbf{F}(\mathbf{z}^*), 
\mathbf{w}-\mathbf{z}^*\rangle
\\
&=
\frac{\alpha}{N}\left[\sum_{h=1}^N(w_h-z^*_h)\right]^2
+(1-\alpha-\beta)\sum_{h=1}^N\left(f(w_h)-f(z^*_h)\right)(w_h-z^*_h)
-\sum_{h=1}^N (w_h-z^*_h)^2.
\end{split}
\end{equation}
If we choose an index $k$ such that $z^*_k=\tilde{z}_2$ and we
take $w_h=z^*_h$ for all $h\neq k$ and $w_k=\tilde{z}_2+\epsilon$, with
$\epsilon\neq 0$, the above scalar product
\eqref{scalar-product-no-sincro} can be written as
\begin{equation*}
-\left(1-\frac{\alpha}{N}\right)\epsilon^2
+(1-\alpha-\beta)f'(\xi)\epsilon^2\,,
\end{equation*}
with $\xi$ a suitable point in the interval with extremes
$\tilde{z}_2$ and $\tilde{z}_2+\epsilon$. Hence, if we take
$\epsilon\neq 0$ sufficiently small so that
$f'(\xi)>(1-\frac{\alpha}{N})/(1-\alpha-\beta)$, the above quantity is
strictly positive. This fact implies that $\mathbf{z}^*$ is linearly
unstable (see Appendix~\ref{app-sto-approx}). A similar argument shows
that if $N_2=0, \tilde{z}_1=x^*_1$ or $N_2=0, \tilde{z}_2=x^*_2$, the
point $\mathbf{z}^*$ is linearly unstable.\\ \indent Now, let us
consider a zero point $\mathbf{z}^*$ of the form~\eqref{type-1} with
$N_2=0$.  If $f'(\tilde{z}_1)=f'(\tilde{z}_3)$ we can apply Corollary
\ref{cor-autovalori-sincro} and conclude that $\mathbf{z}^*$ is stable
if and only if $(1-\alpha-\beta)f'(\tilde{z}_1)\leq 1-\alpha$. Since
$2(1-\alpha)\leq 1+(1-\alpha)$, this last condition implies
\begin{equation}\label{eq-intermedia}
(1-\alpha-\beta)[f'(\tilde{z}_1)+f'(\tilde{z}_3)]<1+(1-\alpha).
\end{equation}
If $f'(\tilde{z}_1)\neq f'(\tilde{z}_3)$, Corollary
\ref{cor-autovalori-no-sincro} provides conditions for the stability
of $\mathbf{z}^*$ and, by Remark~\ref{cor-autovalori-no-sincro-rem}, a
necessary condition for the stability of $\mathbf{z}^*$ is given by
\eqref{negative-nec}, that is
$$\alpha\frac{N_i}{N}< 1-(1-\alpha-\beta)f'(\tilde{z}_i)\qquad\forall
i=1,3.$$ Since $N_3=N-N_1$, we find
\begin{equation}\label{cond-nec-proof}
-(1-\alpha)+(1-\alpha-\beta)f'(\tilde{z}_3)
<
\alpha\frac{N_1}{N}
<
1-(1-\alpha-\beta)f'(\tilde{z}_1).
\end{equation}
Note that the above inequalities implies condition
\eqref{eq-intermedia} again. Moreover, since $f'$ is strictly
increasing on $[0,x^*)$ and strictly decreasing on $(x^*,1]$ and
$\tilde{z}_1<x^*<\tilde{z}_3$, condition~\eqref{eq-intermedia}
necessarily implies 
\begin{equation}\label{cond-existence-no-sincro}
(1-\alpha-\beta)\left[f'(0)+f'(1)\right]<1+(1-\alpha).
\end{equation}
Summing up, under the assumptions of the considered theorem, if
condition~\eqref{cond-existence-no-sincro} is not satisfied (that is
\eqref{cond-CS} is satisfied), then we have the almost sure asymptotic 
synchronization of the system.  Otherwise, if
\eqref{cond-existence-no-sincro} is satisfied, then we always have a
strictly positive probability that $\mathbf{Z}_\infty$ is equal to a
synchronization zero point of $\mathbf{F}$, but we may also have a
strictly positive probability that it is equal to a no-synchronization
zero point of $\mathbf{F}$ of the form~\eqref{final-no-sincro-points}
(note that $\tilde{z}_{\infty,1}=\tilde{z}_1$ and
$\tilde{z}_{\infty,2}=\tilde{z}_3$).
\end{proof}

\begin{remark}\label{remark-restriction} 
({\em Restrictions on the possible values for $N_1$})\\ \rm Suppose to
  be under the same assumptions of Theorem~\ref{th-LogP-nosincro}. It
  could be useful to observe that, as seen in the above proof, when
  $f'(\tilde{z}_1)\vee f'(\tilde{z}_3)\geq
  (1-\alpha)/(1-\alpha-\beta)$ (and so $f'(\tilde{z}_1)\neq
  f'(\tilde{z}_3)$), relation~\eqref{cond-nec-proof} may provide a
  restriction on the possible values for $N_1$. Note that the two
  bounds depend on the values $\tilde{z}_i$, $i=1,3$. However, when
  $f'(0)\vee f'(1)\geq (1-\alpha)/(1-\alpha-\beta)$, recalling that
  $f'$ is strictly increasing on $[0,x^*)$ and strictly decreasing on
    $(x^*,1]$ and $\tilde{z}_1<x^*<\tilde{z}_3$, we obtain
\begin{equation*}
-(1-\alpha)+(1-\alpha-\beta)f'(1)
<
\alpha\frac{N_1}{N}
<
1-(1-\alpha-\beta)f'(0),
\end{equation*}
that may provide two bounds not depending on the values of the
component of the limit point.
\end{remark}

In the following remark we discuss the possible asymptotic
synchronization of the system toward the values $1/2$.

\begin{remark}\label{remark-1mezzo} 
({\em Possible asymptotic synchronization toward $1/2$})\\ \rm As
  already said, the almost sure asymptotic synchronization toward the
  value $1/2$ means that in the limit the two inclinations (in the
  first interpretation) or the two strategies in all the games (in the
  second interpretation) coexist in the proportion $1:1$. With
  $f=f_{LogP}$, the point $1/2\mathbf{1}$ is a synchronization zero
  point if and only if we have
\begin{equation}\label{value-1mezzo}
f(1/2)+\frac{2\beta q - (1-\alpha)}{2(1-\alpha-\beta)}=0,
\end{equation}
that, since $f$ takes values in $(0,1)$, implies
$(1-\alpha)>2\beta[q\vee (1-q)]$. Moreover, by Lemma
\ref{lemma-synchro-points}, if $(1/2) \mathbf{1}$ is a zero point of
$\mathbf{F}$, then it is stable (and so a possible limit point for the
system) if and only if one of the conditions U1) or U2) is satisfied
or when $1/2$ belongs to $(0,\widehat{x}_1]\cup[\widehat{x}_2,1)$
  (note that U3) is included in this last condition).
\end{remark}

In the next two remarks, we discuss some of the conditions introduced
in the above results, providing simple conditions on $x^*,\,\alpha$
and $\beta$ sufficient to guarantee or to exclude them.

\begin{remark}\label{remark-NO-S2-S4-U2} 
({\em Regarding conditions S2), S4) and U2)})\\ \rm In this remark we show
  that if $\alpha$, $\beta$ and $x^*$ satisfy a particular condition
  (see~\eqref{simple} below), the above cases S2) and S4) are not
  possible.  Indeed, taking $f=f_{LogP}$, we have
  $|x-x^*|f'(x)=g(\theta |x-x^*|)$, where
\begin{equation}\label{def-g}
g(x):=\frac{x\exp(x)}{(1+\exp(x))^2}=
\frac{x\exp(-x)}{(1+\exp(-x))^2}.
\end{equation}
Therefore, observing that $\max_{[0,+\infty[} g<1/4$ (see Fig.~\ref{graph:g}), we get that the
    condition
\begin{equation}\label{simple}
\min\{x^*,(1-x^*)\}\geq \frac{(1-\alpha-\beta)}{4(1-\alpha)}
\end{equation} 
implies $f'(0)\vee f'(1)< (1-\alpha)/(1-\alpha-\beta)$. Hence, if
\eqref{simple} holds true, then S2) and~\eqref{cond-CS} (and so S4))
are not possible. Furthermore, under~\eqref{simple}, case U2) of Lemma
\eqref{lemma-synchro-points} is also not possible.\\ \indent Note that
the above condition~\eqref{simple} is verified when $x^*=1/2$.
\end{remark}

\begin{figure}[h!]
\centering
\includegraphics[scale=0.6,keepaspectratio=true]{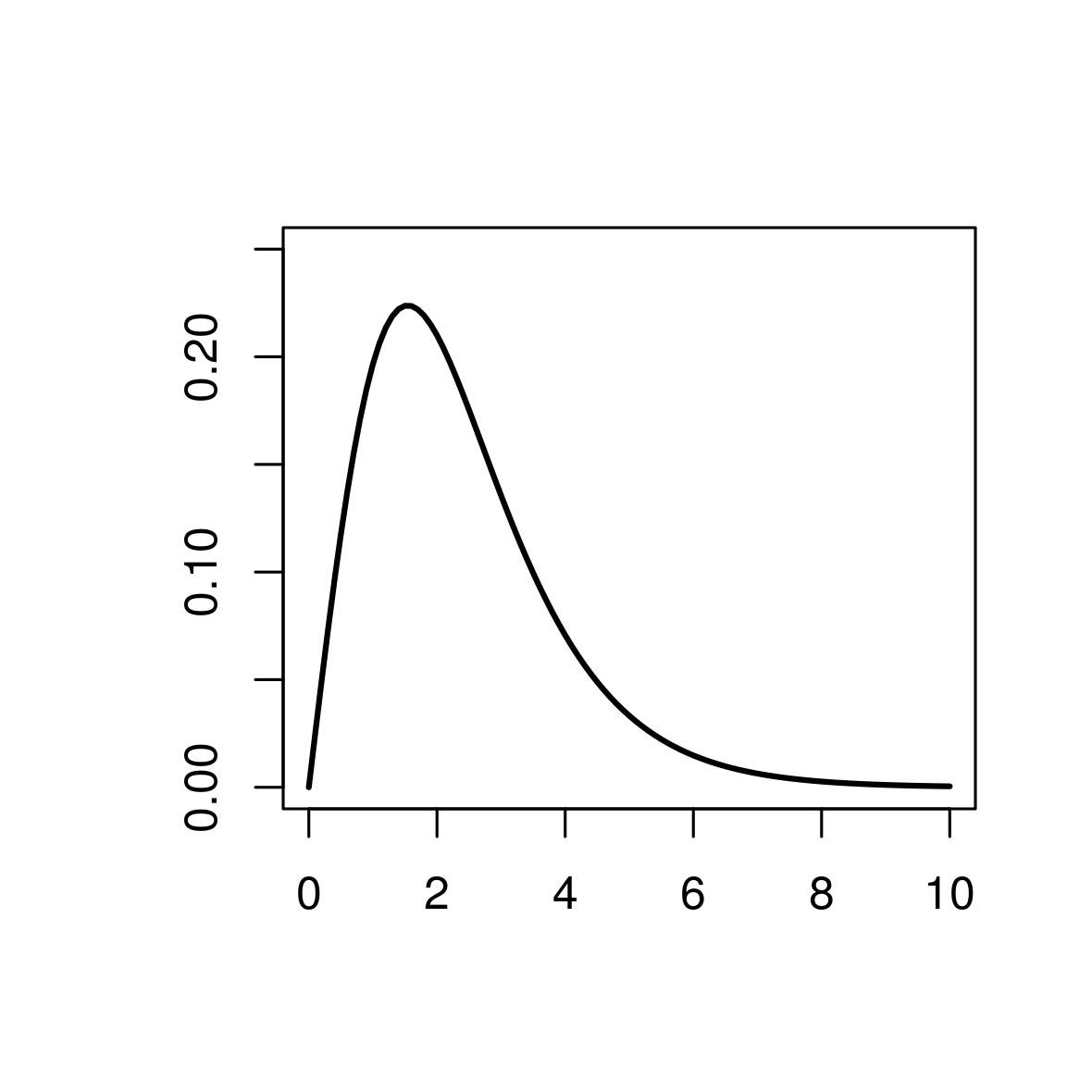}
\caption{Graph of the function~$g$.}
\label{graph:g}
\end{figure}

\afterpage{\clearpage}

\begin{remark}\label{remark-OK-complementary} 
({\em Regarding conditions S3) and~\eqref{complementary}})\\ 
\rm Take $f=f_{LogP}$ and suppose to be in the case $f'(0)\vee
f'(1)<\frac{1}{(1-\alpha-\beta)}<\theta/4$ and let $x_1^*<x^*<x_2^*$
such that $f'(x_i^*)=1/(1-\alpha-\beta)$. If $x^*$ belongs to the
interval
$$\left(\frac{1}{2}-\frac{(\alpha+\beta)}{2},\,
\frac{1}{2}+\frac{(\alpha+\beta)}{2}\right)
$$ (for instance, this is the case when $x^*=1/2$), then we
necessarily have $f(x_1^*)<x_1^*/(1-\alpha-\beta)$ and
$f(x_2^*)>x_2^*/(1-\alpha-\beta)-(\alpha+\beta)/(1-\alpha-\beta)$.
Indeed, when $x^*$ belongs to the above interval, then
$f(x^*)-x^*/(1-\alpha-\beta)=1/2-x^*/(1-\alpha-\beta)$ belongs to
$\left(-(\alpha+\beta)/(1-\alpha-\beta),\, 0\right)$ and so, since the
function $f-(1-\alpha-\beta)^{-1}id$ is strictly increasing on
$(x_1^*,x_2^*)$, we get the two desired inequalities for
$f(x_i^*)-x_i^*/(1-\alpha-\beta)$, $i=1,\,2$. As a consequence, case
S3) is not possible.
\end{remark}

As a consequence of the above results and remarks, we obtain the
following corollary, that deals with the special case $x^*=1/2$ and
either $\beta=0$ or $q=1/2$. See Fig.~\ref{f4-setparam3:figures}.

\begin{corollary}

{\em (Special case: $x^*=1/2$ and either $\beta=0$ or $q=1/2$)}

\label{cor-special-case} 

 Take $f=f_{LogP}$ with $x^*=1/2$ and
  suppose that one of the conditions $\beta=0$ or $q=1/2$ is
  satisfied. Assume $\mathcal{Z}(\mathbf{F})$ is finite. Then, using
  the same notation as in Lemma~\eqref{lemma-synchro-points} and
  Theorem~\ref{th-LogP-nosincro}, only the following cases are
  possible:
\begin{itemize}
\item[a)] $\theta/4\leq (1-\alpha)/(1-\alpha-\beta)$ and, if this is
  the case, the system almost surely asymptotically synchronizes and
  it is predictable, and the unique limit point is $x^*=1/2$;
\item[b)] $(1-\alpha)/(1-\alpha-\beta)<\theta/4\leq
  1/(1-\alpha-\beta)$ and, if this is the case, the system almost
  surely synchronizes, but there are two possible limit points,
  $\mathbf{z}_{i}^*=z_{i}^*\mathbf{1}$, $i=1,\,2$, with
  $0<z_{1}^*<\widehat{x}_1<1/2<\widehat{x}_2<z_{2}^*=1-z_{1}^*<1$;
\item[c)] $\theta/4>1/(1-\alpha-\beta)$ and, if this the case, the
  system almost surely converges to a random variable
  $\mathbf{Z}_\infty$, taking values in the set of the stable zero
  points of $\mathbf{F}$, which is contained in $(0,1)^N$. Such set
  always contains two stable synchronization zero points,
  $\mathbf{z}_{i}^*=z_{i}^*\mathbf{1}$, $i=1,\,2$, with
  $0<z_{1}^*<\widehat{x}_1<1/2<\widehat{x}_2<z_{2}^*=1-z_{1}^*<1$, and
  it may contain also no-syncronization zero points of the form
 ~\eqref{final-no-sincro-points}. In particular, when
\begin{equation}\label{cond-suf-stability}
x_1^*\leq x^*-
\frac{1-\alpha-\beta}{4(1-\alpha)}=
\frac{1}{2}-\frac{1-\alpha-\beta}{4(1-\alpha)},
\end{equation}
the points of the form~\eqref{final-no-sincro-points} with
$0<\tilde{z}_{\infty,1}<x_1^*<1/2$,
$1/2<x_2^*<\tilde{z}_{\infty,2}=1-\tilde{z}_{\infty,1}<1$, $N_1=N/2$
and
$(1-\alpha-\beta)f(\tilde{z}_{\infty,1})-\tilde{z}_{\infty,1}=-(\alpha+\beta)/2$,
are stable no-synchronization zero points of $\mathbf{F}$.
\end{itemize}
\end{corollary}
(Note that, for $\beta=0$, the above condition
\eqref{cond-suf-stability} simply becomes $x_1^*\leq 1/4$.)
\\

\begin{proof} By Remark
\ref{remark-x^*}, when $x^*=1/2$ and one of the conditions $\beta=0$
or $q=1/2$ is satisfied, then $(1/2)\mathbf{1}$ is a synchronization
zero point of $\mathbf{F}$. Moreover, it is stable if and only if U1)
is satisfied and, if this is the case, then the system almost surely
asymptotically synchronizes and it is predictable. If U1) is not
satisfied, we have two stable synchronization zero points, whose
components are symmetric with respect to $x^*=1/2$, that is
$\mathbf{z}_{i}^*=z_{i}^*\mathbf{1}$, $i=1,\,2$, with
$0<z_{1}^*<\widehat{x}_1<1/2<\widehat{x}_2<z_{2}^*=1-z_{1}^*<1$. Moreover,
by Remarks~\ref{remark-NO-S2-S4-U2} and~\ref{remark-OK-complementary},
cases S2), S3) and S4) are not possible (because $x^*=1/2$ implies
$f'(0)\vee f'(1)< (1-\alpha)/(1-\alpha-\beta)$,
$f(x_1^*)<x_1^*/(1-\alpha-\beta)$ and
$f(x_2^*)>x_2^*/(1-\alpha-\beta)-(\alpha+\beta)/(1-\alpha-\beta)$). Summing
up, we can have only: case S1), in which we have the almost sure asymptotic 
synchronization of the system and, when U1) is satisfied, also its
predictability (see cases a) and b) in the statement); or the case when
\eqref{complementary} is satisfied, but not~\eqref{cond-CS}, and so
the convergence toward no-synchronization zero points of the form
\eqref{final-no-sincro-points} may be possible (see case c) in the
statement). Moreover, we observe that, when we are in this last case,
taking
$c=f(x^*)-x^*/(1-\alpha-\beta)=-(\alpha+\beta)/[2(1-\alpha-\beta)]$,
the zero points of the corresponding function $\widetilde{\varphi}$
are symmetric with respect to $x^*=1/2$ (that is
$\widetilde{\varphi}(z)=0 \Leftrightarrow
\widetilde{\varphi}(1-z)$). It follows that points of the form
\eqref{type-1} with $0<\tilde{z}_1<x_1^*<1/2$, $\tilde{z}_2=x^*=1/2$,
$1/2<x_2^*<\tilde{z}_3=1-\tilde{z}_1<1$ are zero points of
$\mathbf{F}$ if and only if $f(\tilde{z}_1)-z_1/(1-\alpha-\beta)=c$
and~\eqref{eq-z_2tilde} is satisfied, that is
\begin{equation}\label{eq-z_2tilde-special}
\alpha \frac{1}{N}
\left(N_1\tilde{z}_1+N_2/2+N_3\tilde{z}_3\right)
+\beta q+
(1-\alpha-\beta)c=0\,.
\end{equation}
In particular, if we take $N_1=N_3$ and we use the symmetry between
$\tilde{z}_1$ and $\tilde{z}_3$, we obtain that
\eqref{eq-z_2tilde-special} is satisfied if and only if $\beta=0$ or
$q=1/2$. Therefore, when $x^*=1/2$ and one of the conditions $\beta=0$
or $q=1/2$ is satisfied, $\mathbf{F}$ has no-synchronization zero
points of the form~\eqref{type-1} with $0<\tilde{z}_1<x_1^*<1/2$,
$\tilde{z}_2=x^*=1/2$, $1/2<x_2^*<\tilde{z}_3=1-\tilde{z}_1$,
$N_1=N_3$ and $f(\tilde{z}_1)-z_1/(1-\alpha-\beta)=c$. Now, such
points are stable if and only if $N_2=0$ (and so $N_1=N_3=N/2$) and
$(1-\alpha-\beta)f'(\tilde{z}_1)\leq 1-\alpha$ (note that we are in
the case $f'(\tilde{z}_1)=f'(\tilde{z}_3)$).  Since $f'$ is strictly
increasing on $[0,x^*)$, this last condition is satisfied when
  $(1-\alpha-\beta)f'(x_1^*)\leq 1-\alpha$. Finally, using the fact
  that, for $f=f_{LogP}$, we have $f'(x)=g(\theta |x-x^*|)/|x-x^*|$,
  where $g$ is the function defined in~\eqref{def-g} and such that
  $\max_{[0,+\infty[} g<1/4$, in order to guarantee the stability of
      the considered no-synchronization zero points, it is enough to
      require~\eqref{cond-suf-stability}.
\end{proof}

We conclude this section with two remarks: one regarding the case
$N=1$ and the other the rate of convergence.

\begin{remark} \label{remark-N=1}  ({\em Case $N=1$}) \\ 
\rm This remark is devoted to the case $N=1$ and the relationship with
the results obtained in~\cite{Fagiolo}.  \\ \indent The above proofs
(with the due simplifications) also work in the case $N=1$ and
$\alpha=\beta=0$, that corresponds to the case studied in
\cite{Fagiolo}. Indeed, in this case, we have to consider only Theorem
\ref{general-as-conv}, Lemma~\ref{lemma-synchro-points} and Remark
\ref{remark-x^*} (with $N=1$ and $\alpha=\beta=0$). As a consequence,
when $x^*=1/2$, if $\theta\leq 4$, then the system is predictable and
the unique limit configuration is given by $z_\infty=1/2$; otherwise
it almost surely converges, but it is not predictable, and the two
possible limit configurations belong to $(0,\widehat{x}_1]\cup
    [\widehat{x}_2,1)\subset (0,1)\setminus\{1/2\}$ and they are
      symmetric with respect to $1/2$. This is the same result
      obtained in \cite{Fagiolo}. For the case $x^*\neq 1/2$, in
      \cite{Fagiolo} there are only some numerical analyses; while
      here we have proven a precise result: when one of the conditions
      U1), U2) or U3) is satisfied, then the system is predictable and
      the limit configuration belongs to $(0,1)\setminus\{x^*,\,1/2\}$
      (more precisely, it is strictly smaller than $1/2$ when
      $x^*>1/2$ and strictly greater than $1/2$ when $x^*<1/2$,
      because $\varphi(1/2)=f(1/2)-1/2=f(1/2)-f(x^*)$ and $f$ is
      strictly increasing); otherwise it almost surely converges, but
      the system is not predictable, and the two possible limit
      configurations belong to $(0,1)\setminus\{x^*,\,1/2\}$ (more
      precisely, one belongs to
      $(0,\widehat{x}_1]\setminus\{1/2\}\subset
    (0,x^*)\setminus\{1/2\}$ and one in
    $[\widehat{x}_2,1)\setminus\{1/2\}\subset (x^*,1)\setminus\{1/2\}$
      and so as before, taking into account that
      $\varphi(1/2)=f(1/2)-f(x^*)$, one is strictly smaller than $1/2$
      and the other strictly greater than $1/2$). 
\end{remark}

\begin{remark}\label{remark-tlc-LogP} ({\em Rate of convergence})\\
 \rm 
When the system is predictable with
$\mathbf{z}_\infty=z_\infty\mathbf{1}$ as the unique possible limit
value for $\mathbf{Z}_\infty$, applying the same arguments used in
Remark~\ref{remark-tlc-LP}, we can obtain a central limit theorem
where the rate of convergence is driven by
$\lambda=(1-\alpha)-(1-\alpha-\beta)f'(z_\infty)$ (see Theorem
\ref{clt} and Remark~\ref{clt-alternative-expression}).  \\ \indent
When the system almost surely converges to $\mathbf{Z}_\infty$, but it
is not predictable, applying Remark~\ref{clt-more-limit-points}), we
get $1/\sqrt{n}$ as the rate of convergence, for any
$\mathbf{z}_\infty$ with $P(\mathbf{Z}_n\to\mathbf{z}_\infty)>0$ and
$\lambda(\mathbf{z}_\infty)>0$. In particular, with some computations
similar to those done in Remark~\ref{remark-tlc-LP}, we have that the
matrix $\Gamma=\Gamma(\mathbf{z}_\infty)$ is a diagonal matrix with
diagonal elements equal to $z_{\infty,h}(1-z_{\infty,h})$ with
$z_{\infty,h}=z_{\infty}$ (synchronization point) or $z_{\infty,h}\in
\{\tilde{z}_{\infty,1},\tilde{z}_{\infty,2}\}$ (no-synchronization
point). Finally, note that, when $\mathbf{z}_\infty$ is a possible
no-synchronization limit point with
$\tilde{z}_{\infty,2}=2x^*-\tilde{z}_{\infty,1}$, we have
$f'(\tilde{z}_{\infty,1})=f'(\tilde{z}_{\infty,2})$ and so again
$\lambda(\mathbf{z}_\infty)=(1-\alpha)-(1-\alpha-\beta)f'(\tilde{z}_{\infty,1})$.
\end{remark}
\bigskip 


\subsection{Case $f=f_{\textrm{Tech}}$} In this subsection we consider the function 
$f_{\textrm{Tech}}$ defined in~\eqref{f-Tech} with $\theta\in (1/2,1)$, that is
\begin{equation}\label{f-Tech-bis}
f(x)=f_{\textrm{Tech}}(x)=(1-\theta)+(2\theta-1)(3x^2-2x^3)
\qquad\mbox{with } \theta\in (1/2,1).
\end{equation}

Similarly to $f_{LogP}$, the function $f=f_{\textrm{Tech}}$ is a sigmoid
function, \textit{i.e.}~its first derivative is a strictly positive function,
which is strictly increasing on $[0, 1/2)$ and strictly decreasing on
  $(1/2, 1]$ with a maximum given by $f'(1/2)=3\theta-3/2$.
Furthermore, we have $f'(x)=f'(1-x)$ for all $x\in [0,1]$. Differently
from $f_{LogP}$, we have $f'(0)=f'(1)=0$. Therefore, arguing exactly
as in the proof of Lemma~\ref{lemma-synchro-points} and Remark
\ref{remark-x^*}, but using the fact that $x^*=1/2$ and $f'(0)=f'(1)$,
we obtain the following lemma and remark.

\begin{lemma}\label{lemma-synchro-points-fTech} (Synchronization zero points)\\
Let $f=f_{\textrm{Tech}}$. Then, accordingly to the values of the parameters,
$\mathcal{Z}(\mathbf{F})$ contains at least three synchronization zero
points. Moreover, at most two of them are stable.  In particular, if
one of the following conditions is satisfied, $\mathbf{F}$ has a
unique stable synchronization zero point:
\begin{itemize}
\item[U1)] $3\theta-3/2\leq (1-\alpha)/(1-\alpha-\beta)$ or  
\item[U2)] $(1-\alpha)/(1-\alpha-\beta) < 3\theta-3/2$
  and either
  $f(\widehat{x}_1)>(1-\alpha)/(1-\alpha-\beta)\widehat{x}_1-\beta
  q/(1-\alpha-\beta)$ or
  $f(\widehat{x}_2)<(1-\alpha)/(1-\alpha-\beta)\hat{x}_2-\beta
  q/(1-\alpha-\beta)$, where $\widehat{x}_1\in (0,1/2)$ and
  $\widehat{x}_2=1-\widehat{x}_1\in (1/2,1)$ are the solutions of
  $f'=(1-\alpha)/(1-\alpha-\beta)$.
\end{itemize}
Otherwise, $\mathbf{F}$ has two stable synchronization zero points
belonging to $(0,\widehat{x}_1]\cup [\widehat{x}_2,1)$ (more
      precisely, one in each of these two intervals).
\end{lemma}

\begin{remark}\label{remark-x^*-fTech}
\rm Note that if $1/2\mathbf{1}$ is a synchronization zero point of
$\mathbf{F}$, that is $\beta=0$ or $q=1/2$, then U2) is not possible. Indeed,
$f-(1-\alpha)id/(1-\alpha-\beta)$ is strictly increasing on
$(\widehat{x}_1,\widehat{x}_2)$ and so we have
$f(\widehat{x}_1)-(1-\alpha)\widehat{x}_1/(1-\alpha-\beta)
<f(1/2)-(1-\alpha)/2(1-\alpha-\beta)= -\beta
q/(1-\alpha-\beta)<f(\widehat{x}_2)-(1-\alpha)\widehat{x}_2/(1-\alpha-\beta)$. 
Moreover, when $1/2\mathbf{1}$ is a
synchronization zero point of $\mathbf{F}$, it is stable if and only
if U1) is satisfied. Otherwise, there are three synchronization zero
points: $1/2$ (linearly unstable) and two stable, say
$\mathbf{z}_1^*=z_1^*\mathbf{1}$ and $\mathbf{z}_2^*=z_2^*\mathbf{1}$,
with $0<z_1^*<\widehat{x}_1<1/2<\widehat{x}_2<z_2^*=1-z_1^*<1$.
\end{remark}

As an immediate consequence of Lemma~\ref{lemma-synchro-points-fTech},
we get that, if the system almost surely asymptotically synchronizes
and one of the conditions U1) or U2) holds true, then it is
predictable.\\

Now, we observe that $f=f_{\textrm{Tech}}$ admits the primitive function
$$ 
\phi(x)=(1-\theta)x+(2\theta-1)\left(1-\frac{x}{2}\right)x^3+const.
$$ Moreover, for $\theta\in (1/2,1)$, we have $f(0)=1-\theta>0$ and
$f(1)=\theta<1$. Therefore, arguing exactly as done in the proof of
Theorem\ref{th-LogP-sincro}, but with some simplifications due to the
fact that $x^*=1/2$ (see Remark~\ref{remark-OK-complementary}) and
$f'(0)=f'(1)=0$, we obtain the following result.

See in section~\ref{Sec-simulations}, simulations and illustrations associated in Fig.~\ref{ftech-setparam4:figures}, Fig.~\ref{ftech-setparam6:figures}, Fig.~\ref{ftech-setparam2:figures}.

\begin{theorem}\label{th-Tech-sincro}
Let $f=f_{\textrm{Tech}}$.  If $3\theta-3/2\leq 1/(1-\alpha-\beta)$, then we
have the almost sure asymptotic synchronization of the system, \textit{i.e.}
$$
\mathbf{Z}_n\stackrel{a.s.}\longrightarrow\mathbf{Z}_\infty
$$ 
and 
$$
\overline{\mathbf{I}}_n=\frac{1}{n}\sum_{k=1}^n\mathbf{I}_k
\stackrel{a.s.}\longrightarrow \mathbf{Z}_\infty,
$$ where $\mathbf{Z}_\infty$ is a random variable of the form
$\mathbf{Z}_\infty=Z_\infty\mathbf{1}$. Moreover, the random variable
$\mathbf{Z}_\infty$ takes values in the set of the stable zero points
of $\mathbf{F}$, which is contained in $(0, 1)^N$ and consists of at
most two different points.
\end{theorem}

Next theorem deal with the case not covered by Theorem
\ref{th-Tech-sincro}. In particular, we characterize the possible
``no-synchronization limit configurations'' for the system. The proof
is exactly the same given for Theorem~\ref{th-LogP-nosincro}, but
taking into account that $x^*=1/2$ and $f'(0)=f'(1)=0$ and, above all,
that $\mathcal{Z}(\mathbf{F})$ is finite
(see Lemma~\ref{lemma-Tech-zeri-finiti} in Appendix).

\begin{theorem}\label{th-Tech-nosincro}
Let $f=f_{\textrm{Tech}}$. If $1/(1-\alpha-\beta)<3\theta-3/2$, then 
$$
\mathbf{Z}_n\stackrel{a.s.}\longrightarrow\mathbf{Z}_\infty
$$ 
and 
$$
\overline{\mathbf{I}}_n=\frac{1}{n}\sum_{k=1}^n\mathbf{I}_k
\stackrel{a.s.}\longrightarrow \mathbf{Z}_\infty,
$$ where $\mathbf{Z}_\infty$ takes values in the set $\mathcal{SZ}(\mathbf{F})$
of the stable
zero points of $\mathbf{F}$, which is contained in $(0,1)^N$. Such set
always contains at most two synchronization zero points. Moreover, any 
$\mathbf{z}_\infty\in \mathcal{SZ}(\mathbf{F})$ which is not a
synchronization point, has the form, up to permutations,
$\mathbf{z}_\infty=(z_{\infty,1}, \ldots, z_{\infty,N})^T$ with
\begin{equation}\label{final-no-sincro-points-fTech}
z_{\infty,h}=\begin{cases}
\tilde{z}_{\infty,1}\in (0, x^*_1) \qquad &\mbox{for } h=1,\dots, N_1\\
\tilde{z}_{\infty,2}\in (x^*_2, 1)\qquad &\mbox{for } h=N_1+1,\dots, N,
\end{cases}
\end{equation} 
where $x^*_1\in (0,1/2)$ and $x_2^*=1-x^*_1\in (1/2,1)$ are the
solutions of $f'=1/(1-\alpha-\beta)$, and $N_1\in\{1,\dots, N-1\}$.
\end{theorem}

Since $f_{\textrm{Tech}}$ is a polynom of degree~$3$, thanks
 to Lemma~\ref{lemma-Tech-zeri-finiti}, 
 it holds $\mathcal Z(\mathbf{F})$, set of roots, is finite.
  Summing up, taking $f=f_{\textrm{Tech}}$, if $3\theta-3/2\leq
1/(1-\alpha-\beta)$, then, we have the almost sure asymptotic
synchronization of the system. Otherwise, we may have a non-zero
probability that the system does not synchronize as time goes to
$+\infty$. More precisely, the system almost surely converges, and we
have a non-zero probability of asymptotic synchronization, but we may
also have a non-zero probability of observing the system splitting
into two groups of components that converge towards two different
values.  We also point out that the above results state that the
random limit $\mathbf{Z}_\infty$ always belongs to $(0,1)^N$. In the
first interpretation, this fact means that in the limit configuration
the $N$ agents always keep a strictly positive inclination for both
actions; while in the interpretation regarding the technological
dynamics, this fact means that in the limit configuration, both
technologies coexist in all the $N$ markets.  Moreover, regarding the
possible ``no-synchronization limit configurations'' we have that,
independently from the value of $N$, we always have at most two groups
of agents (or markets) that approach in the limit two different
values. We never have a more complicated asymptotic fragmentation of
the whole system. Furthermore, we are able to localize the two limit
values: one is strictly smaller than $x_1^*<1/2$ and the other
strictly bigger than $x_2^*>1/2$, where the points $x_i^*$ only depend
on $\theta$ and on $(1-\alpha-\beta)$, which is the ``weight'' of the
personal inclination component of $P_{n,h}$ in
\eqref{prob-model-intro}. Moreover, arguing as in Remark
\ref{remark-restriction}, the following inequality may provide restrictions
on the possible values for $N_1$:
\begin{equation*}
-(1-\alpha)+(1-\alpha-\beta)f'(\tilde{z}_{\infty,2})
<
\alpha\frac{N_1}{N}
<
1-(1-\alpha-\beta)f'(\tilde{z}_{\infty,1}).
\end{equation*}

In the following remark we discuss the possible asymptotic
synchronization of the system toward the values $1/2$.

\begin{remark}\label{remark-1mezzo-fTech} 
({\em Possible asymptotic synchronization toward $1/2$})\\ \rm As
  already said, the almost sure asymptotic synchronization toward the
  value $1/2$ means that in the limit the two inclinations (in the
  first interpretation) or the two technologies in all the markets (in
  the third interpretation) coexist in the proportion $1:1$. With
  $f=f_{\textrm{Tech}}$, the point $1/2\mathbf{1}$ is a synchronization zero
  point if and only if $\beta=0$ or $q=1/2$. Moreover, by Remark
 ~\ref{remark-x^*-fTech}, if $(1/2) \mathbf{1}$ is a zero point of
  $\mathbf{F}$, then it is stable (and so a possible limit point for
  the system) if and only if U1) is satisfied.
\end{remark}

As a consequence of the above results and remarks, arguing as in the
proof of Corollary~\ref{cor-special-case}, we obtain the following
corollary, that deals with the special case $\beta=0$ or $q=1/2$.

\begin{corollary}{\em (Special case: $\beta=0$ or $q=1/2$)}\\ 
Take $f=f_{\textrm{Tech}}$ and suppose that one of the conditions $\beta=0$ or
$q=1/2$ is satisfied. 
Then, using the same notation as in Lemma
\eqref{lemma-synchro-points-fTech} and Theorem~\ref{th-Tech-nosincro},
only the following cases are possible:
\begin{itemize}
\item[a)] $3\theta-3/2\leq (1-\alpha)/(1-\alpha-\beta)$ and, if this
  is the case, the system almost surely asymptotically synchronizes
  and it is predictable, and the unique limit point is $x^*=1/2$;
\item[b)] $(1-\alpha)/(1-\alpha-\beta)<3\theta-3/2\leq
  1/(1-\alpha-\beta)$ and, if this is the case, the system almost
  surely synchronizes, but there are two possible limit points,
  $\mathbf{z}_{i}^*=z_{i}^*\mathbf{1}$, $i=1,\,2$, with
  $0<z_{1}^*<\widehat{x}_1<1/2<\widehat{x}_2<z_{2}^*=1-z_{1}^*<1$;
\item[c)] $3\theta-3/2>1/(1-\alpha-\beta)$ and, if this the case, the
  system almost surely converges to a random variable
  $\mathbf{Z}_\infty$, taking values in the set of the stable zero
  points of $\mathbf{F}$, which is contained in $(0,1)^N$. Such set
  always contains two stable synchronization zero points,
  $\mathbf{z}_{i}^*=z_{i}^*\mathbf{1}$, $i=1,\,2$, with
  $0<z_{1}^*<\widehat{x}_1<1/2<\widehat{x}_2<z_{2}^*=1-z_{1}^*<1$, and
  it may contain also no-syncronization zero points of the form
 ~\eqref{final-no-sincro-points-fTech}. In particular, when 
$$(1-\alpha-\beta)f'(x_1^*)\leq 1-\alpha,$$ the points of the form
 ~\eqref{final-no-sincro-points-fTech} with
  $0<\tilde{z}_{\infty,1}<x_1^*<1/2$,
  $1/2<x_2^*<\tilde{z}_{\infty,2}=1-\tilde{z}_{\infty,1}<1$, $N_1=N/2$
  and
  $(1-\alpha-\beta)f(\tilde{z}_{\infty,1})-\tilde{z}_{\infty,1}=-(\alpha+\beta)/2$,
  are stable no-synchronization zero points of $\mathbf{F}$
\end{itemize}
\end{corollary}

We conclude this section with two remarks: one regarding the case
$N=1$ and the other the rate of convergence.

\begin{remark} \label{remark-N=1-fTech}  ({\em Case $N=1$}) \\ 
\rm This remark is devoted to the case $N=1$ and the relationship with
the results obtained in~\cite{arthur-1983, dosi}. Indeed, we have to
consider only Theorem~\ref{general-as-conv}, Lemma
\ref{lemma-synchro-points-fTech} and Remark~\ref{remark-x^*-fTech}
(with $N=1$ and $\alpha=\beta=0$, that corresponds to the case studied
in~\cite{arthur-1983, dosi}). As a consequence, when $1/2<\theta\leq
5/6$ the system is predictable and the unique limit configuration is
$1/2$; otherwise it almost surely converges, but it is not
predictable, and the two possible limit configurations $z_1^*,\,z_2^*$
belong to $(0,1)\setminus\{1/2\}$ and they are such that
$0<z^*_1<\widehat{x}_1<1/2$ and
$1/2<\widehat{x}_2=1-\widehat{x}_1<z_2^*=1-z_1^*<1$.
\end{remark}

\begin{remark}\label{remark-tlc-fTech} ({\em Rate of convergence})\\ \rm 
When the system is predictable with
$\mathbf{z}_\infty=z_\infty\mathbf{1}$ as the unique possible limit
value for $\mathbf{Z}_\infty$, applying the same arguments used in
Remark~\ref{remark-tlc-LP}, we can obtain a central limit theorem
where the rate of convergence is driven by
$\lambda=(1-\alpha)-(1-\alpha-\beta)f'(z_\infty)$ (see Theorem
\ref{clt} and Remark~\ref{clt-alternative-expression}).  \\ \indent
When the system almost surely converges to $\mathbf{Z}_\infty$, but it
is not predictable, applying the same arguments used in Remark
\ref{remark-tlc-LogP}, we get $1/\sqrt{n}$ as the rate of convergence,
for any $\mathbf{z}_\infty$ with
$P(\mathbf{Z}_n\to\mathbf{z}_\infty)>0$ and
$\lambda(\mathbf{z}_\infty)>0$ (see Remark
\ref{clt-more-limit-points}). Note that, when $\mathbf{z}_\infty$ is a
possible no-synchronization limit point with
$\tilde{z}_{\infty,2}=1-\tilde{z}_{\infty,1}$, we have
$f'(\tilde{z}_{\infty,1})=f'(\tilde{z}_{\infty,2})$ and so again 
$\lambda(\mathbf{z}_\infty)=(1-\alpha)-(1-\alpha-\beta)f'(\tilde{z}_{\infty,1})$.
\end{remark}

\section{Simulations and figures}\label{Sec-simulations}

In the following section, we do consider some graphical illustrations and numerical simulations or sampling of the stochastic dynamical systems. This can be easily coded thanks to the iterative equations defining the dynamical evolution. 

We have chosen particular parameter sets for each specific~$f$ considered previously. The sets were chosen for their own interest or for their interest in comparison with other sets.
We used different values for $N$.
We considered either deterministic initial conditions or random ones.
When random, we chose independent values, uniformly distributed on [0,1]. 
Note that, when we assume as initial condition
       $(Z_{0,h},\ldots,Z_{N,h})$ exchangeable,
the variables $(Z_{n,1},\ldots, Z_{n,N})$ are exchangeable for all
$n$, and so the set $\mathcal{Z}$ where $\mathbf{Z}_\infty$ takes
values is permutation invariant, \textit{i.e.} if $(z_1, \ldots ,
z_N)^{\top}$ belongs to $\mathcal{Z}$, then, for any permutation
$\sigma$ of $\{1,\ldots, N\}$, the vector $(z_{\sigma(1)}, \ldots ,
z_{\sigma(N)})^{\top}$ also belongs to $\mathcal{Z}$.

As previously noticed, when $\alpha=0$, there is no interaction \textit{i.e.} stochastic independence between the components holds. Since $f$ is non linear, contrary to the models where $f$ is linear, 
combinatorics can create multiple  limit points. For $f=f_{LP}$ only synchronization limit points are possible. For $f=f_{LogP}$ or $f=f_{\textrm{Tech}}$, according to the choice of parameters, many limit points are possible which can be of synchronization (on the diagonal) or of no-synchronization type (off the diagonal). When $f$ is linear synchronization was proved to hold as soon as there is interaction ($\alpha>0$). Here, on the contrary, it needs specific conditions between the parameters to be fulfilled in order to have synchronization almost surely. Finally, as observed in the following samplings, in some region of parameters, limit points may be difficult to observe computationally due to slow dynamical evolution.

\subsection{Case $f=f_{LP}$}

For the set of parameters $\theta=0.9$, $x^*=1/3$, $\alpha=0.1$, $\beta=0.2$, $q=0.4$,  
Fig.~\ref{fig:f3:graph} shows the 
graph of the function~$f_{LP}$ intersecting the straight line defined through~\eqref{system-sincro}. 
 There is 
a unique zero synchronisation point at~$\approx 0.664$. 
 
For the same set of parameters 
Fig.~\ref{fig:f3} shows some numerical simulations samples.
Fig.~\ref{fig:f3}~(A) presents the trajectories of the components values of one sample of the whole system, when $N=30$.
The associated empirical means are represented in Fig.~\ref{fig:f3}~(B).
In both cases, a.s. convergence towards the unique (stable) synchronization point is observed in coherence with the previously stated theoretical result. 
Fig.~\ref{fig:f3}~(C) pictures through an histogram the values observed for a large time ($T=5000$). Remark these values are components' values of $100$ independent samples of the system.
Fig.~\ref{fig:f3}~(D) is a representation when $N=2$ of the tangent/gradient  field 
of~$-V$. Additively values of $-V$ are represented through colors. 
Blue color  is used for low values, showing a unique minimum of~$-V$.
Red color is used for high values.


\begin{figure}[p]
\centering
\includegraphics[scale=0.45,keepaspectratio=true]{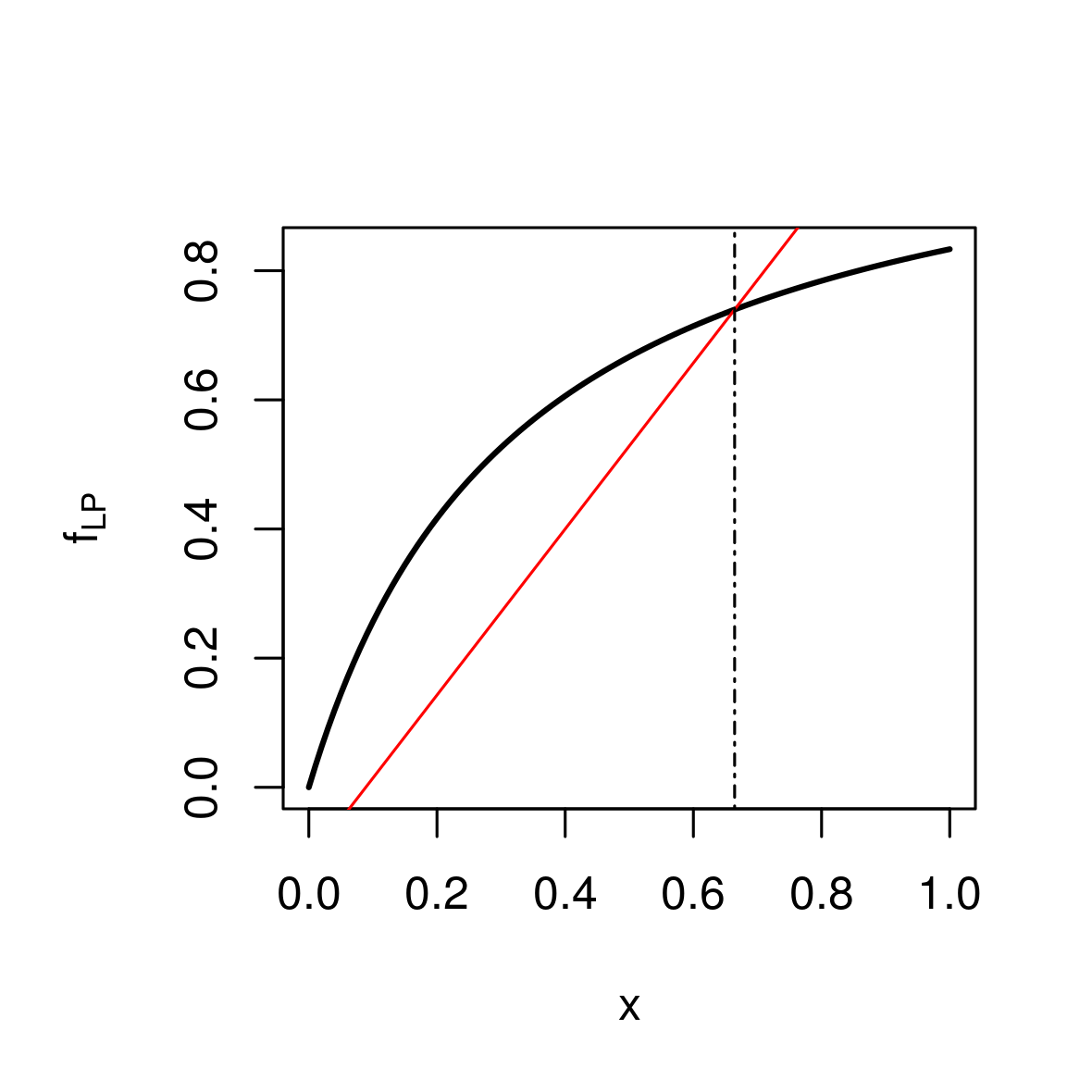}
\caption{Case $f=f_{LP}$. Graph of the function~$f_{LP}$ intersecting the straight line $y= ((1-\alpha) x -\beta q)/(1-\alpha -\beta)$, giving a unique zero synchronisation point at $\approx 0.664$ . Set of parameters  $\theta=0.9$, $x^*=1/3$, $\alpha=0.1$, $\beta=0.2$, $q=0.4$. } \label{fig:f3:graph}
\end{figure}

\begin{figure}[p]
\centering
\begin{subfigure}[b]{0.48\linewidth}
\centering
\includegraphics[scale=0.4,keepaspectratio=true]{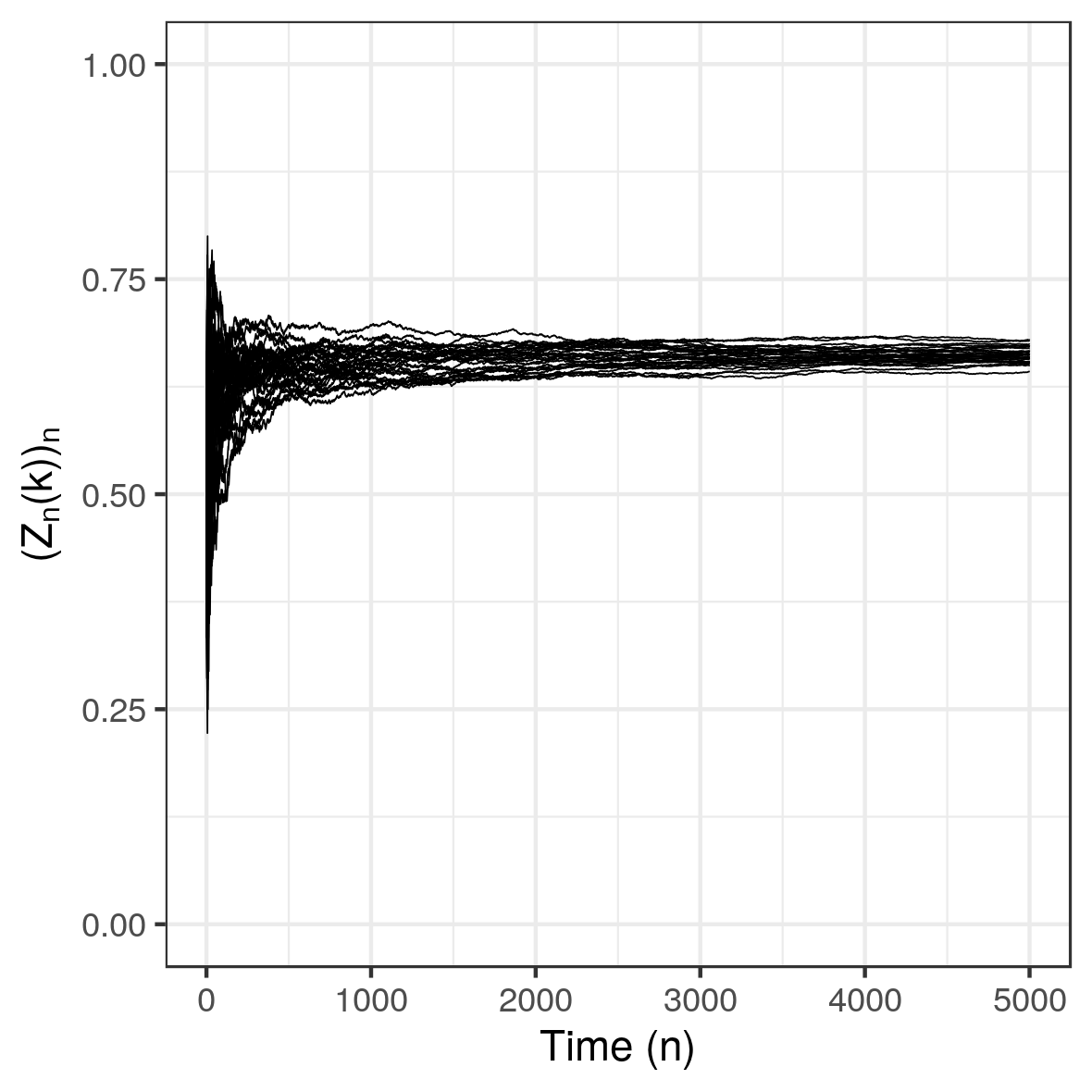}
\caption{One sample of the trajectories $(Z_n(k))_n$ ($1\leq k \leq N$) of a system with $N=30$. Starting condition is 1/2 for all components.}
\end{subfigure}
\begin{subfigure}[b]{0.48\linewidth}
\centering
\includegraphics[scale=0.4,keepaspectratio=true]{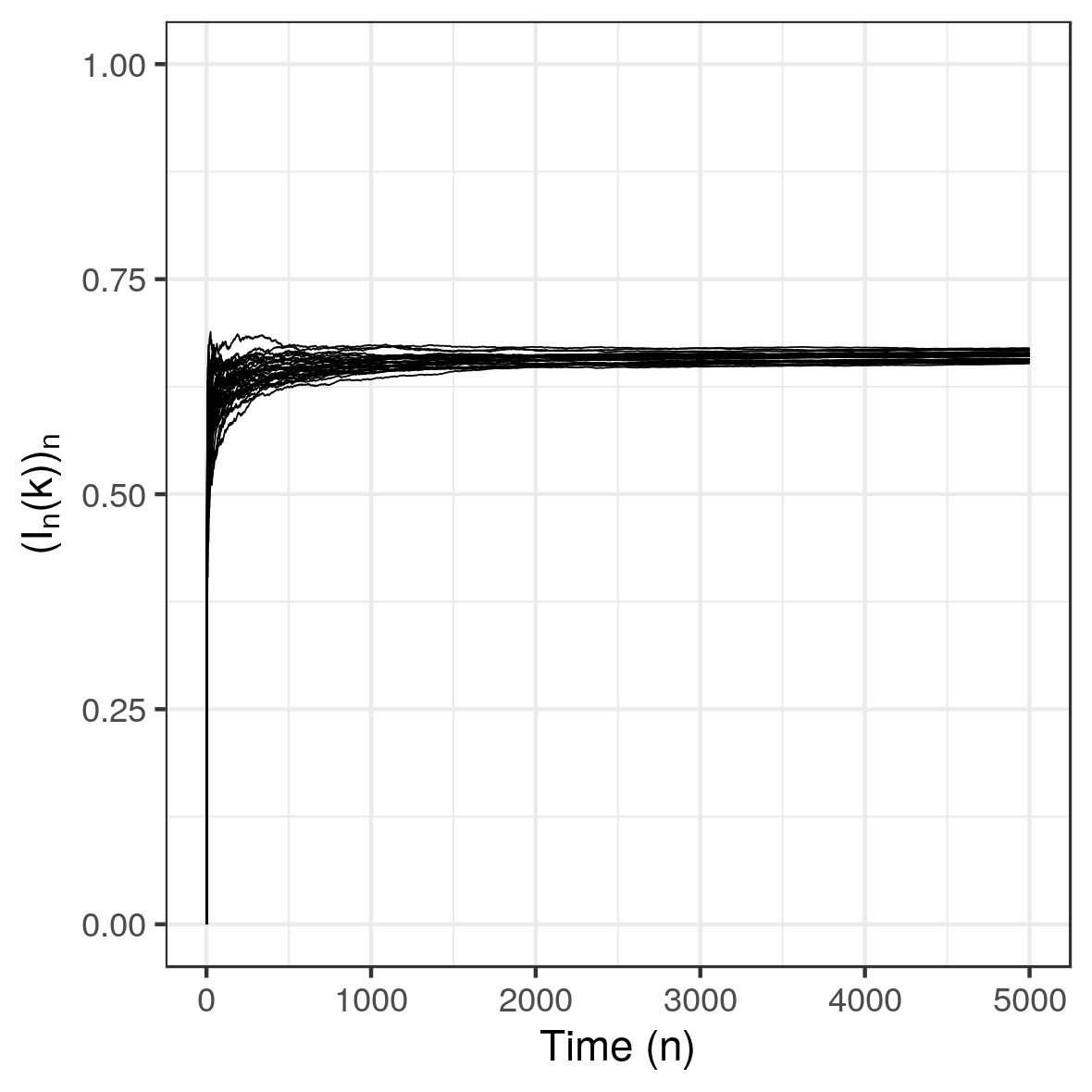}
\caption{One sample of the trajectories of the associated empirical means  $(I_n(k))_n$ ($1\leq k \leq N$).}
\label{fig:f3:one:sample:A}
\end{subfigure}
\\
\begin{subfigure}[b]{0.48\linewidth}
\centering
\includegraphics[scale=0.4,keepaspectratio=true]{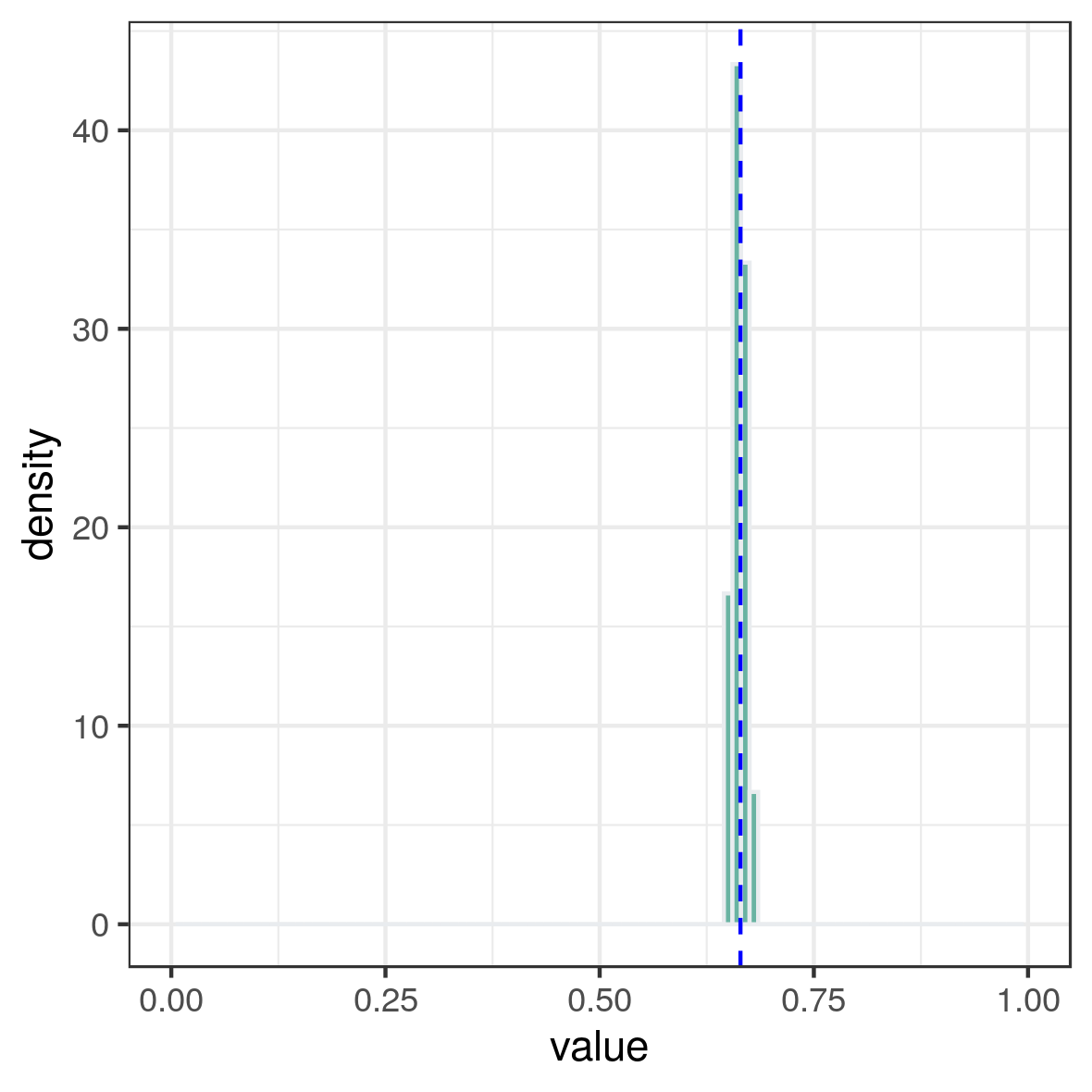}
\caption{Histogram of system's ($N=30$) components values from $100$ (independent) samples, at time $5.000$. Starting conditions are always $1/2$.}
\label{fig:f3:one:sample:B}
\end{subfigure}
\begin{subfigure}[b]{0.48\linewidth}
\centering
\includegraphics[scale=0.35,keepaspectratio=true]{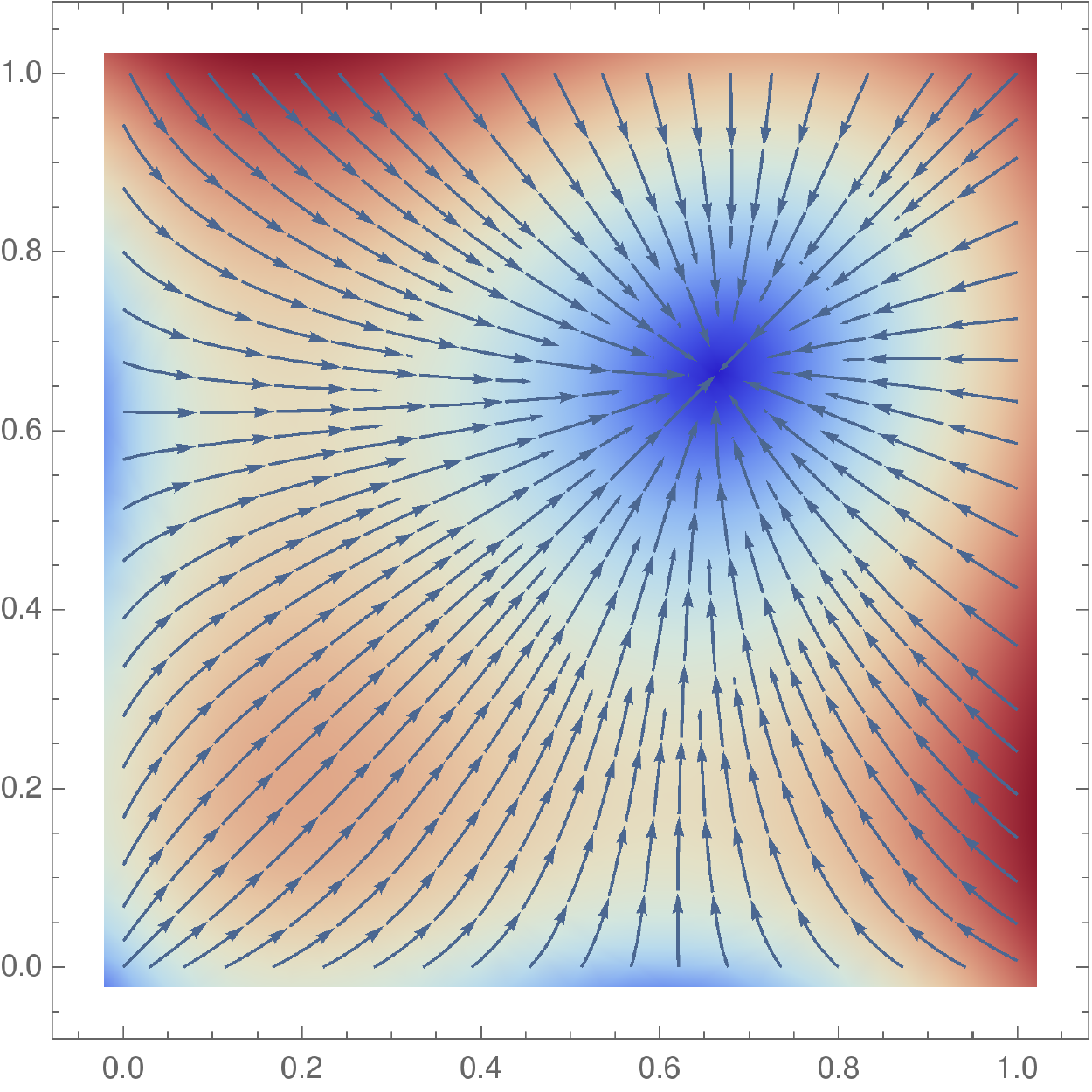}
\caption{Representation of the field~$F=- \nabla V$ when \mbox{$N=2$} 
computed using the software Mathematica~11.3.} \label{fig:f3:field}
\end{subfigure}
\caption{Case $f=f_{LP}$. Set of parameters is 
$\theta=0.9$, $x^*=1/3$, $\alpha=0.1$, $\beta=0.2$, $q=0.4$. There is a unique (stable) synchronisation point at $\approx 0.664$. The system is predictable.}  \label{fig:f3}
\end{figure}

\subsection{Case $f=f_{\mathrm{LogP}}$}

In the following Figures some specific set of parameters were chosen.
\medskip 

In Fig.~\ref{f4-setparam1:figures} parameters 
are chosen such that there is a unique stable synchronization zero point.
Fig.~\ref{f4-setparam1:figures}~(A) illustrates, for one sample, the a.s. convergence towards this value for~$(Z_n(k))_n$ and 
Fig.~\ref{f4-setparam1:figures}~(B)
for the associated 
empirical means~$(I_n(k))_n$.
Fig.~\ref{f4-setparam1:figures}~(C) pictures the histogram of the components' values $Z_n(k)$ for $n$ large and $k \in \{1,\hdots,N\}$. 
There were 100 independent samples of the whole system. Please pay attention the $(N=15)*100$ values used for the histogram are not independent.
As previously, Fig.~\ref{f4-setparam1:figures}~(D) represents, when $N=2$, the "landscape" of $F=-\nabla V$.
\medskip 

In Fig.~\ref{fig:f4:graph}, two parameters sets are considered: the one  from
Fig.~\ref{Fig:Exple3-13}
 and the one from Fig.~\ref{f4-setparam2:figures}, Fig.~\ref{f4-setparam2:figures:N5:meanfield} and Fig.~\ref{f4-case2:fig}.
For illustration, stable synchronization zero points are found at the intersection of the curve associated to~$f$ and the straight line given by Eq.\eqref{system-sincro}. In both cases there are
two stable synchronization zeros and some no-synchronization stable zero points. In Fig.~\ref{Fig:Exple3-13} the components' values are different, contrary to the set of parameters from Fig.~\ref{f4-setparam2:figures}, Fig.~\ref{f4-setparam2:figures:N5:meanfield} and Fig.~\ref{f4-case2:fig}, where the component's values are close between synchronization points and no-synchronization points.

The parameters' set of Fig.~\ref{f4-setparam3:figures}  
is related to Corollary~\ref{cor-special-case}. 
 
In  Fig.~\ref{f4-setparam2:figures}, the initial condition is always 1/2 and $N=30$ was chosen. 
In Fig.~\ref{f4-setparam2:figures:N5:meanfield} initial conditions are independent and uniformly distributed. Case $N=5$ is considered differently from Fig.~\ref{f4-case2:fig} where $N=30$ is chosen.

As deduced from Fig.~\ref{f4-setparam2:figures:N5:meanfield}~(B) and Fig.~\ref{f4-case2:fig}~(B), the convergence is not always towards the same zero point. No-synchronization zero points are observed as limit. It is possible that the synchronization zero points are rarely observed since Fig.~\ref{f4-case2:fig}~(B) and~(C) show no observation going to the synchronization values. For large values of~$N$ it seems that synchronization is never observed (unless starting with very specific starting conditions close to the synchronization points). 

In Fig.~\ref{f4-setparam3:figures}, the "landscape" associated to this parameter set when $N=2$ is shown in Subfig.~(D).
Convergence towards no-synchonization points is observed in particular in the sample Fig.~\ref{f4-setparam3:figures}~(A) with $N=100$.
Subfig.~\ref{f4-setparam3:figures}~(C) depicts trajectories 
represented in the potential landscape when $N=2$.


\begin{figure}[h]
\centering
\begin{subfigure}[t]{0.3\linewidth}
\centering
\includegraphics[scale=0.4,keepaspectratio=true]{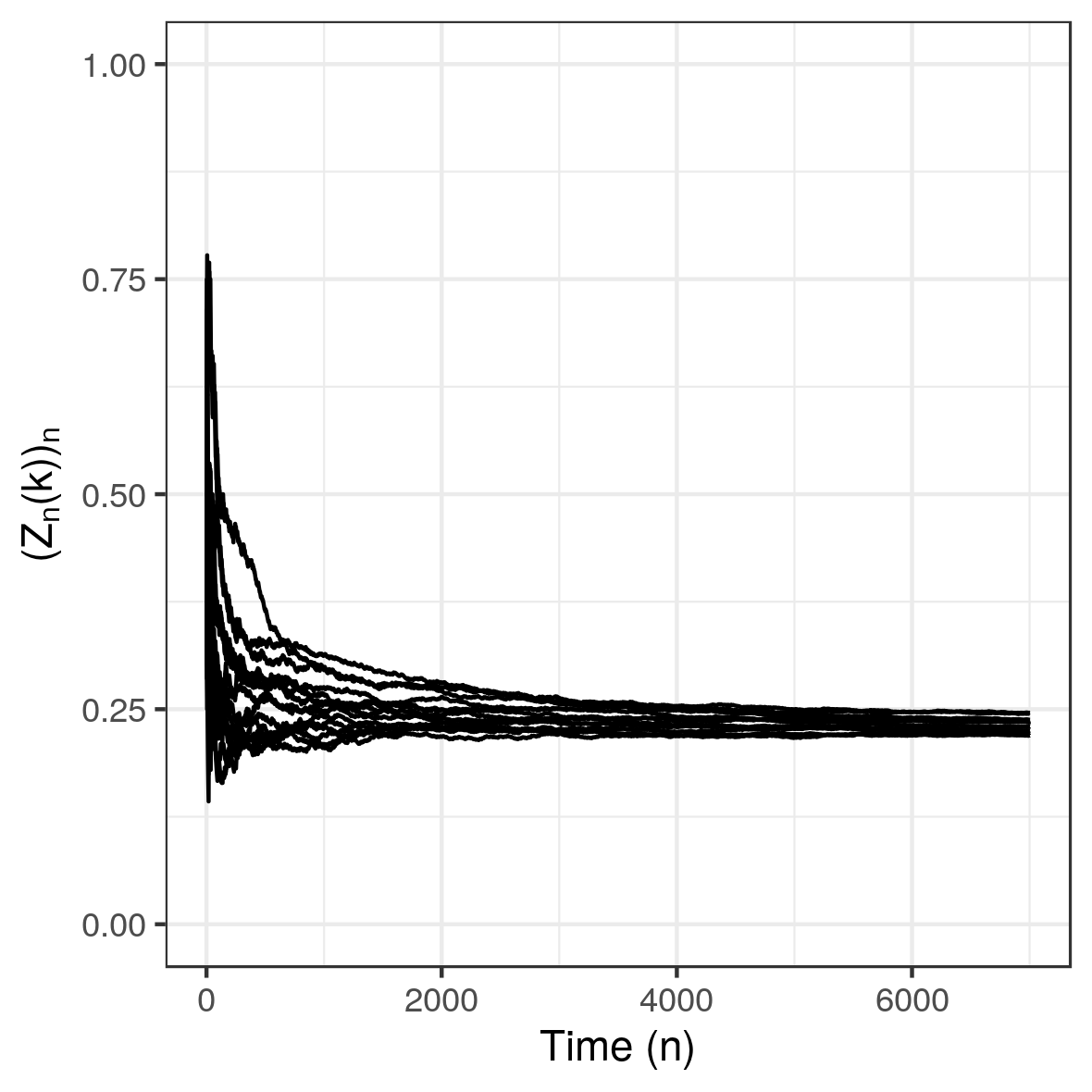}
\caption{One sample of the trajectories $(Z_n(k))_n$ (\mbox{$1\leq k \leq N$}).}
\end{subfigure}
\hfill
\begin{subfigure}[t]{0.3\linewidth}
\centering
\includegraphics[scale=0.4,keepaspectratio=true]{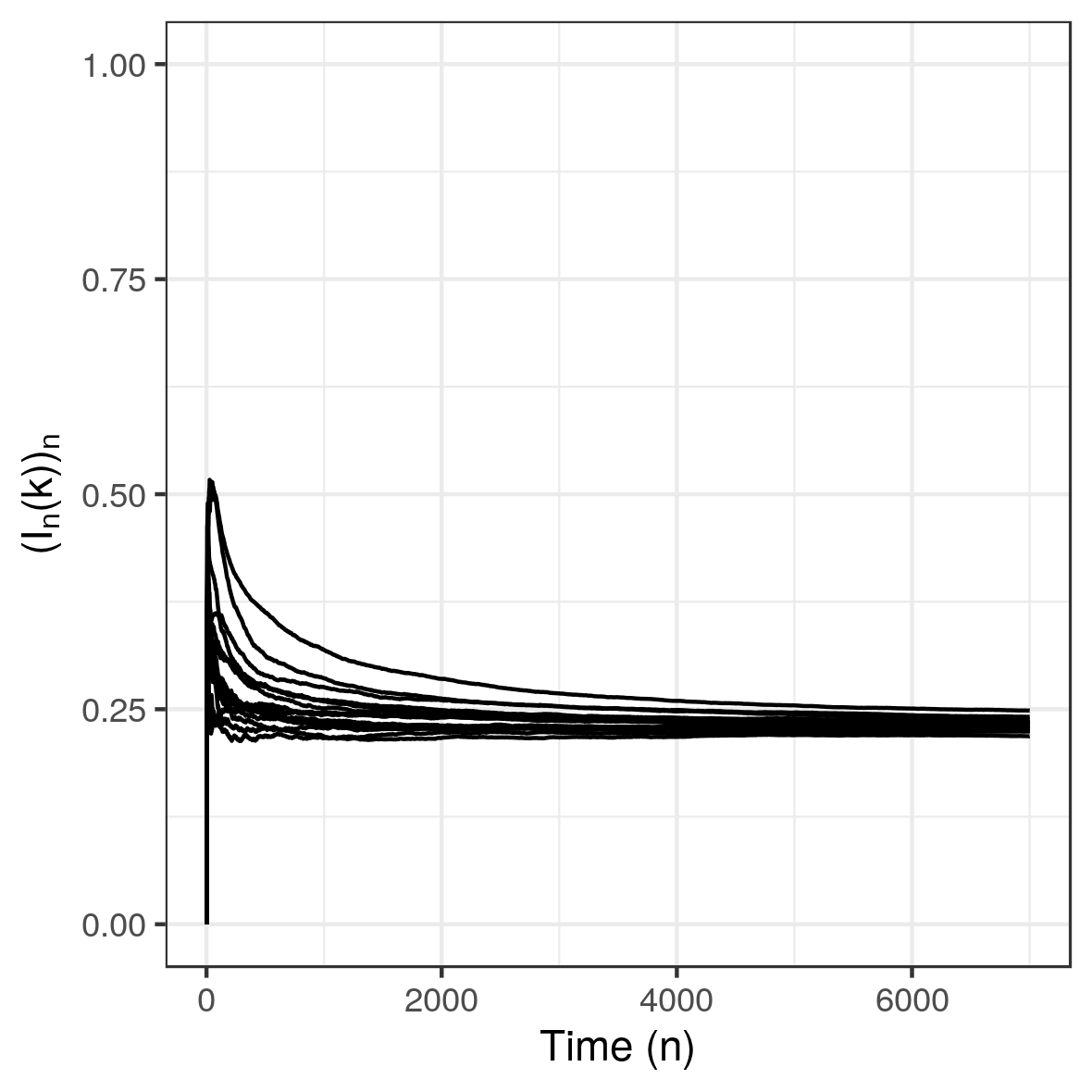}
\caption{One sample of the trajectories of the associated empirical means  $(I_n(k))_n$ ($1\leq k \leq N$).}
\end{subfigure}
\hfill
\begin{subfigure}[t]{0.3\linewidth}
\centering
\includegraphics[scale=0.4,keepaspectratio=true]{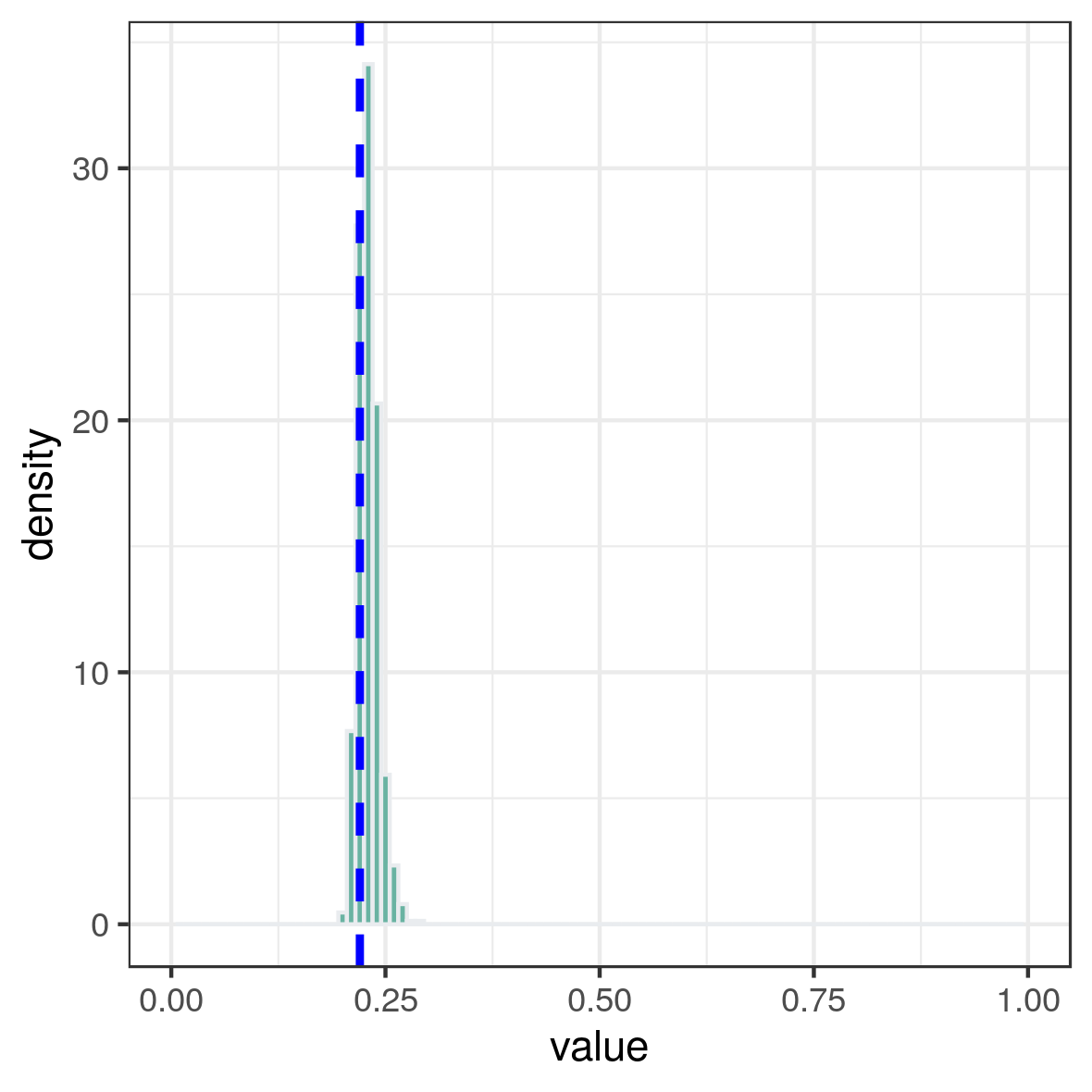}
\caption{Histogram of system's components values, for $100$ independent samples, at time $7.000$.}
\end{subfigure}
\caption{Case $f=f_{LogP}$. Parameters are $x^\star=0.6$, $\theta=5$, $\alpha=0.1$, $\beta=0.3$, $q=0.4$. There is a unique zero of~$F$ and synchronisation point at~$\approx 0.22$ (vertical dashed line in subfig.~(C)). System's size is $N=15$. Starting condition is 1/2 for all components.}
\label{f4-setparam1:figures}
\end{figure}

\begin{figure}[h]
\centering
\begin{subfigure}[t]{0.45\linewidth}
\centering
\includegraphics[scale=0.2,keepaspectratio=true]{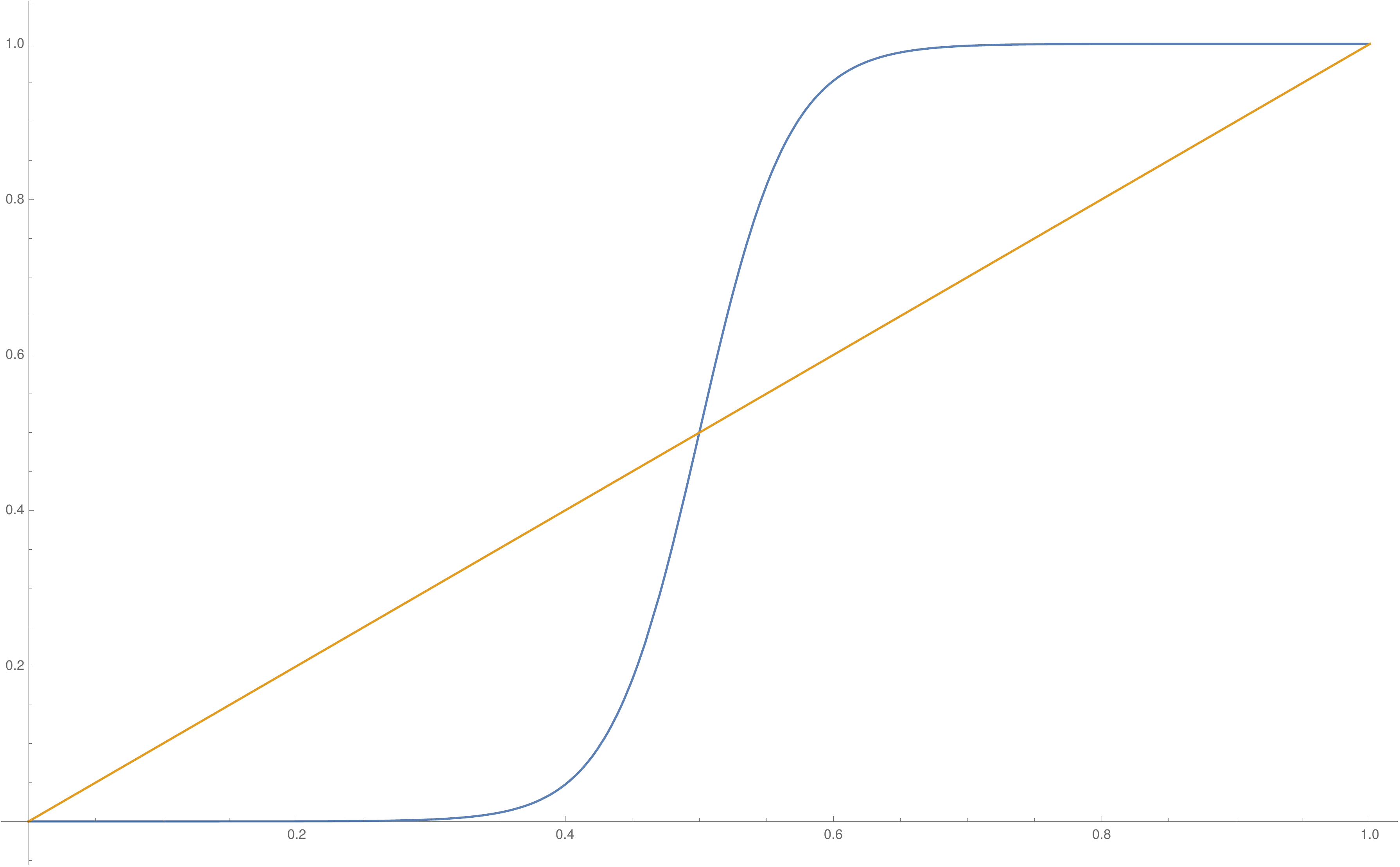}
\caption{Set of parameters  $\theta=30$, $x^*=1/2$, $\alpha=0.4$, $\beta=0$, $q=0$ as in Fig.~\ref{Fig:Exple3-13}.
It gives two stable zero synchronisation points close to 0 and close to 1 at the intersection of both curves }
\end{subfigure}
\hfill
\begin{subfigure}[t]{0.45\linewidth}
\centering 
\includegraphics[scale=0.2,keepaspectratio=true]{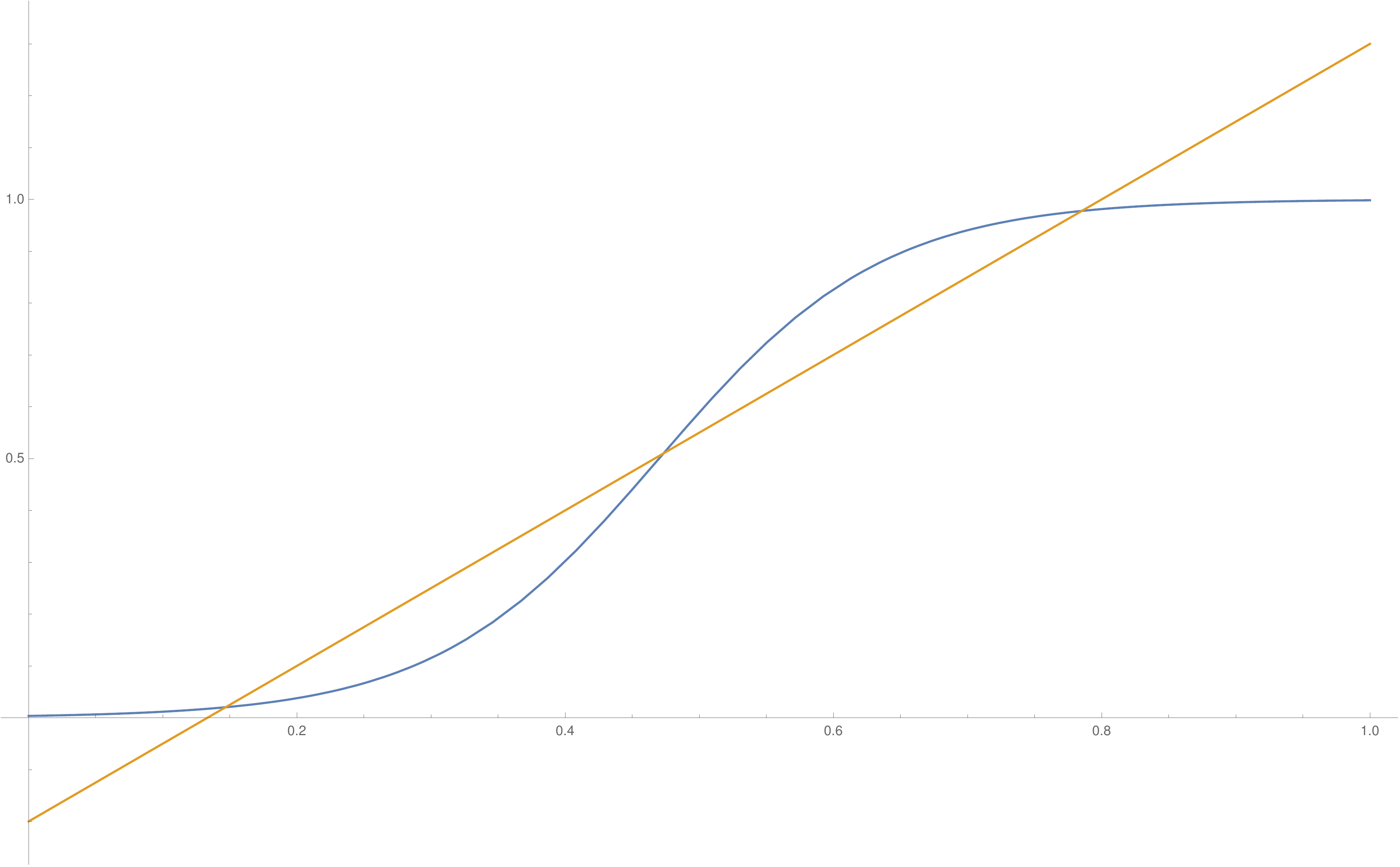}
\caption{Set of parameters  $\theta=12$, $x^*=0.47$, $\alpha=0.1$, $\beta=0.3$, $q=0.4$ as in Fig.~\ref{f4-setparam2:figures}.
It gives two stable zero synchronisation points at the intersection of both curves.}
\end{subfigure}
\caption{Case $f=f_{{LogP}}$. Graph of the function~$f_{LogP}$ in blue intersecting the straight line (orange) defined through $y= ((1-\alpha) x -\beta q)/(1-\alpha -\beta)$~Eq.\eqref{system-sincro}}.
\label{fig:f4:graph}
\end{figure}

\begin{figure}[h]
\centering
\begin{subfigure}[t]{0.49\textwidth}
\includegraphics[scale=0.45,keepaspectratio=true]{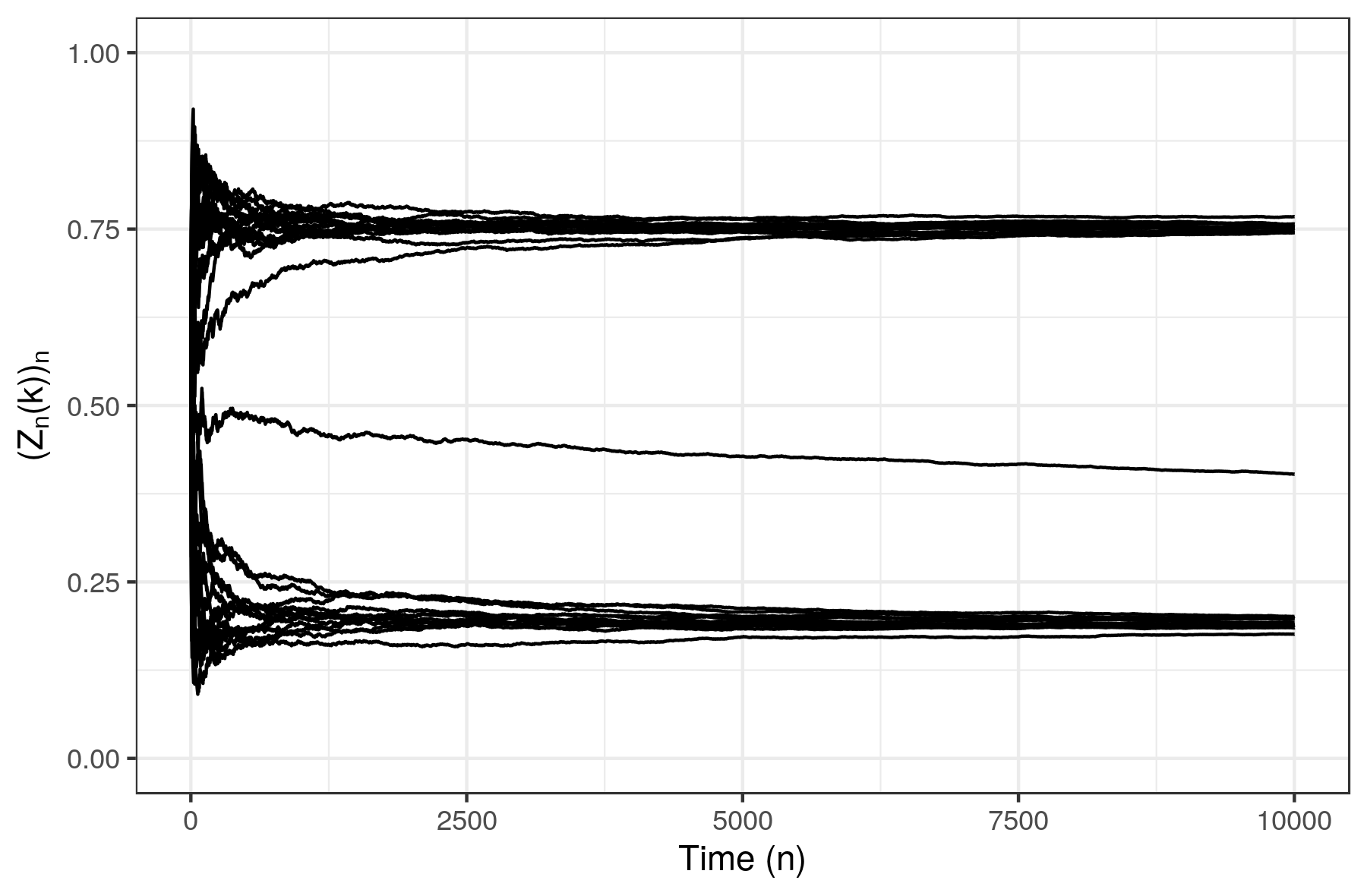}
\caption{One sample of the trajectories $(Z_n(k))_n$ (\mbox{$1\leq k \leq N$}). $N=30$.}
\end{subfigure}
\hfill
\begin{subfigure}[t]{0.49\textwidth}
\includegraphics[scale=0.45,keepaspectratio=true]{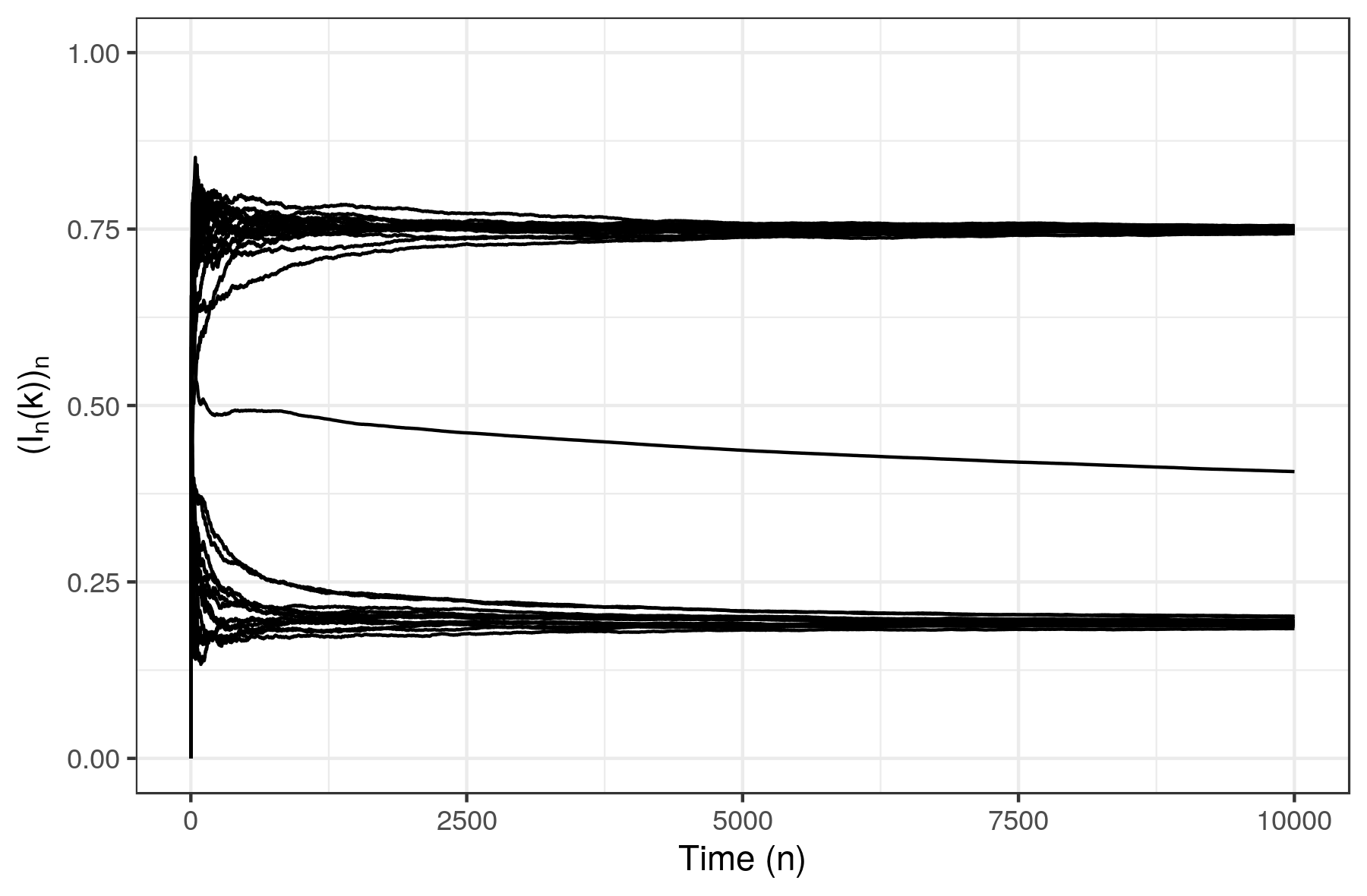}
\caption{One sample of the trajectories of the associated empirical means  $(I_n(k))_n$ ($1\leq k \leq N$).} \phantom{alignement trick}
\end{subfigure}
\\
\begin{subfigure}[t]{0.4\textwidth}
\centering
\includegraphics[scale=0.45,keepaspectratio=true]{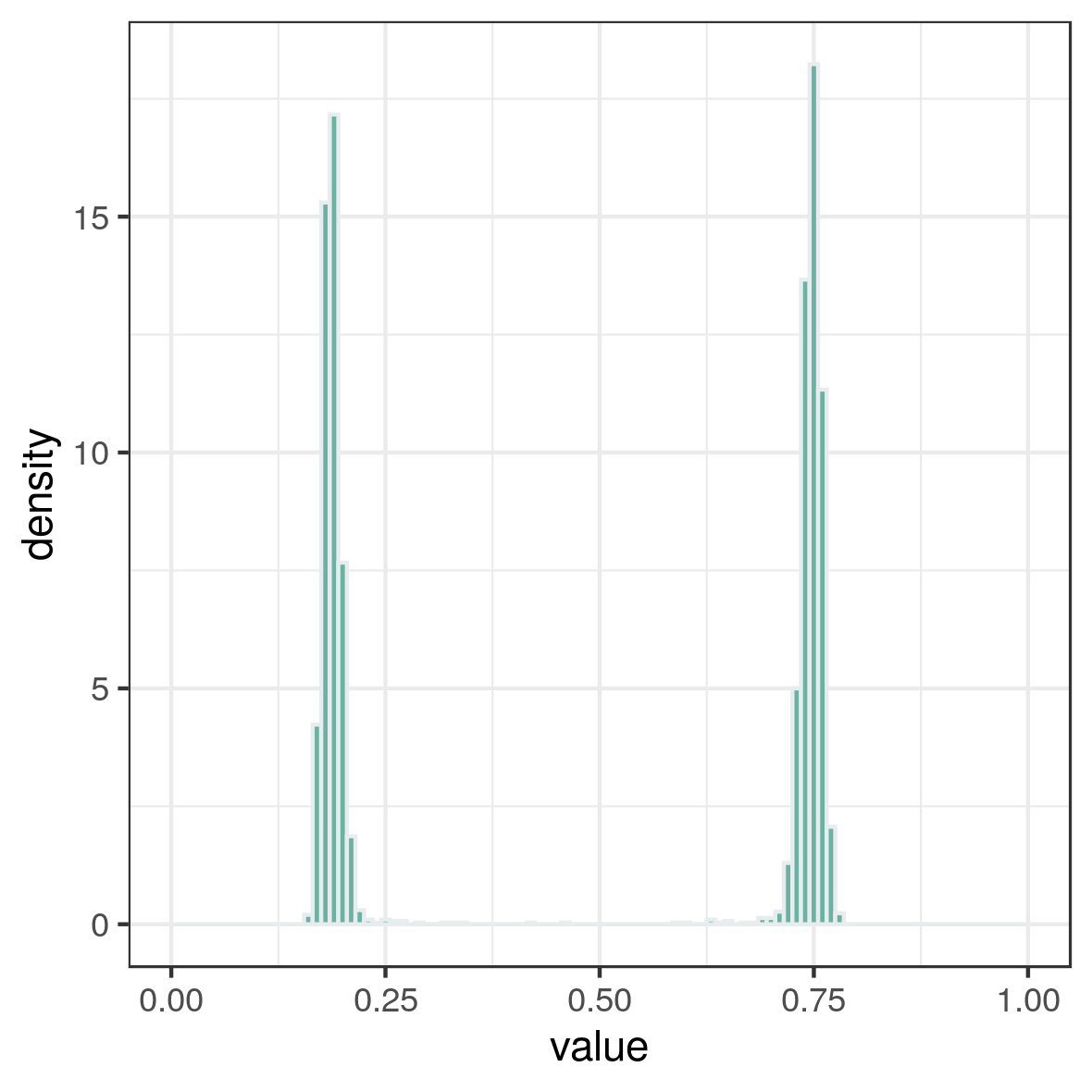}
\caption{$N=30$. Histogram of components' values at $T=10.000$ iterations. Independent system's sample of size $100$.} 
\end{subfigure}
\hfill
\begin{subfigure}[t]{0.45\textwidth}
\centering
\includegraphics[scale=0.4,keepaspectratio=true]{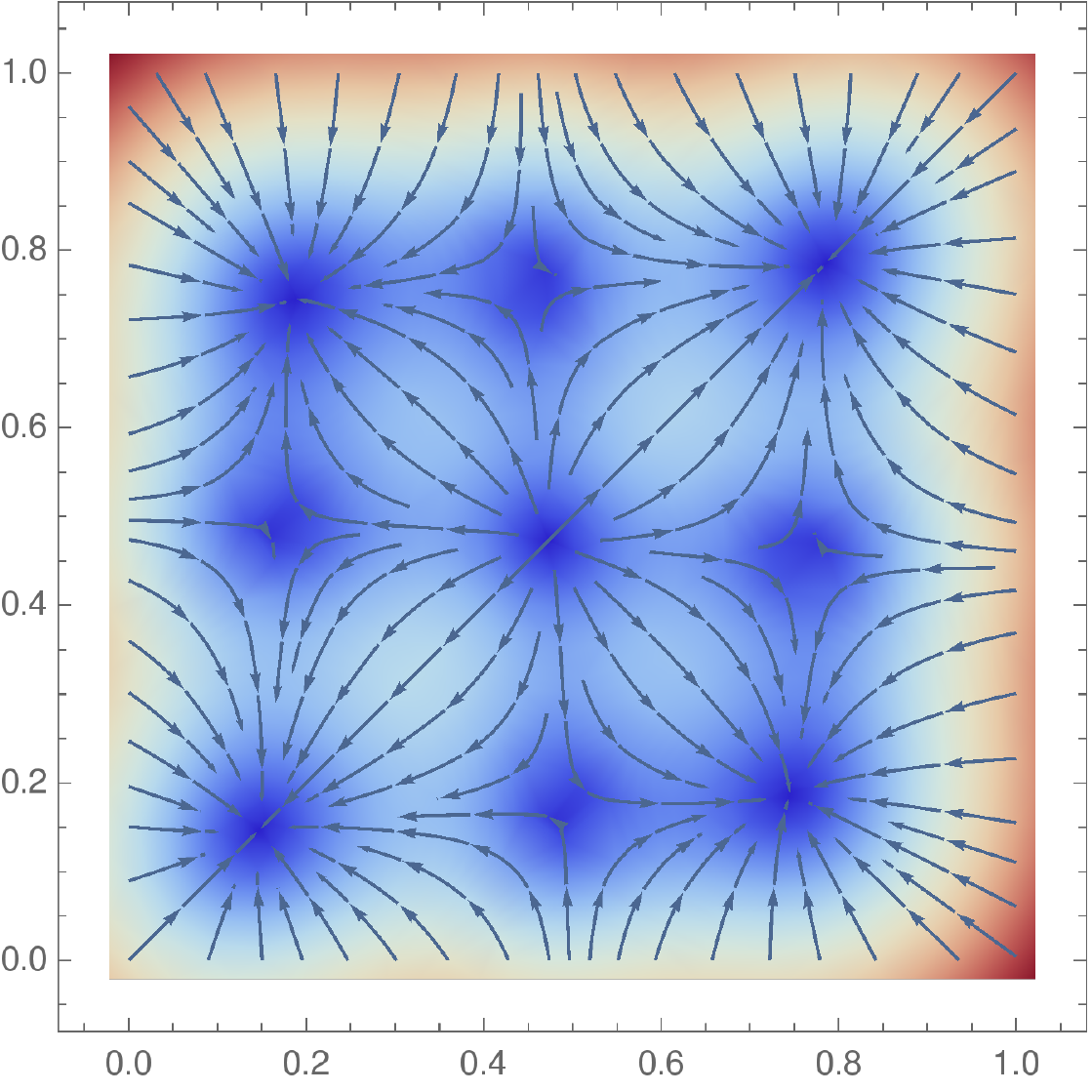}
\caption{Field~$F= -\nabla V$  associated when $N=2$.}
\end{subfigure}
\caption{Case $f=f_{LogP}$. Parameters are $x^\star=0.47$, $\theta=12$, $\alpha=0.1$, $\beta=0.3$, $q=0.4$. There are two  stable synchronization zeros $\approx \{ 0.14,\ 0.78 \}$. Components' values of (stable) non-synchronization points are close to $0.2$ and $0.8$.  Starting conditions are 1/2 for every component and sample.}
\label{f4-setparam2:figures}
\end{figure}

\begin{figure}[h]
\begin{subfigure}[t]{0.31\textwidth}
\centering
\includegraphics[scale=0.4,keepaspectratio=true]{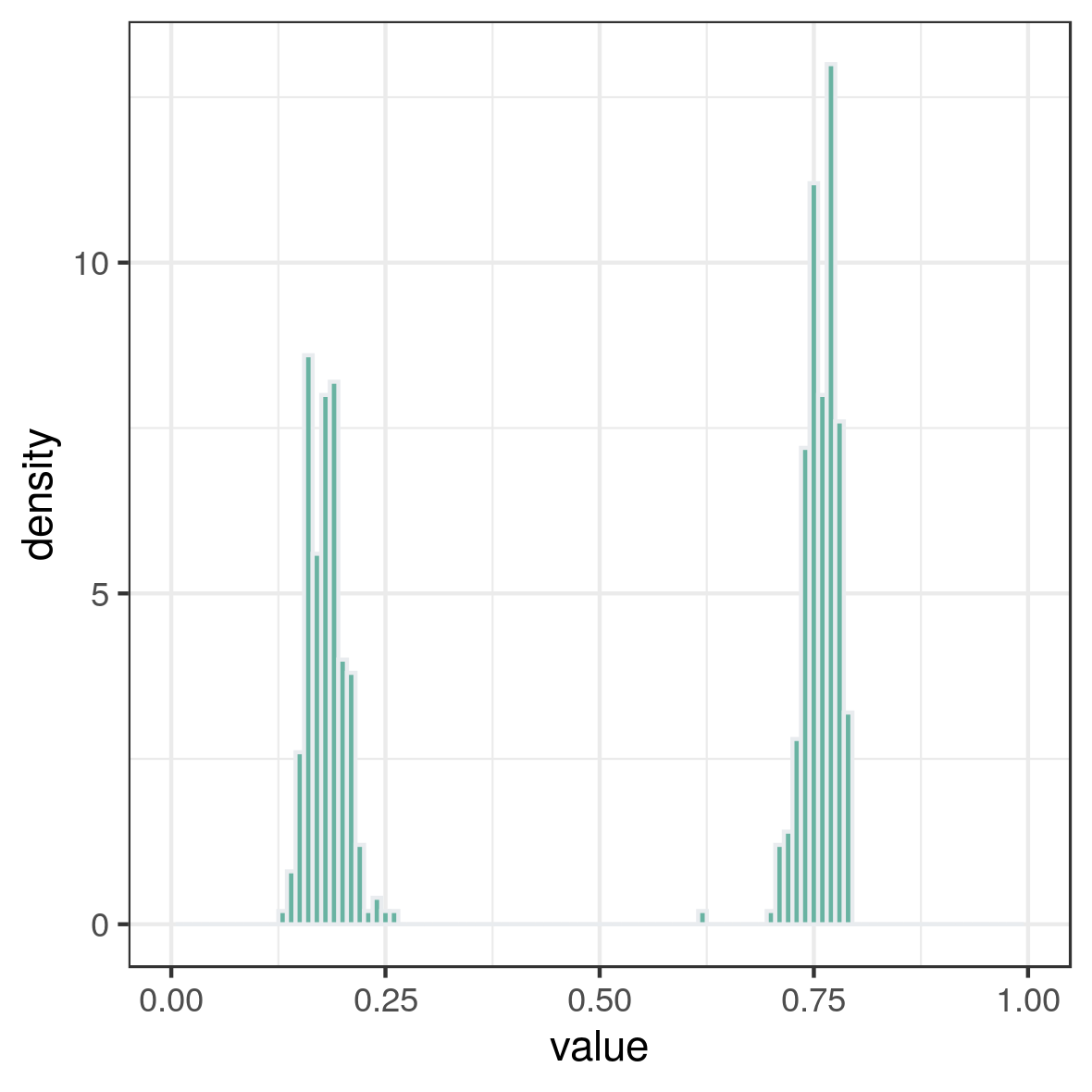}
\caption{Histogram of component's values at time $10000$.}
\label{f4-setparam2:E}
\end{subfigure}
\hfill
\begin{subfigure}[t]{0.31\textwidth}
\centering
\includegraphics[scale=0.4,keepaspectratio=true]{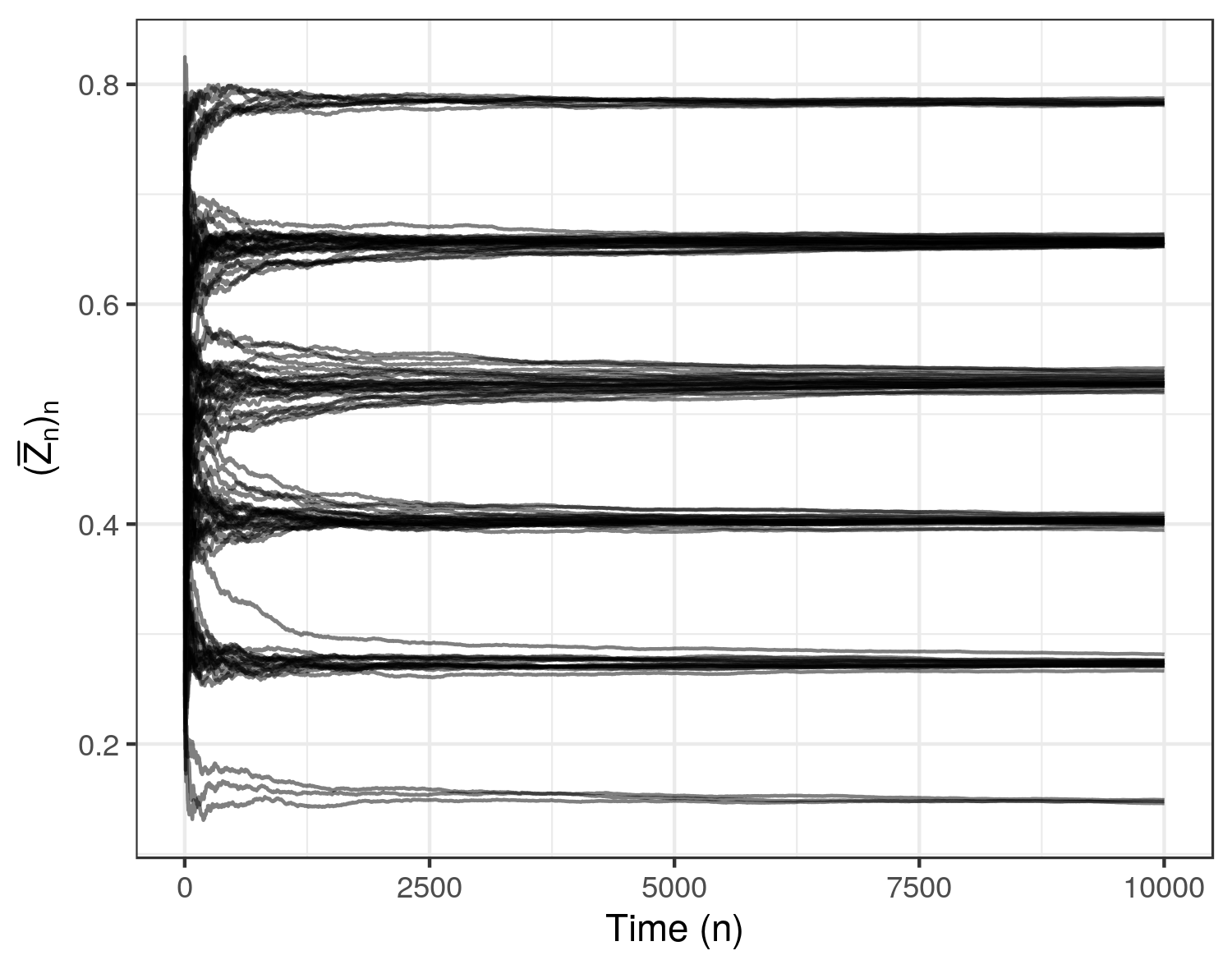}
\caption{Trajectories of mean field for each system sampled. Thus these are 100 independent curves.}
\label{f4-setparam2:mean-field-traj}
\end{subfigure}
\hfill
\begin{subfigure}[t]{0.31\textwidth}
\centering
\includegraphics[scale=0.4,keepaspectratio=true]{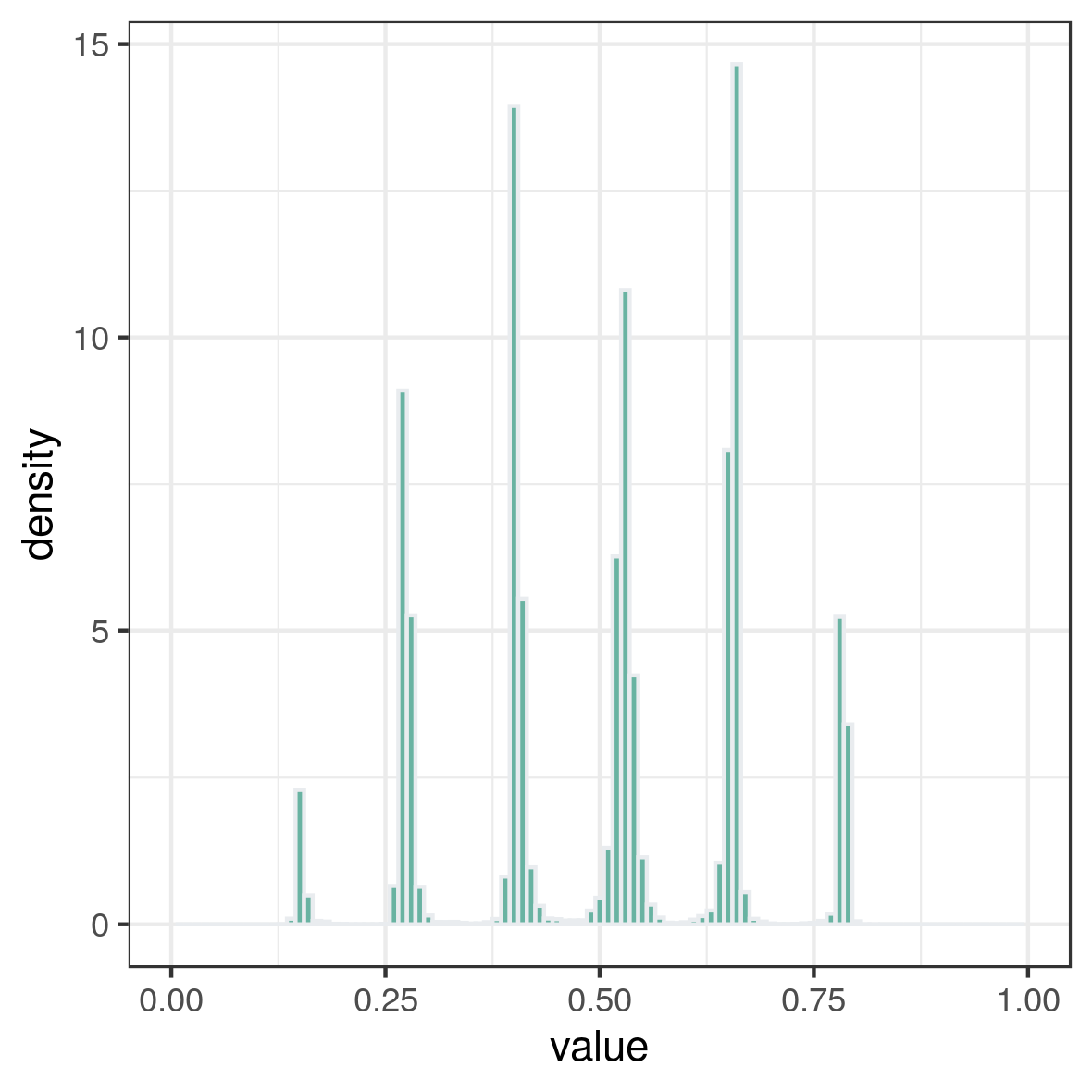}
\caption{Histogram of $100$ independent mean field's values at time $10.000$}
\label{f4-setparam2:mean-field-hist}
\end{subfigure}
\caption{Case $f=f_{LogP}$. Same set of parameters as 
Fig.~\ref{f4-setparam2:figures}: $x^\star=0.47$, 
$\theta=12$, $\alpha=0.1$, $\beta=0.3$, $q=0.4$. Here $N=5$ and  starting conditions are chosen i.i.d. uniformly distributed on [0,1].
Subfigures (A), (B) and (C) are related to the same 100 independent systems' samples.}
\label{f4-setparam2:figures:N5:meanfield}
\end{figure}

\begin{figure}
\centering
\begin{subfigure}[t]{0.3\textwidth}
\centering
\includegraphics[scale=0.4,keepaspectratio=true]{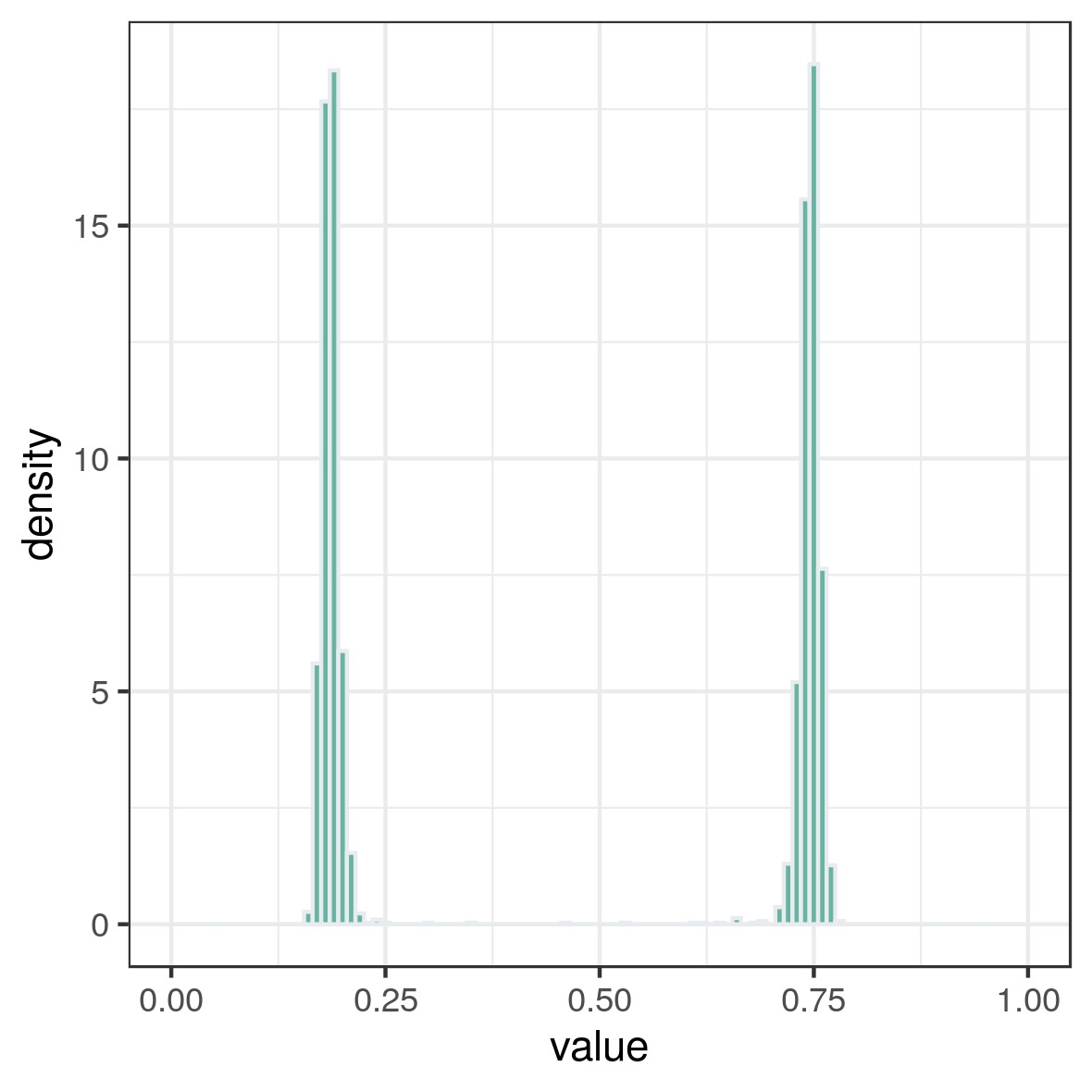}
\caption{Histogram of components' values at time $10.000$}
\end{subfigure}
\hfill
\begin{subfigure}[t]{0.3\textwidth}
\centering
\includegraphics[scale=0.4,keepaspectratio=true]{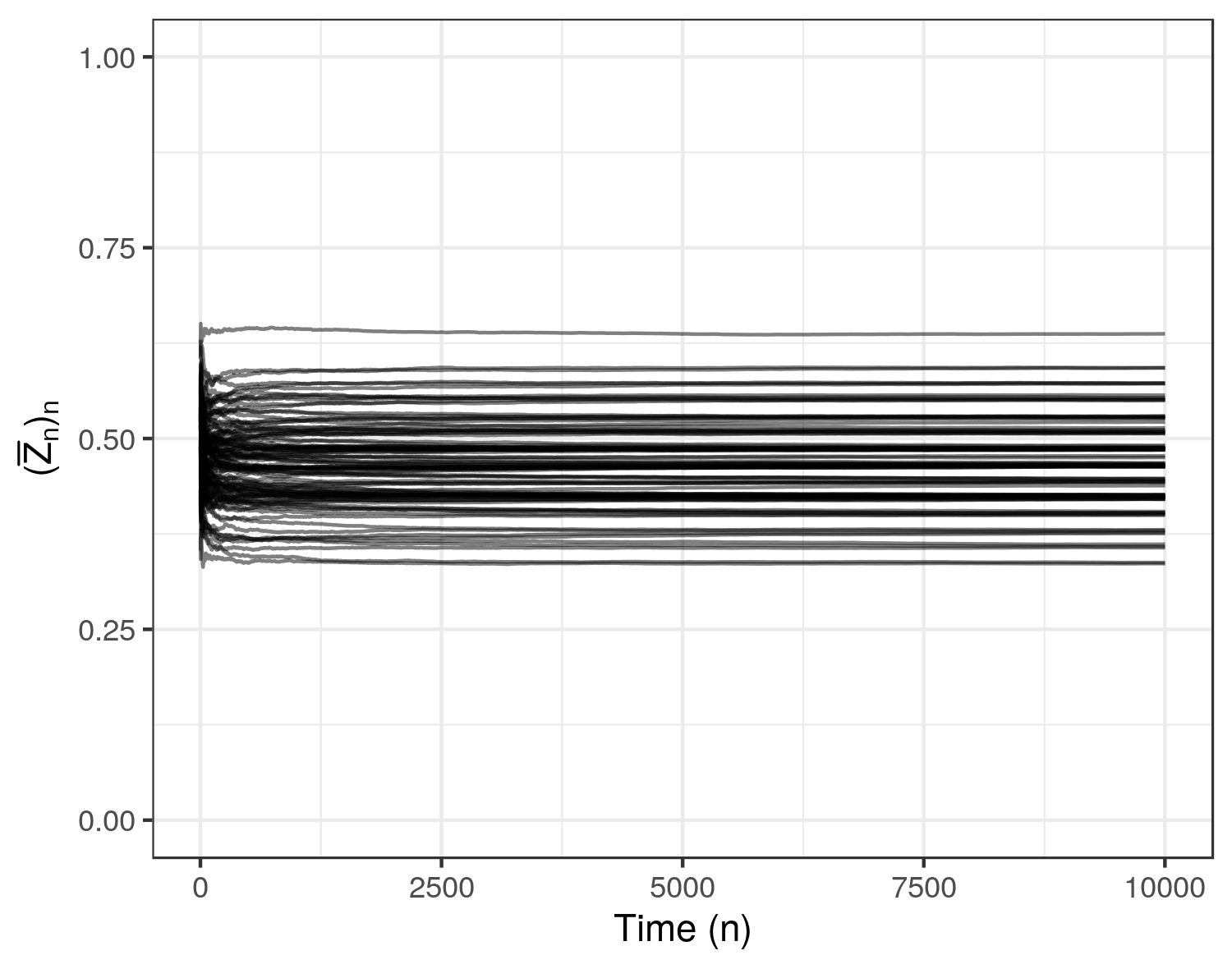}
\caption{Trajectories of mean field from $100$ independent samples.}
\end{subfigure}
\hfill
\begin{subfigure}[t]{0.3\textwidth}
\centering
\includegraphics[scale=0.4,keepaspectratio=true]{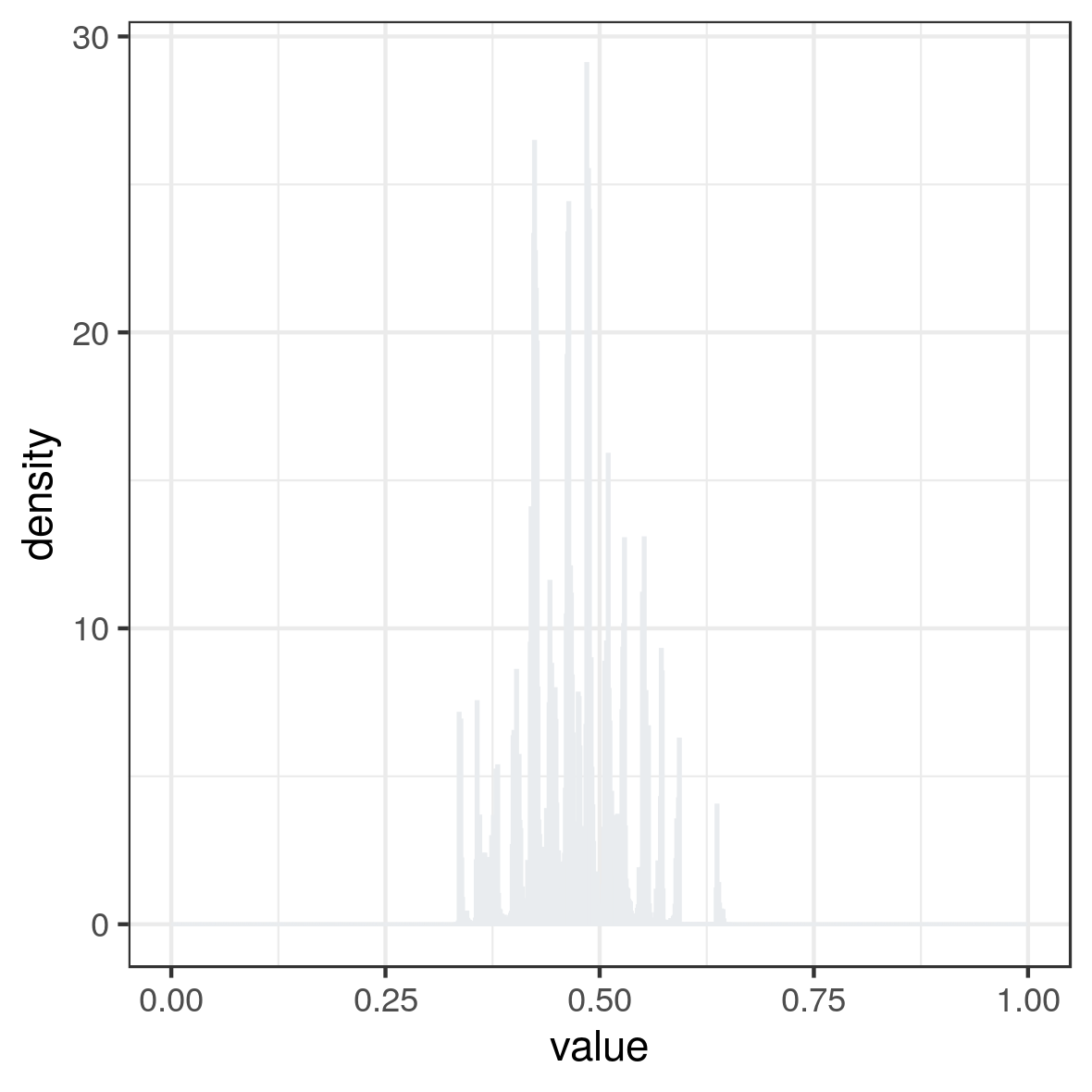}
\caption{Histogram of the mean field's values at time $10.000$, from a set of 100 independent system's samples.}
\end{subfigure}
\caption{Case $f=f_{LogP}$. Parameters as
Fig.~\ref{f4-setparam2:figures}
and Fig.~\ref{f4-setparam2:figures:N5:meanfield}
but with uniformly distributed starting conditions. Here, $N=30$.}
\label{f4-case2:fig}
\end{figure}

\clearpage 

\begin{figure}
\centering
\begin{subfigure}[t]{0.48\linewidth}
\centering
\includegraphics[scale=0.55]{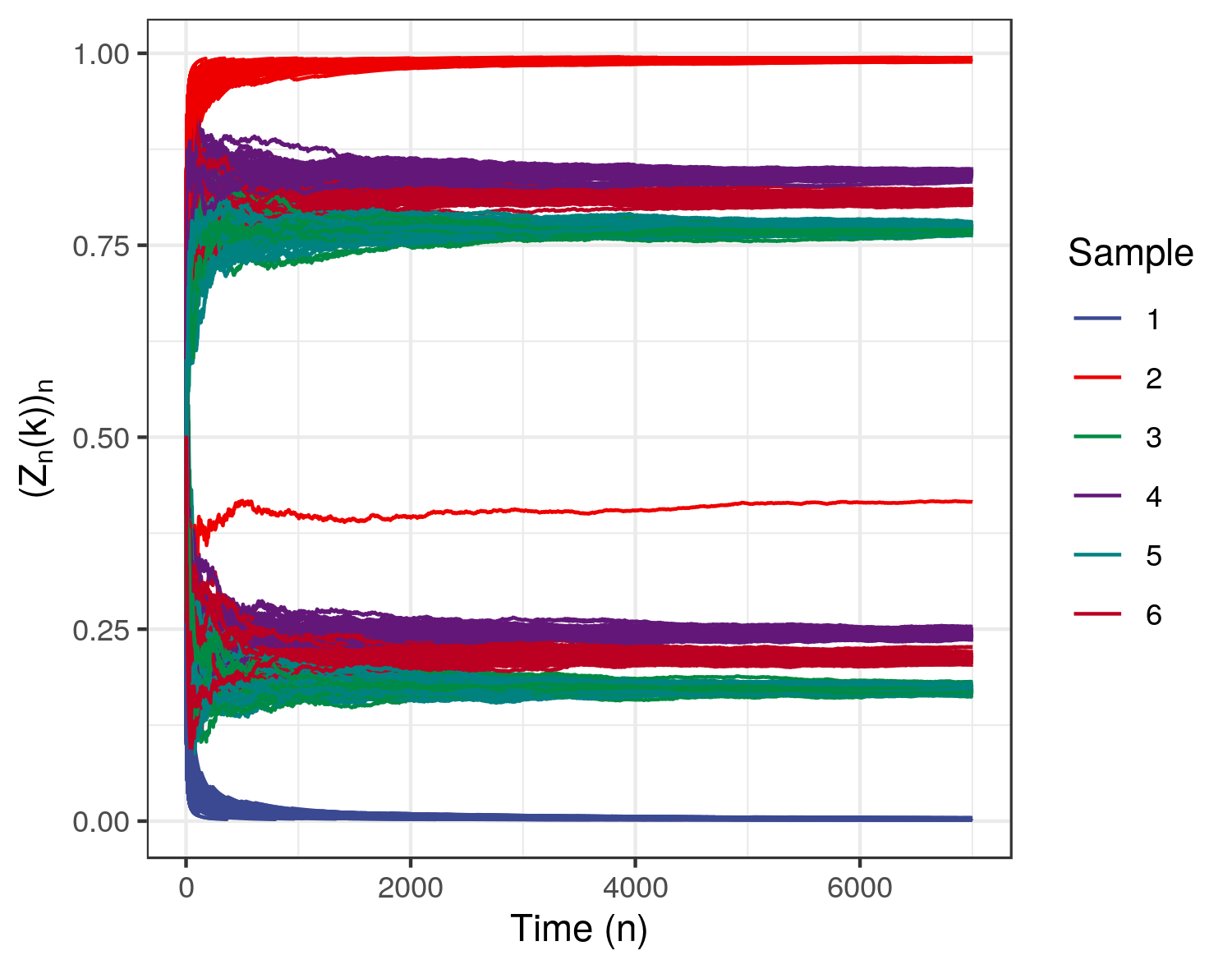} 
\caption{Six samples of the system $(Z_n(k))_n$ (\mbox{$1\leq k \leq N$}) with $N=100$. Sample~1 starts with $0.2$ on the diagonal. Sample~2 starts with $0.7$ on the diagonal.
Samples~3 to 6 starts with $0.5$ on the diagonal. All trajectories of one system's sample share the same color.}
\end{subfigure}
\hfill
\begin{subfigure}[t]{0.48\linewidth}
\centering
\includegraphics[scale=0.5]{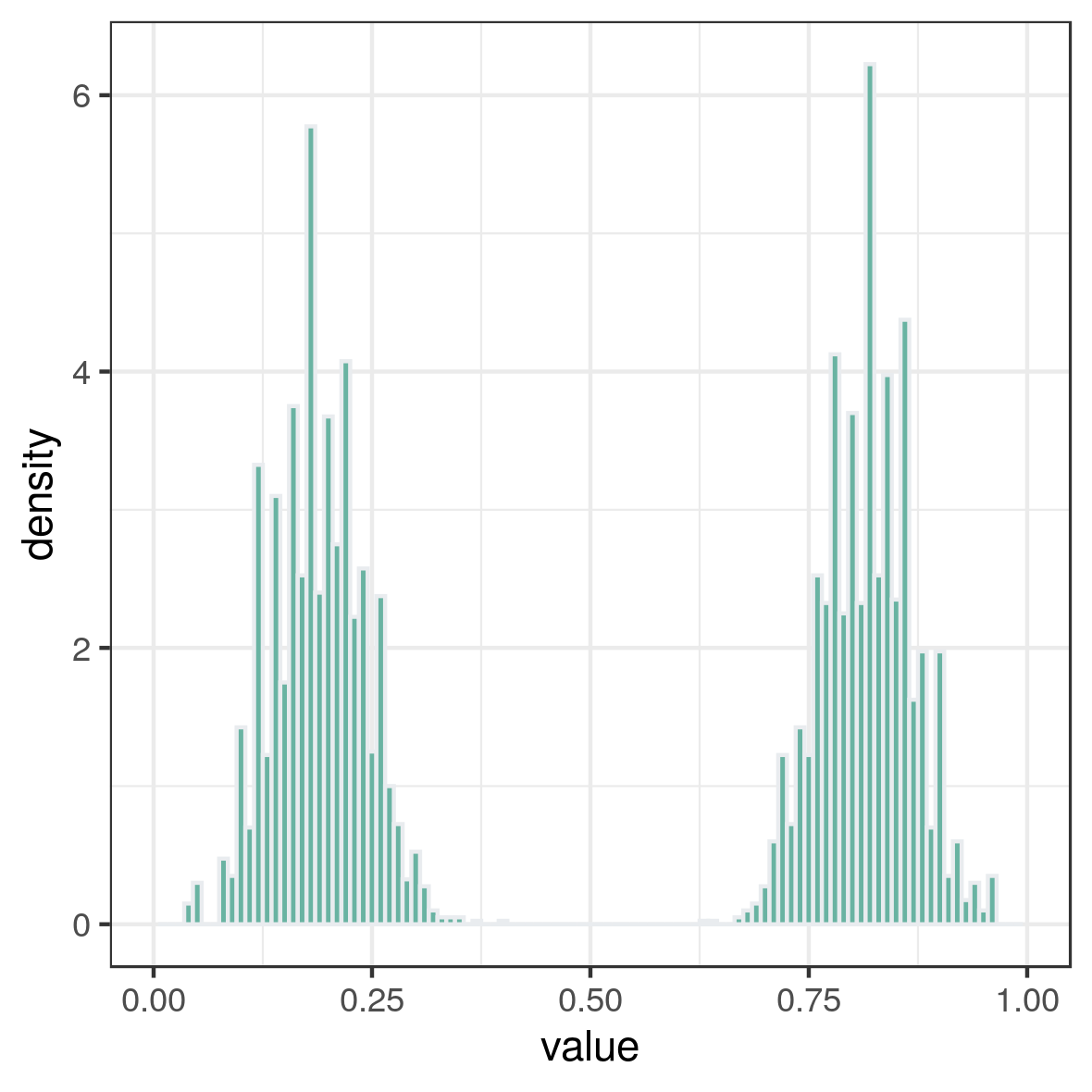}  
\caption{Histogram of components' values at  $T=6000$ for $200$ samples of the systems with $N=20$. 
For each component of each sample, the starting condition is chosen, independetly, uniformly distributed on [0,1].}
\end{subfigure}
\\
\begin{subfigure}[t]{0.69\linewidth}
\centering
\includegraphics[scale=0.55]{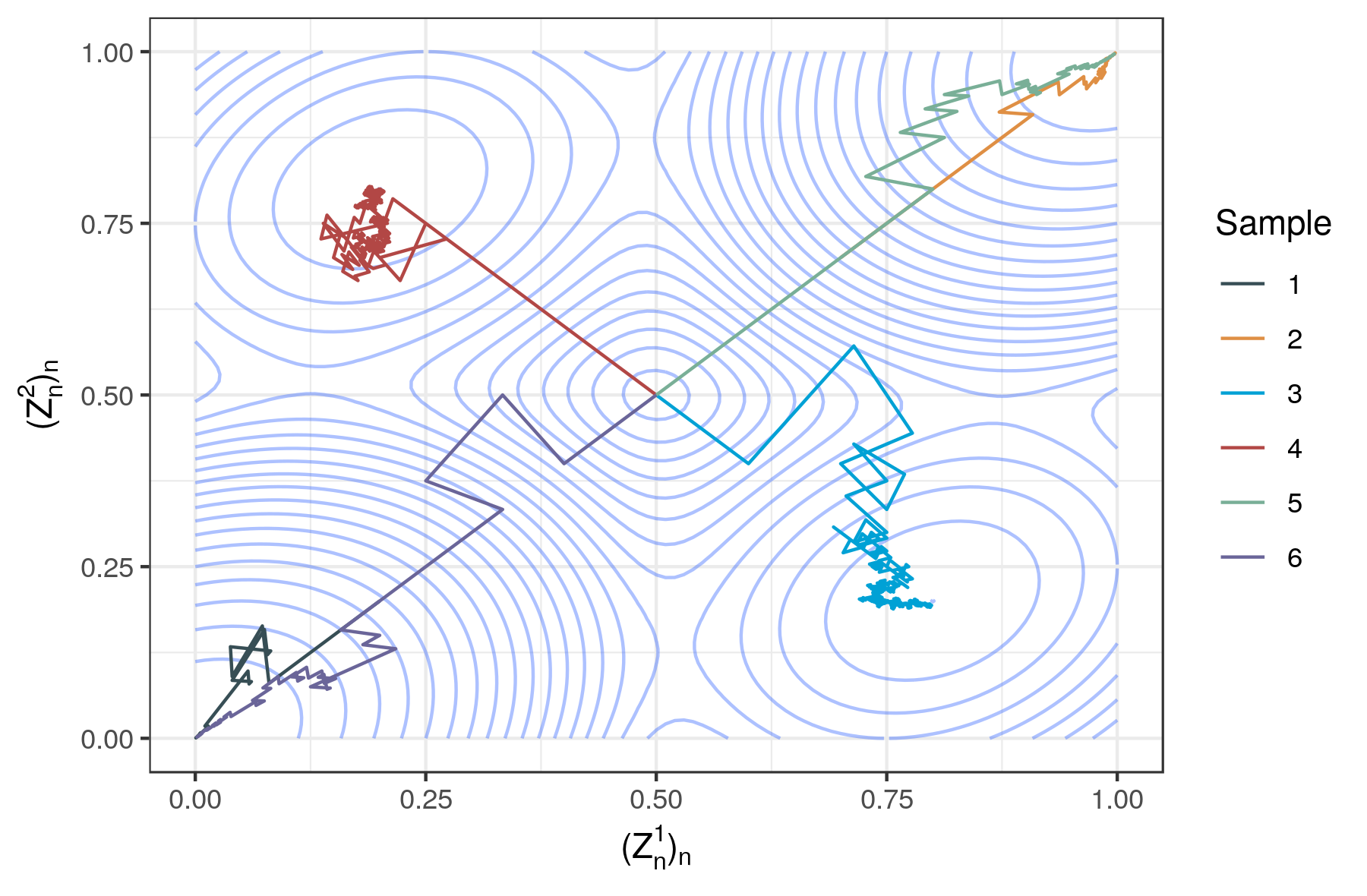}
\caption{Six samples of the system when $N=2$. Representation of trajectories $(Z_n(1),Z_n(2))_n$ up to $30.000$ iterations. Each color means a different sample. Sample~1 starts with $(0.2,0.2)$. Sample~2 starts with $(0.7,0.7)$.
Samples~3 to 6 start with $0.5$ on the diagonal. In background some level sets of~$-V$.}
\label{f4:Exple3-13-phase-space-traj}
\end{subfigure}
\hfill
\begin{subfigure}[t]{0.29\linewidth}
\centering
\includegraphics[scale=0.65,keepaspectratio=true]{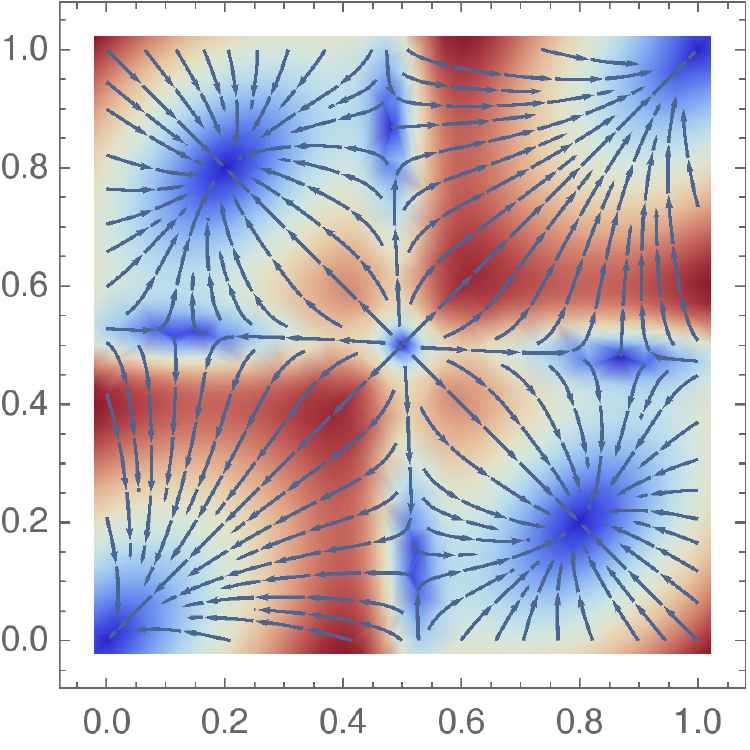}
\caption{Field~$F=-\nabla V$ when \mbox{$N=2$}.} 
\end{subfigure}
\\
\begin{subfigure}[t]{0.48\linewidth}
\centering
\includegraphics[scale=0.45]{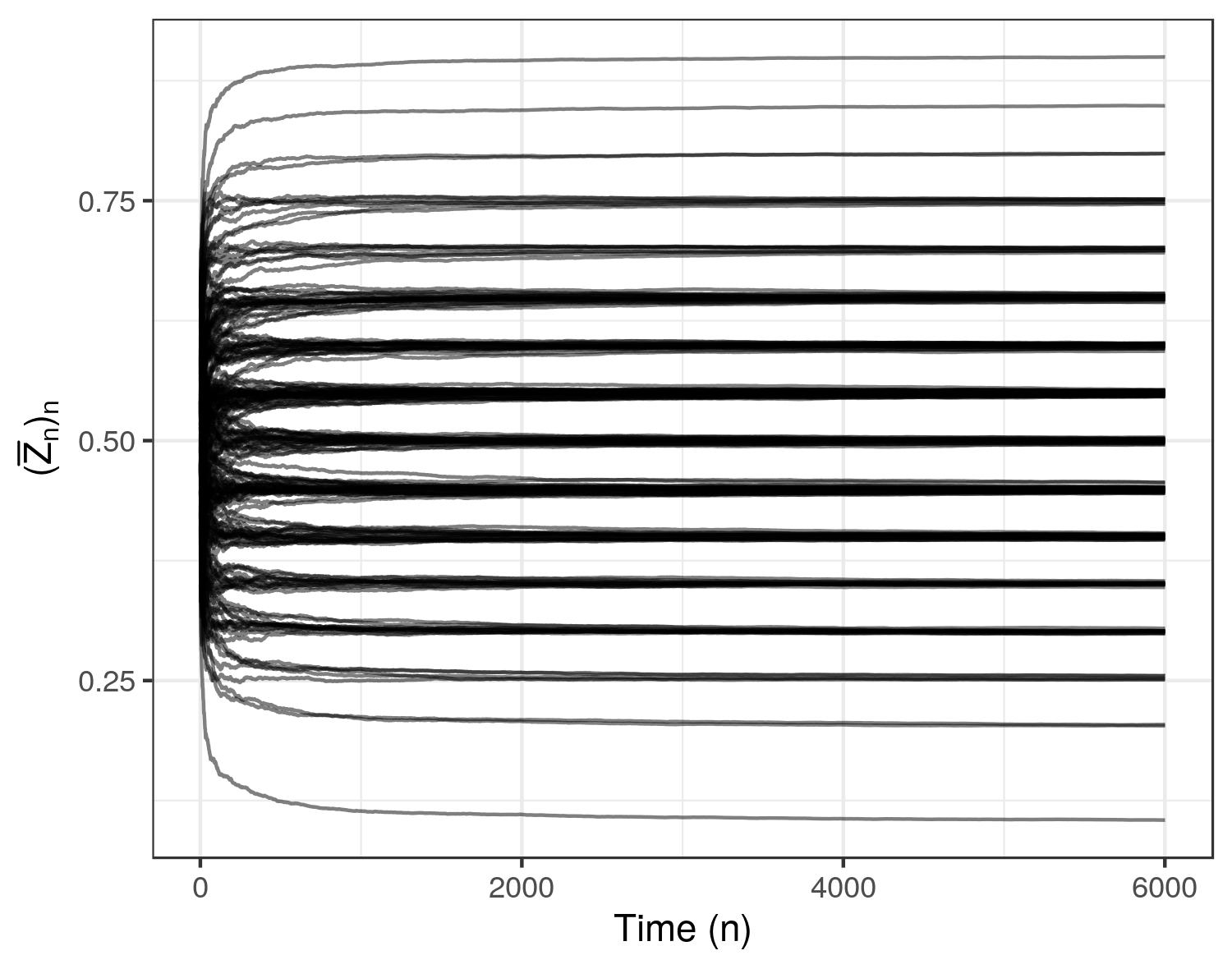} 
\caption{Mean fields trajectories associated to (B).} 
\end{subfigure}
\hfill
\begin{subfigure}[t]{0.48\linewidth}
\centering
\includegraphics[scale=0.45]{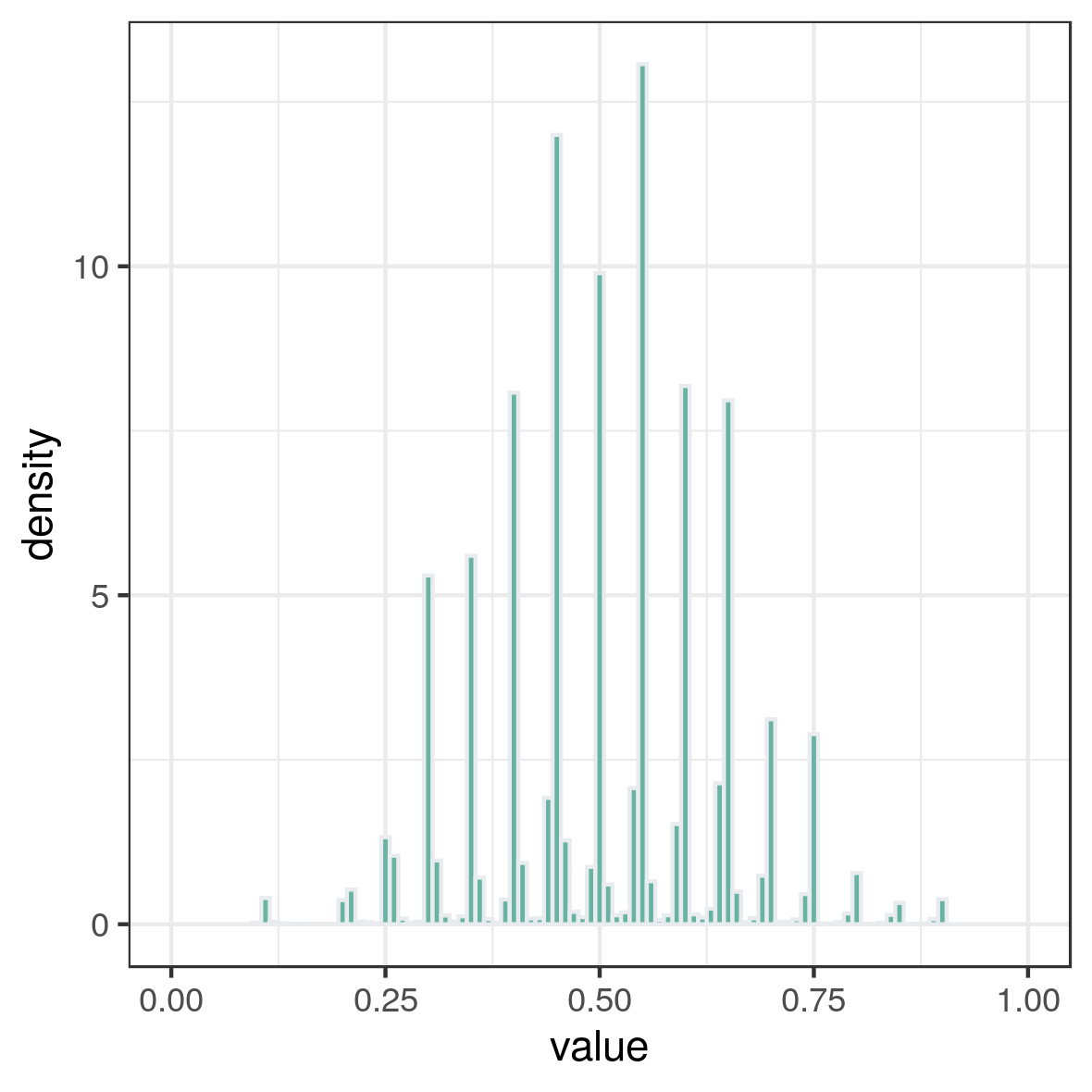} 
\caption{Histogram of mean field values
 at  $T=6000$ for $200$ independent samples of the system with $N=20$. 
 Same sample as in (B) and (E)}
\end{subfigure}
\caption{Case $f=f_{LogP}$. Parameters are $x^\star=0.5$, $\theta=30$, $\alpha=0.4$, $\beta=0$. This is related to Corollary~\ref{cor-special-case}. Synchronization points are close to $0$ and $1$ (stable) and $x^*=0.5$ (unstable). Components' values of no-synchronization (stable) points are close to $0.2$ and $0.8$.}
 \label{Fig:Exple3-13}  \label{f4-setparam3:figures}
\end{figure}

\clearpage

\subsection{Case $f=f_{\textrm{Tech}}$}


Parameters related to Fig.~\ref{ftech-setparam4:figures} are such that there are two stable synchronization zero points and there exist stable no-synchronization zero points. As it can be observed from simulations in~(C) and in the landscape~(D), the dynamics can be very slow close to these points. Samples from~(A) comforts this observation.

Parameters related to Fig.~\ref{ftech-setparam6:figures} 
are such that there are two stable synchronization zero points and there are no stable no-synchronization zero points. In Subfig.~\ref{ftech-setparam6:figures}~(C) it can be seen on these samples that the dynamical behavior is slow in the neighborhood of the unstable no-synchronization points. Contrary to what is observed, due to finite number of iterations, thanks to the previously mentioned theoretical results, we know convergence will eventually happen towards stable synchronization points.

In Fig.~\ref{ftech-setparam2:figures}, parameters are set up such that there are only two stable synchronization points.
From $N=2$ cases, unstable no-synchronization points can be guessed to be in regions where the dynamics are slow. For instance, in 
Subfig.~Fig.~\ref{ftech-setparam2:figures}~(B), the sample~6 starts at $(0.6,0.1)$ and does not succeed to reach the neighborhood of a synchronization point before 150.000 iterations.

\begin{figure}
\centering
\begin{subfigure}[t]{0.69\linewidth}
\centering
\includegraphics[scale=0.6]{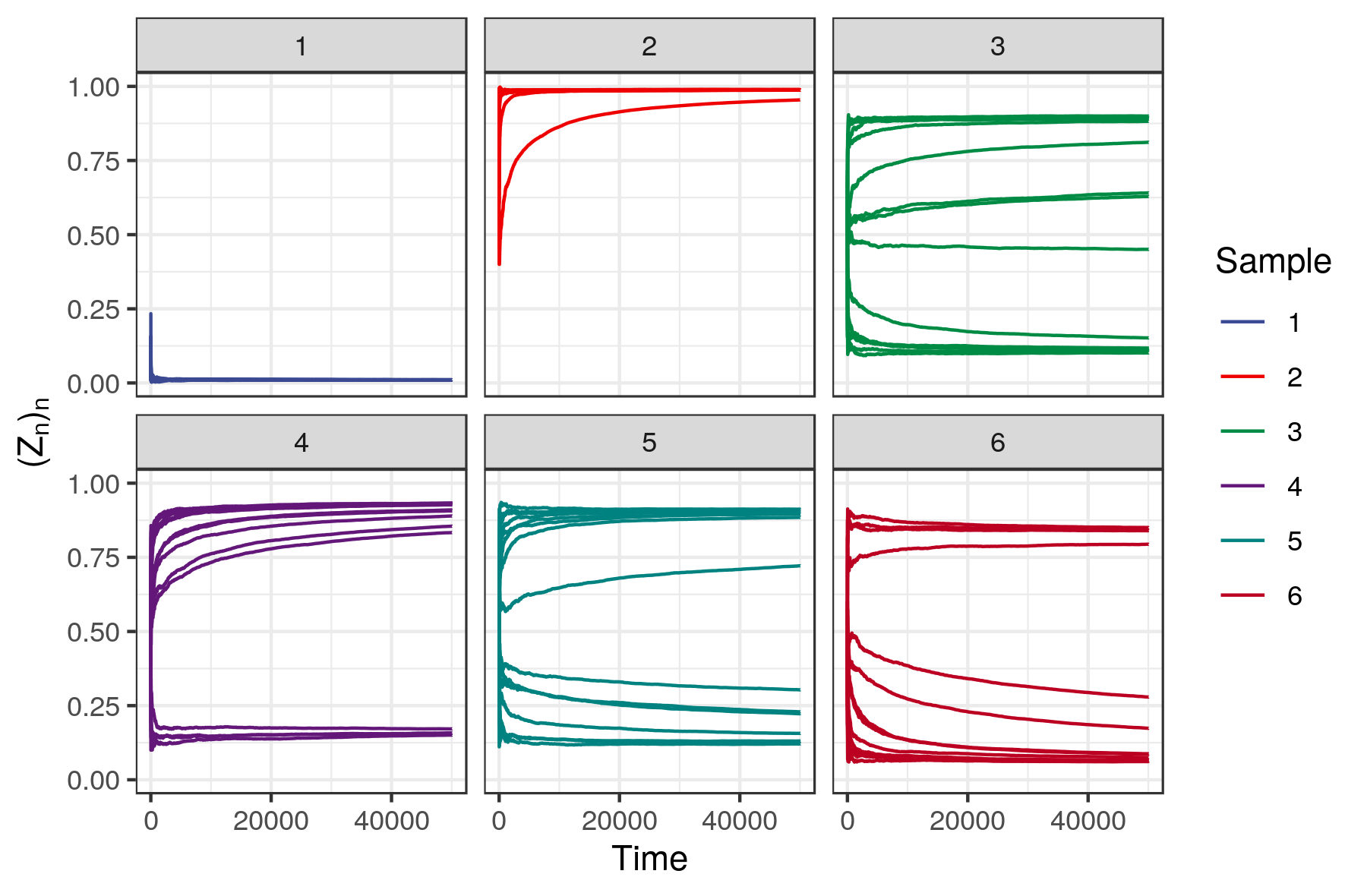}
\caption{Six samples of the system's trajectories when $N=15$. Each component starts at $0.5$ on the diagonal.}
\end{subfigure}
\begin{subfigure}[t]{0.29\linewidth}
\centering
\includegraphics[scale=0.6]{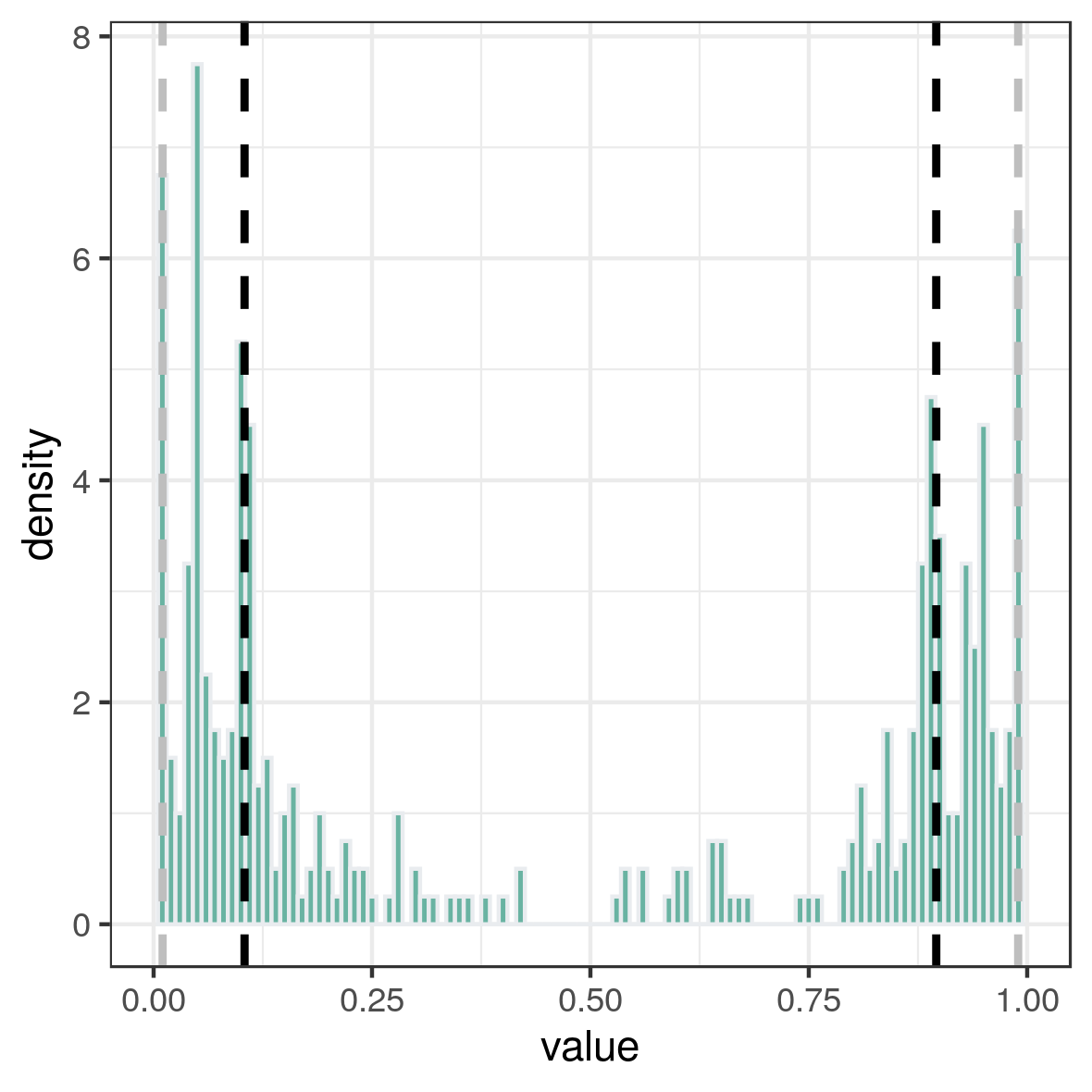}
\caption{Histogram of components' values at time $T=50.000$ when $N=4$. All trajectories start in $1/2$. Vertical grey dashed lines are at stable synchronization points values $0.010$ and $0.989$. Vertical black dashed lines are at non-synchronization values $0.104$, and $0.896$.}
\end{subfigure}
\\
\begin{subfigure}[t]{0.69\linewidth}
\centering
\includegraphics[scale=0.6]{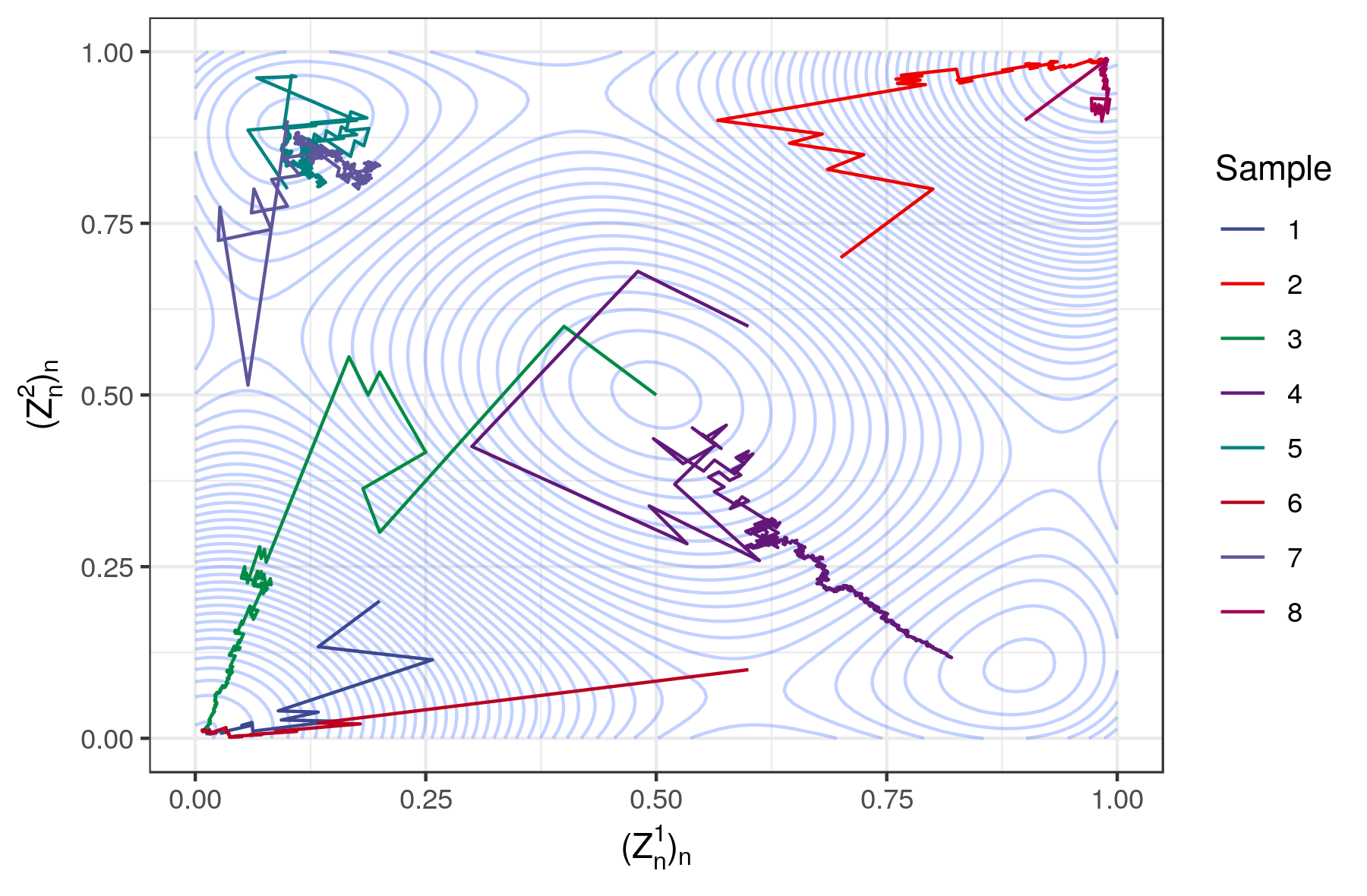}
\caption{Eight samples of the systems when $N=2$.  Representation of trajectories $(Z_n(1),Z_n(2))_n$. Each color means a different sample. Sample~1 starts with $(0.2,0.2)$. Sample~2 starts with $(0.7,0.7)$.
Samples~3 and 4 start with $(0.5,0.5)$.
Sample~5 starts with $(0.1,0.8)$. Sample~6 starts with $(0.6,0.1)$.
Sample~7 starts with $(0.1,0.9)$. Sample~8 starts with $(0.4,0.95)$.
In background some level sets of the associated~$-V$.  Time up to $20.000$ iterations.}
\label{ftech-phase-space-traj-setparam4}
\end{subfigure}
\hfill
\begin{subfigure}[t]{0.29\linewidth}
\centering
\includegraphics[scale=0.3]{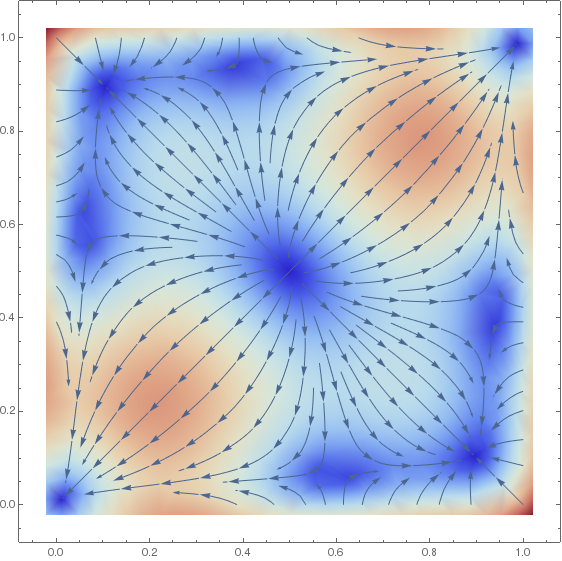}
\caption{Representation of the field $F=-\nabla V$ when $N=2$.}
\label{ftech-setparam4-N2}
\end{subfigure}
\caption{Case $f=f_{\textrm{Tech}}$. Parameters are $\theta=0.99$, $\alpha=0.14$, $\beta=0$. There are two stable synchronization points 
$\{\approx 0.0103, \ \approx 0.989\}$ 
 and  there are stable non-synchronization points.
Components' values of non-synchronization zero points belong to $\{\approx 0.104,\ \approx 0.896\}$.} 
\label{ftech-setparam4:figures}
\end{figure}
%


\begin{figure}
\centering
\begin{subfigure}[t]{0.69\linewidth}
\centering
\includegraphics[scale=0.5]{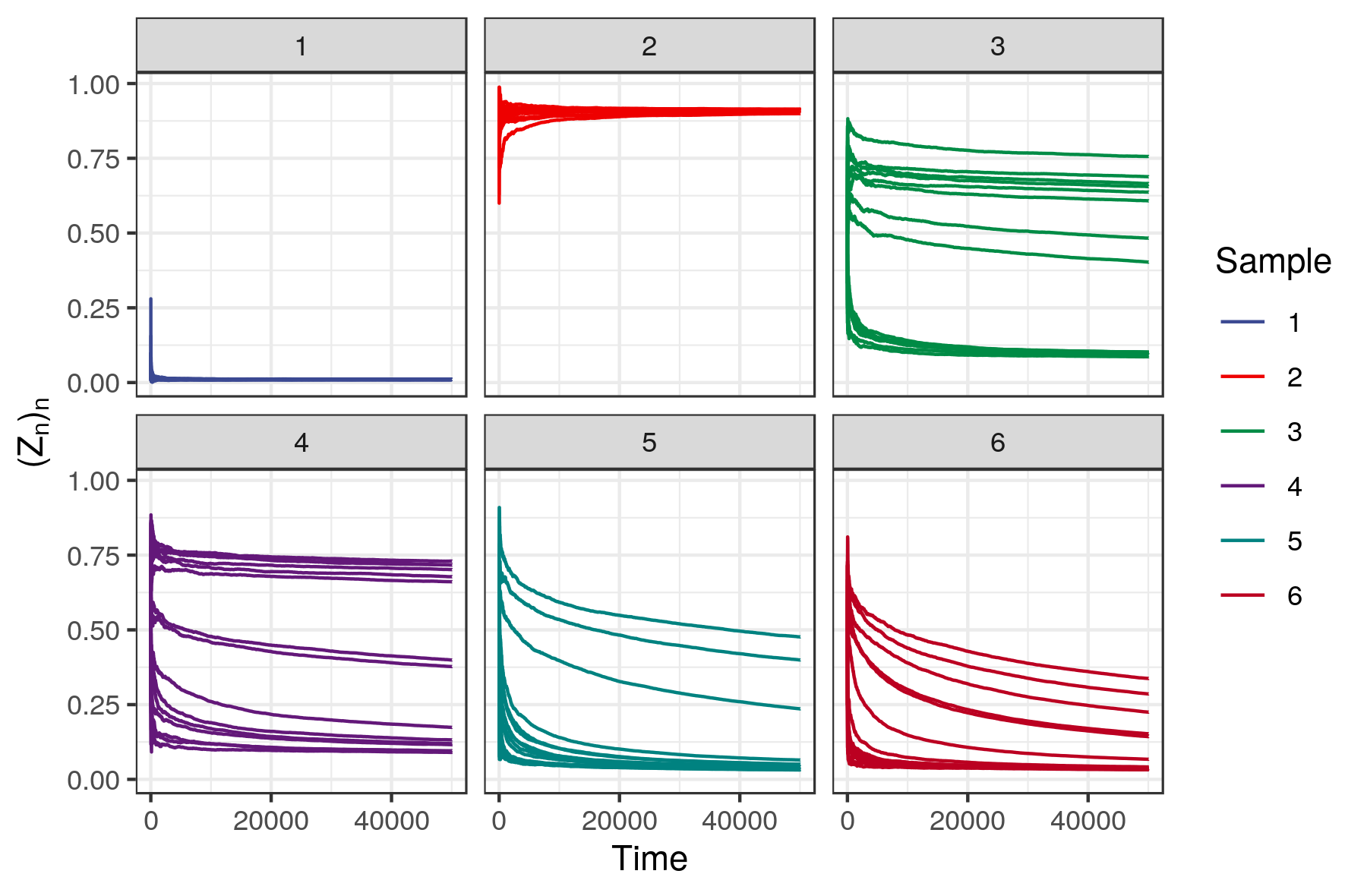}
\caption{Six samples of the system's trajectories when $N=15$. Each system starts at 1/2 on the diagonal.}
\label{ftech-setparam6-traj-Hilbert}
\end{subfigure}
\hfill 
\begin{subfigure}[t]{0.29\linewidth}
\centering
\includegraphics[scale=0.5]{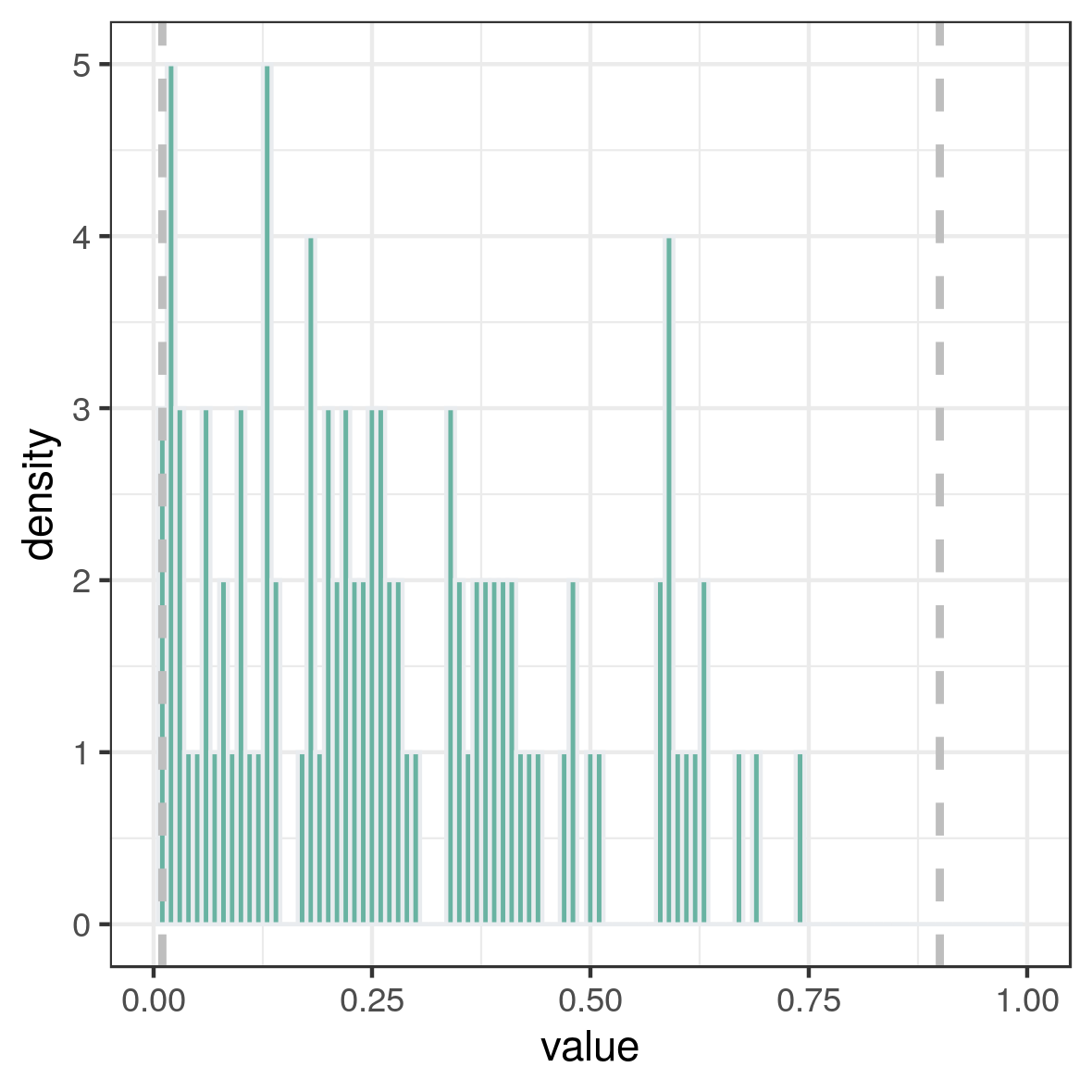}
\caption{Histogram of mean field's final values, at $T=50.000$, 
when $N=10$, sample of size $100$, uniformly distributed initial condition.}
\end{subfigure}
\\
\begin{subfigure}[t]{0.69\linewidth}
\centering
\includegraphics[width=0.8\linewidth]{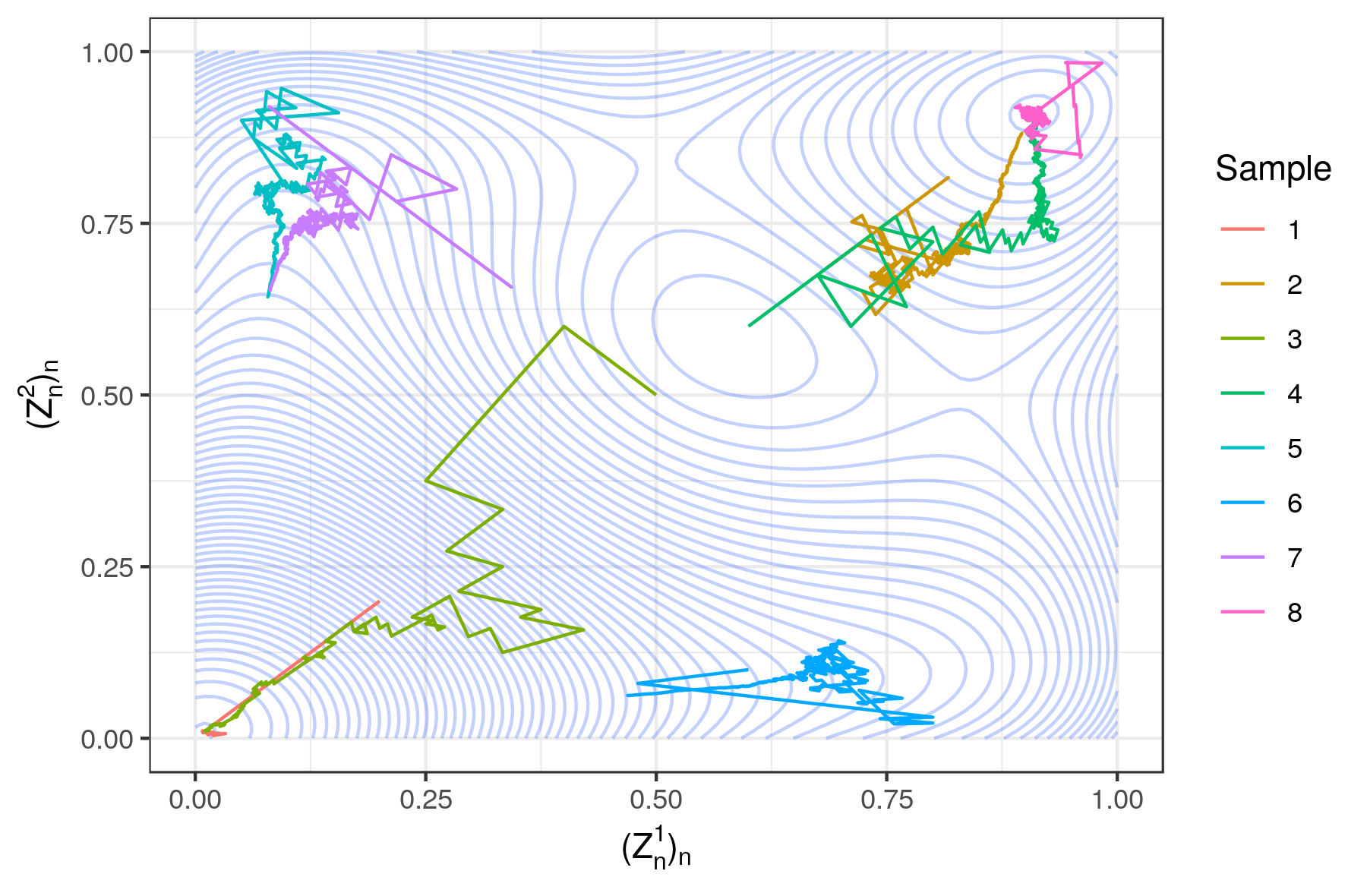}
\caption{Eight samples of the systems when $N=2$.  Representation of trajectories $(Z_n(1),Z_n(2))_n$. Each color means a different sample. 
Sample~1 starts with $(0.2,0.2)$. Sample~2 starts with $(0.7,0.7)$.
Samples~3 starts with $(0.5,0.5)$. Samples~4 starts with $(0.6,0.6)$.
Sample~5 starts with $(0.1,0.8)$. Sample~6 starts with $(0.6,0.1)$.
Sample~7 starts with $(0.1,0.9)$. Sample~8 starts with $(0.9,0.9)$.
In background some level sets of the associated~$-V$. Time up to $200.000$ iterations.}
\label{ftech-phase-space-traj-setparam6}
\end{subfigure}
\begin{subfigure}[t]{0.29\linewidth}
\centering
\includegraphics[scale=0.25]{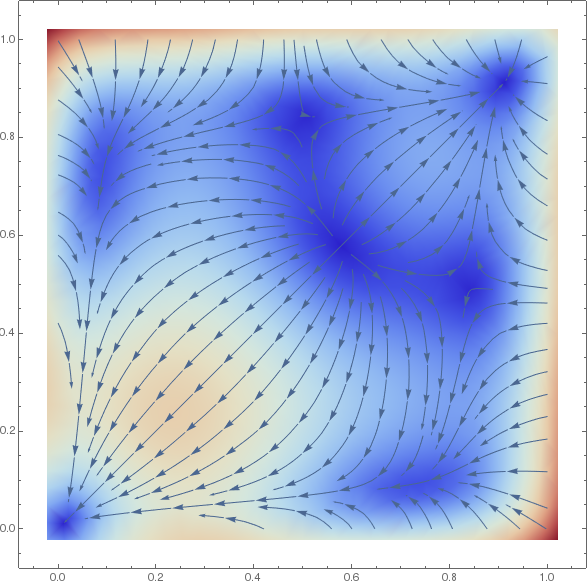}
\caption{Representation of the field  when $N=2$.}
 \label{ftech-setparam6-N2}
\end{subfigure}
\caption{Case $f=f_{\textrm{Tech}}$.  Parameters are $\theta=0.99$, $\alpha=0.15$, $\beta=0.05$, $q=0.005$. There are only two stable synchronization zero points. No-synchronization points exist but are unstable.}
\label{ftech-setparam6:figures}
\end{figure}


\begin{figure}
\centering
\begin{subfigure}[t]{0.69\textwidth}
\centering
\includegraphics[scale=0.5]{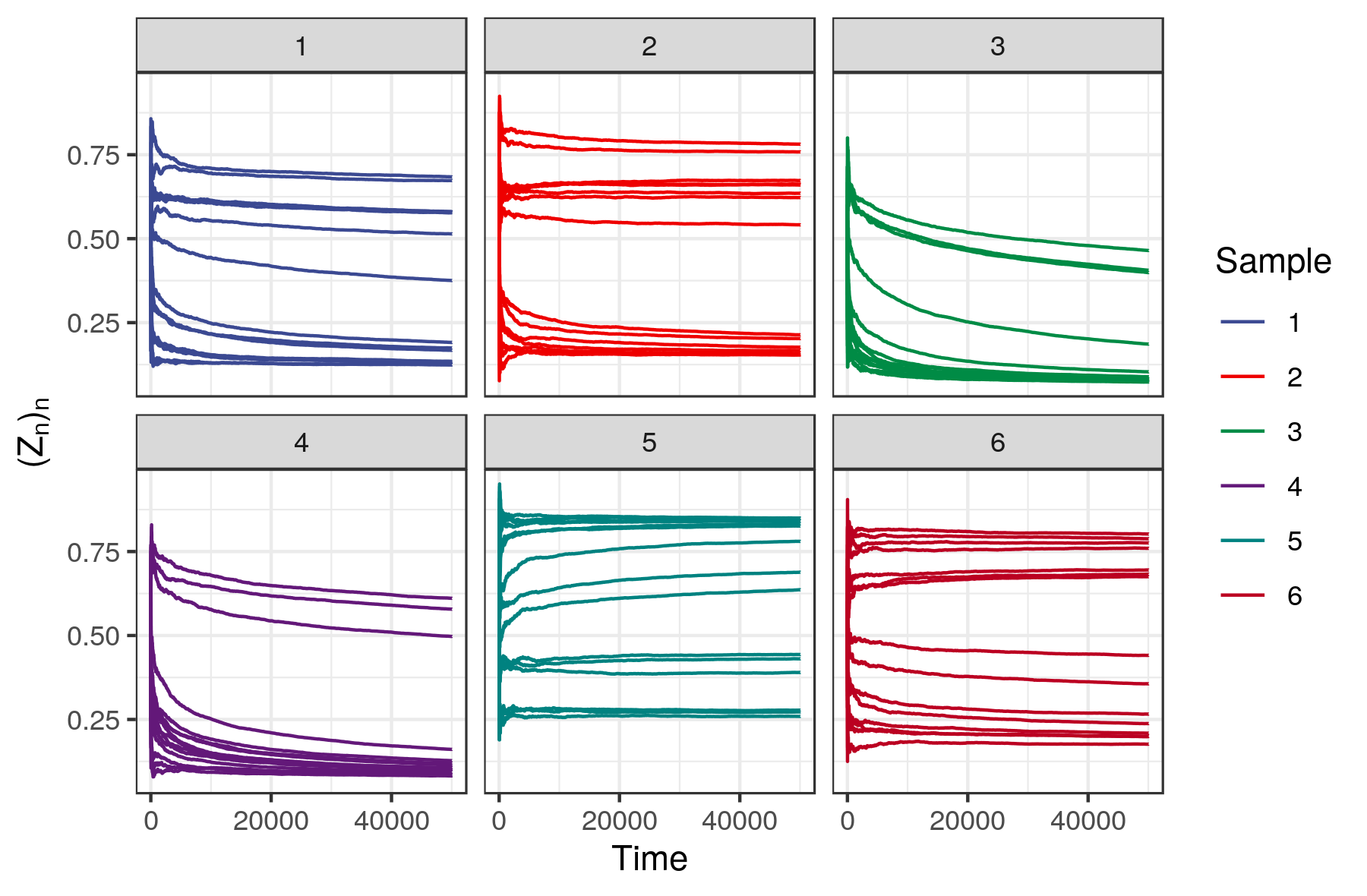}
\caption{Six samples of the system's trajectories when $N=10$. Each component start at 1/2.}
\label{ftech-setparam2-traj-Hilbert}
\end{subfigure}
\hfill 
\\
\begin{subfigure}[t]{0.69\textwidth}
\centering
\includegraphics[width=\textwidth, height=0.4\textheight]{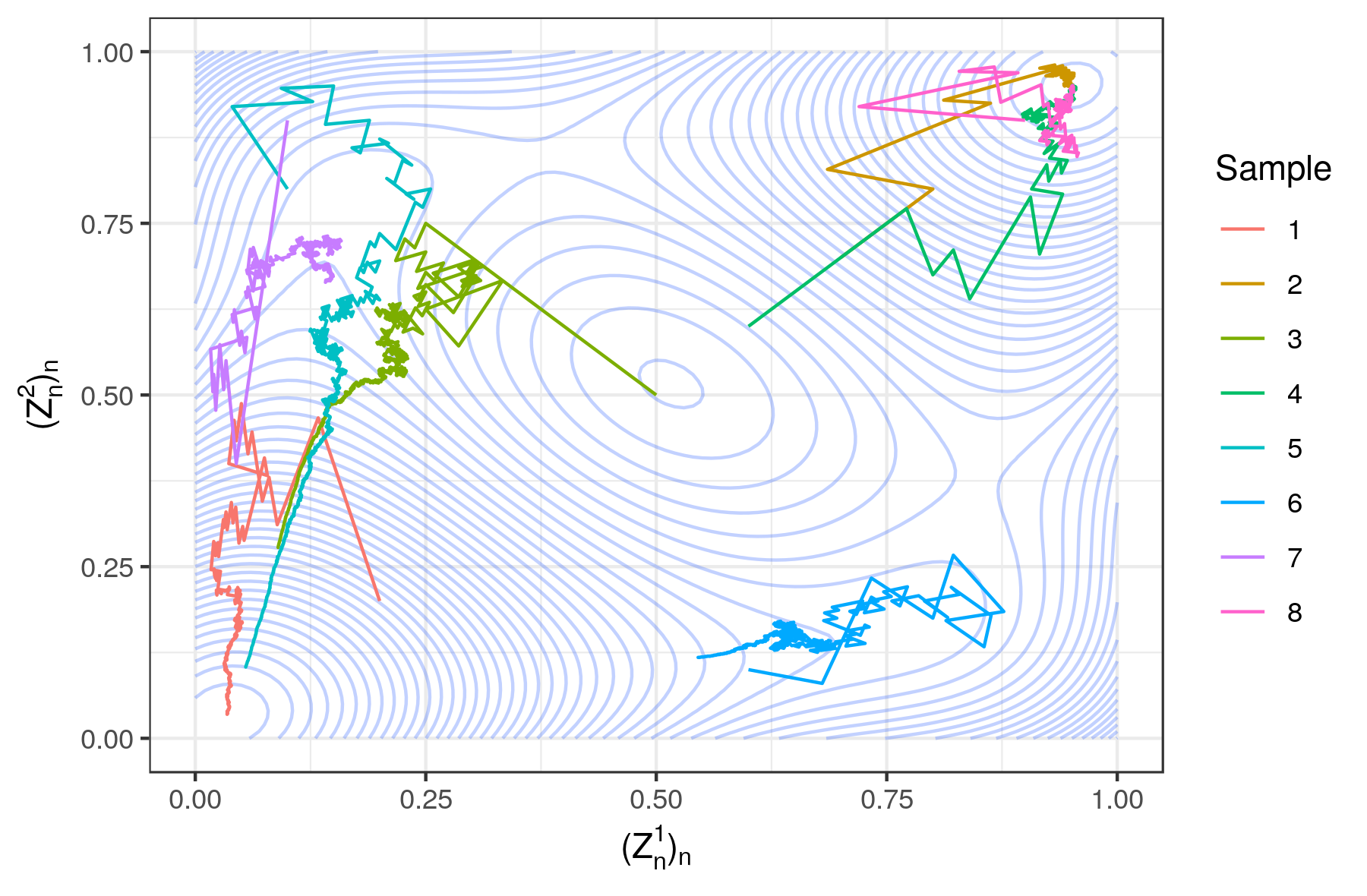}
\caption{Six samples of the systems when $N=2$.  Representation of trajectories $(Z_n(1),Z_n(2))_n$. Time up to $150.000$ iterations. Each color means a different sample. Sample~1 starts with $0.2$ on the diagonal. Sample~2 starts with $0.7$ on the diagonal.
Samples~3 to 8 start diversely, for instance sample~6 starts at $(0.6,0.1)$. In background some level sets of the associated~$-V$.} 
\label{ftech-phase-space-traj-setparam2}
\end{subfigure}
\hfill 
\begin{subfigure}[t]{0.29\textwidth}
\includegraphics[scale=0.25]{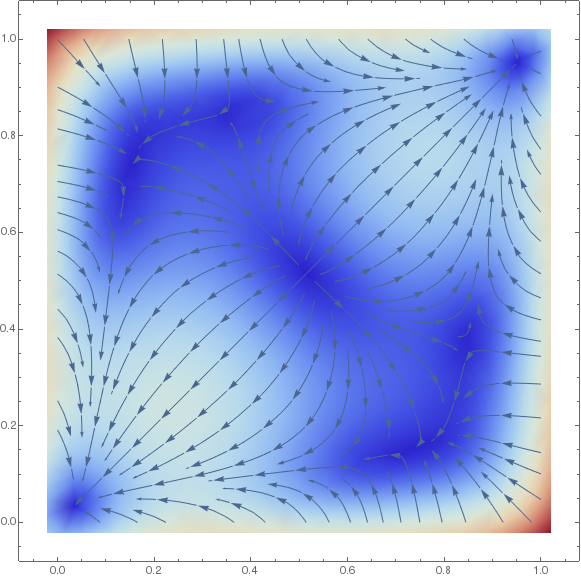}
\caption{Representation of the field $F=-\nabla V$ when $N=2$.}
\end{subfigure}
\caption{Case $f=f_{\textrm{Tech}}$. Parameters are $\theta=0.97$, $\alpha=0.18$, $\beta=0.001$, $q=0.01$. There are only two stable synchronization zero points.}
\label{ftech-setparam2:figures}
%
\end{figure}


\clearpage


\noindent {\bf Acknowledgments}\\
\noindent Irene Crimaldi and Ida Minelli are members of the Italian
Group ``Gruppo Nazionale per l'Analisi Matematica, la Probabilit\`a e
le loro Applicazioni'' of the Italian Institute ``Istituto Nazionale
di Alta Matematica''.  P.-Y. Louis acknowledges the
  International Associated Laboratory \textit{Ypatia Laboratory of
    Mathematical Sciences} (LYSM) for funding travel
  expenses. I.~G.~Minelli and P.-Y. Louis thanks IMT Lucca for
  welcoming and supporting stays at IMT.  The
  authors thanks M.~Benaim, S.~Laruelle for providing references about
  stochastic approximation results. 
  The authors warmly thanks
  R.~Pemantle for discussions about possible extension of our work to
  the case of a discontinuous function~$f$.
A.~Sarti and S.~Boissi\`ere (Laboratoire de Math\'ematiques et Applications, UMR~7348 Univ.~Poitiers et CNRS) are gratefully
  acknowledged for providing Lemma~\ref{lemma-Tech-zeri-finiti}.
\medskip 

\noindent {\bf Funding Sources}\\
\noindent Irene Crimaldi is partially supported by the Italian
``Programma di Attivit\`a Integrata'' (PAI), project ``TOol for
Fighting FakEs'' (TOFFE) funded by IMT School for Advanced Studies
Lucca.
\medskip 

\noindent {\bf Declaration}\\
\noindent All the authors developed the theoretical results, performed
the numerical simulations, contributed to the final version of the
manuscript.


\newpage  
\appendix

\section{Stochastic approximation}
\label{app-sto-approx}
Here, we briefly recall the results of Stochastic Approximation theory
used in the present work. We refer the interested reader to more
complete monographs ({\it e.g.}~\cite{benaim, borkar, del, duflo-fr, duflo,
  KusYin03, laru-page, pem}).  \\

Let $\mathbf{Z}=(\mathbf{Z}_n)_{n\geq 0}$ be an $N$-dimensional
stochastic process with values in $[0,1]^N$, adapted to a filtration
$\mathcal{F}=(\mathcal{F}_n)_{n\geq 0}$. Suppose that $\mathbf{Z}$ satisfies 
\begin{equation}\label{eq-vector-appendix}
\mathbf{Z}_{n+1}=\mathbf{Z}_n+r_n\mathbf{F}(\mathbf{Z}_n)+
r_n{\Delta \mathbf{M}}_{n+1}\,,
\end{equation}
where $r_n\sim 1/n$, $\mathbf{F}$ is a bounded $\mathcal{C}^1$
vector-valued function on an open subset $\mathcal{O}$ of
$\mathbb{R}^N$, with $[0,1]^N\subset\mathcal{O}$, and $({\Delta
  \mathbf{M}}_{n})_n$ is a bounded martingale difference with respect
to $\mathcal{F}$.  \\

We have the following results:

\begin{theorem} ({\it e.g.}~\cite{laru-page, mailler})\\
The limit set of $\mathbf{Z}$, \textit{i.e.} the set defined as $
\mathcal{L}(\mathbf{Z})= \bigcap_{n\geq 0}\overline{\bigcup_{m\geq
    n}\mathbf{Z}_m}$, is almost surely a compact connected set stable
by the flow of the differential equation
$\dot{\mathbf{z}}=\mathbf{F}(\mathbf{z})$.
\end{theorem}

Therefore, the asymptotic behaviour of the stochastic process
$\mathbf{Z}$ is related to the properties of the zero points of the
vector field $\mathbf{F}$. In the next definition, we give a
classification of these points.

\begin{definition} A zero point of $\mathbf{F}$ is a point $\mathbf{z}$ 
such that $\mathbf{F}({\mathbf{z}})=\mathbf{0}$. We denote by
$\mathcal{Z}(\mathbf{F})$ the set of all the zero points of
$\mathbf{F}$. Moreover, denoting by $J({\mathbf{F}})(\mathbf{z})$ the
Jacobian matrix of $\mathbf{F}$ computed at the point $\mathbf{z}$, we
classify the zero points of $\mathbf F$ according to the sign of the
real part of the eigenvalues of $J(\mathbf{F})(\mathbf{z})$ as follows:
\begin{itemize}
\item $\mathbf x$ is said a {\em stable} zero point if all the
  eigenvalues of $J(\mathbf{F})(\mathbf{z})$ have negative (we mean
  ``non-strictly positive'', that is $\leq 0$) real parts;
\item $\mathbf x$ is said a {\em strictly stable} zero point  if all the
  eigenvalues of $J(\mathbf{F})(\mathbf{z})$ have strictly negative real
  parts;
\item $\mathbf x$ is said a {\em linearly unstable} (or {\em
  unstable}) zero point if $J(\mathbf{F})(\mathbf{z})$ has at least one
  eigenvalue with strictly positive real part;
\item $\mathbf x$ is said a {\em repulsive} zero point if all the
  eigenvalues of $J(\mathbf{F})(\mathbf{z})$ have strictly positive
  real parts.
\end{itemize}
\end{definition}

Suppose that, for each $\mathbf{z}$, the Jacobian matrix
$J(\mathbf{F})(\mathbf{z})$ is symmetric. Then all its eigenvalues are
real and, since the sign of the scalar product $\langle
\mathbf{F}(\mathbf{z}')-\mathbf{F}(\mathbf{z}), \mathbf{z}'-\mathbf{z}
\rangle$ for $\mathbf{z}'$ in a neighborhood of $\mathbf{z}$ is
related to the property of $J(\mathbf{F})(\mathbf{z})$ of being
positive/negative (semi)definite, and this last property is related to
the sign of the eigenvalues of $J(\mathbf{F})(\mathbf{z})$, we can
state:
\begin{itemize}
\item $\mathbf{z}$ is a stable zero point if and only if $\langle
  \mathbf{F}(\mathbf{z}'),
  \mathbf{z}'-\mathbf{z} \rangle\leq 0$ for all $\mathbf{z}'$ in a
  neighborhood of $\mathbf{z}$;
\item $\mathbf{z}$ is a linearly unstable zero point if and only if, for any
  neighborhood $\mathcal{B}_{\mathbf{z}}$ of $\mathbf{z}$, there
  exists $\mathbf{z}'\in \mathcal{B}_{\mathbf{z}}$ such that $\langle
  \mathbf{F}(\mathbf{z}'), \mathbf{z}'-\mathbf{z} \rangle>0$.
\end{itemize}

\begin{theorem}\label{cond-suff} {\rm ({\it e.g.}~\cite{mailler})}\\
If there exists a stable zero point $\mathbf{z}$ of $\mathbf{F}$ such that
$$
\langle \mathbf{F}(\mathbf{Z}_n),\mathbf{Z}_n-\mathbf{z}\rangle<0
\qquad\forall n\;\mbox{with}\;\mathbf{Z}_n\neq\mathbf{z},
$$
then $\mathbf{Z}_n\stackrel{a.s.}\longrightarrow \mathbf{z}$.
\end{theorem}

\begin{theorem}\label{cond-suff-lyapunov} {\rm (\cite[Ch.~2,\, Th.~2]{del} or 
\cite[Ch.~5, Th.~2.1]{KusYin03} or~\cite[Th.~2.18]{pem})}\\ If
$\mathbf{F}=-\nabla \mathbf{V}$ and $\mathcal{Z}(\mathbf{F})$ is not
empty and finite, then there exists a random variable
$\mathbf{Z}_\infty$, which takes values in $\mathcal{Z}(\mathbf{F})$
and such that
$$
\mathbf{Z}_n\stackrel{a.s.}\longrightarrow \mathbf{Z}_\infty.
$$
\end{theorem}

\begin{theorem}\label{no-conv} 
{\rm (\cite[Th.~1]{Pemantle1990})}\\
If there exists a constant $C>0$, such that we have 
\begin{equation}\label{cond-no-conv}
E\left[ \langle \Delta
  \mathbf{M}_{n+1},\mathbf{v}\rangle)^+|\mathcal{F}_n \right] \geq
C\qquad\forall\mathbf{v}\in\mathbb{\R}^N\;\mbox{with}\; |v|=\sum_{h=1}^N v_h=1,
\end{equation}
then, for each linearly unstable zero point $\mathbf{z}$ of
$\mathbf{F}$, we have $P(\mathbf{Z}_n\to \mathbf{z})=0$.
\end{theorem}

If $\mathbf{F}$ belongs to $\mathcal{C}^2$ and, for each $\mathbf{z}$,
the Jacobian matrix $J(\mathbf{F})(\mathbf{z})$ is symmetric (and so
all its eigenvalues are real and it is diagonalizable), then from
\cite{Zhang2016} we get the following Central Limit Theorem (CLT):

\begin{theorem}\label{clt}
Suppose $\mathbf{F}\in\mathcal{C}^2$, $J(\mathbf{F})(\mathbf{z})$
symmetric for each $\mathbf{z}$. Let $\mathbf{z}_\infty\in (0,1)^N$ be
a strictly stable zero point of $\mathbf{F}$ such that
$\mathbf{Z}_n\stackrel{a.s.}\longrightarrow
\mathbf{z}_\infty$. Suppose that
\begin{equation}\label{ass-clt}
E[{\Delta \mathbf{M}}_{n+1}({\Delta \mathbf{M}}_{n+1})^{\top}|\mathcal{F}_n]
\stackrel{a.s.}\longrightarrow \Gamma,
\end{equation}
where $\Gamma=\Gamma(\mathbf{z}_\infty)$ is a deterministic symmetric
positive definite matrix.  Denote by
$\lambda=\lambda(\mathbf{z}_\infty)$ be the smallest eigenvalue of
$-J(\mathbf{F})(\mathbf{z}_\infty)$.  Then, we have:
\begin{itemize} 
\item If $\lambda>1/2$, then 
$$
\sqrt{n}(\mathbf{Z}_n-\mathbf{z}_\infty)\stackrel{d}\longrightarrow
\mathcal{N}\left(\mathbf{0}, \Sigma\right),
$$ 
where 
$$ \Sigma=\Sigma(\mathbf{z}_\infty)=\int_0^{+\infty} e^{(
  J(\mathbf{F})(\mathbf{z}_\infty) + \frac{Id}{2} )u} \Gamma e^{(
  J(\mathbf{F})(\mathbf{z}_\infty) + \frac{Id}{2} )u} \,du.
$$
\item If $\lambda=1/2$, then 
$$
\sqrt{\frac{n}{\ln(n)}}
(\mathbf{Z}_n-\mathbf{z}_\infty)\stackrel{d}\longrightarrow
\mathcal{N}\left(\mathbf{0}, \Sigma\right),
$$ 
where 
$$ \Sigma=\Sigma(\mathbf{z}_\infty)=\lim_{n\to
  +\infty}\frac{1}{\ln(n)}\int_0^{\ln(n)} e^{(
  J(\mathbf{F})(\mathbf{z}_\infty) + \frac{Id}{2} )u} \Gamma e^{(
  J(\mathbf{F})(\mathbf{z}_\infty) + \frac{Id}{2} )u} \,du.
$$
\item If $0<\lambda<1/2$, then 
$$
n^\lambda(\mathbf{Z}_n-\mathbf{z_\infty})\stackrel{a.s.}\longrightarrow V,
$$
where $V$ is a suitable finite random variable.
\end{itemize}
\end{theorem}

Note that Assumption 2.2. in~\cite{Zhang2016} is satisfied since we
take $\mathbf{F}\in \mathcal{C}^2$. Equation (2.3) of Assumption 2.3 in
\cite{Zhang2016} holds because we assume $(\Delta\mathbf{M}_n)_n$
bounded. Equation (2.4) of the same assumption is implied by the above
condition~\eqref{ass-clt}. All the assumptions on the remainder term
$\mathbf{r}_n$ in~\cite{Zhang2016} are verified because we have
$\mathbf{r}_n=\mathbf{0}$. Finally, since $J(\mathbf{F})(\mathbf{z}_\infty)$
is symmetric, also the matrix $e^{( J(\mathbf{F})(\mathbf{z}_\infty) +
  \frac{Id}{2} )u}$ is symmetric. 
\\


\begin{remark}\label{clt-alternative-expression}\rm 
With the same assumptions and notation as in Theorem~\ref{clt}, the
limit covariance matrix $\Sigma$ in the case $\lambda>1/2$ can be
rewritten using the Lyapounov equation ({\it e.g.}~\cite{fort} or
\cite{del}):
$$
\left(J(\mathbf{F})(\mathbf{z}_\infty)+\frac{1}{2}Id\right)
\Sigma+\Sigma
\left(J(\mathbf{F})(\mathbf{z}_\infty)^{\top}+\frac{1}{2}Id\right)=-\Gamma
$$ 
Since $J(\mathbf{F})(\mathbf{z}_\infty)$ is symmetric by assumption
and $\Sigma$ is symmetric by definition, we have
$$
\Sigma= \left(-2J(\mathbf{F})(\mathbf{z}_\infty)-Id\right)^{-1}
\Gamma.
$$
\end{remark}

\begin{remark}\label{clt-more-limit-points}\rm 
In~\cite{fort} we have a CLT also when there exist more limit points
for $(\mathbf{Z}_n)_n$. Indeed, under the same assumptions on
$\mathbf{F}$ as in Theorem~\ref{clt}, when condition~\eqref{ass-clt}
is satisfied and $\mathbf{z}_\infty\in (0,1)^N$ is a strictly stable
zero point of $\mathbf{F}$ such that $P(\mathbf{Z}_n\to
\mathbf{z}_\infty)>0$ and $\lambda(\mathbf{z}_\infty)>1/2$, we have to
consider the convergence in distribution under the probability measure
$P(\cdot|\mathbf{Z}_n\to \mathbf{z}_\infty)$ and the corresponding
limit distribution is the one with characteristic function
$$ \mathbf{u}\mapsto E\left[
  \exp(-\frac{1}{2}\mathbf{u}^\top\Sigma(\mathbf{z}_\infty)
  \mathbf{u})\, |\, \mathbf{Z}_n\to\mathbf{z}_\infty \right],
$$ where $\Sigma(\mathbf{z}_\infty)$ is defined as in Theorem
\eqref{clt} or, equivalently, as in Remark
\eqref{clt-alternative-expression}.
\end{remark}
\bigskip 

\section{Eigenvalues of the Jacobian matrix}
\label{app-eigen}
We observe that, letting $d_i:=(1-\alpha-\beta)f'(z_i)-1$ for
$i=1,\ldots, N$, the Jacobian matrix of $\mathbf{F}$ is given by
\begin{equation}\label{JG}
J\mathbf{F}(\mathbf{z})=
\frac{\alpha}{N}
\begin{pmatrix}
&1 &\dots &1\\
&\vdots &\ldots &\vdots\\
&1 &\ldots &1
\end{pmatrix}
+ 
diag\left(d_1,\dots, d_N\right)\,,
\end{equation} 
where $diag(d_1,\dots,d_N)$ denotes the diagonal matrix with diagonal
elements $d_1,\dots,d_N$.\\

In order to compute its eigenvalues, we use the following results: 
\begin{lemma}
Assume that the matrix $A$ has the form
$$
A=c^2
\begin{pmatrix}
&1 &\dots &1\\
&\vdots &\ldots &\vdots\\
&1 &\ldots &1
\end{pmatrix}
+ 
diag(d_1,\dots,d_N),
$$ 
where $c>0$ and $diag(d_1,\dots,d_N)$ denotes the diagonal matrix
with diagonal elements equal to $d_i$ for $i=1,\dots, N$.  Then the
characteristic polynomial of $A$ can be written as
\begin{equation}\label{plambda}
p(\lambda)=\prod_{i=1}^N(d_i-\lambda)+c^2\sum_{i=1}^N\prod_{j\neq i}(d_j-\lambda).
\end{equation}
\end{lemma}

\begin{proof}
Set $\mathbf{v}=(c,\dots,c)^{\top}$.  By the matrix determinant lemma,
we get for all $\lambda\notin\{d_1,\dots,d_N\}$
\begin{equation*}
\begin{split}
p(\lambda)&=det(A-\lambda I)=
det(c^2\mathbf{v}\mathbf{v}^{\top}+diag(d_1-\lambda,\dots,d_N-\lambda)
\\
&=\left(1+c^2\sum_{i=1}^N\frac{1}{d_i-\lambda}\right)\prod_{i=1}^N(d_i-\lambda)
=\prod_{i=1}^N(d_i-\lambda)+c^2\sum_{i=1}^N\prod_{j\neq i}(d_j-\lambda).
\end{split}
\end{equation*}
By continuity, we can conclude that $p(\lambda)$ has the above
expression for all $\lambda$.
\end{proof}

\begin{corollary}\label{cor-generale}
With the same assumptions and notation of the above lemma, the number
$d_k$ is an eigenvalue of $J\mathbf{F}(\mathbf{z})$ if and only if
there exists at least one $j\neq k$ such that $d_j=d_k$.
\end{corollary}

\begin{proof}
Clearly, if $d_k=d_j$ for at least one pair $k\neq j$ we have
$p(d_k)=0$. On the other hand, if $p(d_k)=0$ we have necessarily
$\prod_{j\neq k} (d_j-d_k)=0$ which implies $d_j=d_k$ for at least one
$j\neq k$.
\end{proof}

\begin{corollary}\label{cor-autovalori-sincro}
With the same assumption and notation notation of the above lemma, if $d_i=d$
for $i=1,\dots, N$, then the eigenvalues of $A$ are:
$$
\lambda_1=d\qquad\mbox{and}\qquad\lambda_2=d+c^2N
$$
\end{corollary}
\begin{proof}
In this case, we have 
$$
p(\lambda)=(d-\lambda)^{N-1}(d-\lambda+c^2N)
$$ 
and so $d$ is an eigenvalue with multiplicity $N-1$ and $d+c^2N$
is an eigenvalue with multiplicity $1$.
\end{proof}

\begin{corollary}\label{cor-autovalori-no-sincro}
With the same assumptions and notation of the above lemma, suppose that
$d_i\in\{D_1, D_2\}$, with $D_1\neq D_2$, for all $i=1,\dots,N$, and
assume that $d_i\neq d_j$ for at least a pair of different indexes. 
Moreover, denote by $N_1\in \{1,\dots, N-1\}$ the number of indexes
such that $d_i=D_1$ and $N_2=N-N_1$. The eigenvalues of
$J\mathbf{F}(\mathbf{z})$ are
\begin{itemize}
\item $\lambda=D_1$ with multiplicity $N_1-1$;
\item $\lambda=D_2$ with multiplicity $N_2-1$;
\item $\lambda=\lambda_{i}$ with $i=1,2$, where the $\lambda_i$'s are
  the solutions of the equation
\begin{equation}\label{eq}
\lambda^2-(D_1+D_2+c^2N)\lambda+D_1D_2+c^2 N_1 D_2+c^2 N_2 D_1=0.
\end{equation}
\end{itemize}
Then, in particular:
\begin{itemize}
\item[(a)] If $D_1\geq 0$ and $D_2\geq 0$, then all the eigenvalues
  are positive.
\item[(b)] If $D_1=0$ and $D_2<0$ (or vice-versa), then there exists 
a strictly positive eigenvalue.
\item[(c)] When $D_1<0$, $D_2<0$, all the eigenvalues are
  negative if and only if we have 
\begin{equation}\label{negative-case}
D_1+D_2+c^2N < 0 \qquad\mbox{and}\qquad 
D_1D_2
+c^2 N_1D_2+c^2 N_2D_1\geq  0.
\end{equation}
\end{itemize}
\end{corollary}

\begin{proof}
We first observe that the matrix $A$ is symmetric, and so all its
eigenvalues are real numbers and this, together with formula
\eqref{plambda}, proves the first assertion.  Let us write the
ploynomial in~\eqref{eq} as $r(\lambda)=\lambda^2-B\lambda +C$ where
$B=D_1+D_2+c^2N$ and $C=D_1D_2+c^2 N_1D_2+c^2 N_2D_1$.  For statement
(a) observe that, if $D_1, D_2\geq 0$, then $B>0$ and $C\geq 0$ and
this implies that the zeros of $r(\lambda)$ are both positive.
Similarly, in case (b), if $D_1=0, D_2<0$ (or vicevera), we have $C<0$
and so one of the zeros of $r(\lambda)$ must be strictly
positive. Finally, in case (c), it is enough to observe that the zeros
of $r(\lambda)$ are both negative if and only if $B<0$ and $C\geq 0$.
\end{proof}

\begin{remark}\label{cor-autovalori-no-sincro-rem}
A necessary condition for~\eqref{negative-case} is
\begin{equation}\label{negative-nec}
D_i<-c^2 N_i\quad\mbox{ for } i=1,2\,;
\end{equation}
while a sufficient condition is 
\begin{equation}\label{negative-suf}
D_i\leq -c^2 N_i (1+\delta_i)\quad\mbox{ for } i=1,2,
\end{equation}
with $\delta_1,\,\delta_2>0$ and $\delta_1\delta_2\geq 1$.\\ Indeed,
observe that the second condition in~\eqref{negative-case} can be
written as $(D_1+c^2 N_1)(D_2+c^2 N_2)-c^4 N_1 N_2\geq 0$, which
implies $(D_1+c^2 N_1)(D_2+c^2 N_2)>0$ (because $c,\,N_1,\,N_2>0$) and
so we have either $D_i<-c^2 N_i$ for both $i=1,2$ or $D_i>-c^2 N_i$
for both $i=1,2$ and the first equation in~\eqref{negative-case}
excludes the second case. Hence we necessarily have $D_i<-c^2 N_i$ for
$i=1,2$. On the other hand, a simple computation shows that condition
\eqref{negative-suf} implies~\eqref{negative-case}: we have
$D_1+D_2+c^2N\leq -c^2(N_1\delta_1+N_2\delta_2)<0$ and
$(D_1+c^2N_1)(D_2+c^2N_2)-c^4N_1N_2=
(-D_1-c^2N_1)(-D_2-c^2N_2)-c^4N_1N_2\geq
c^4N_1N_2(\delta_1\delta_2-1)\geq 0$.
\end{remark}

\begin{remark}\label{remark-special-delta}
\rm If in~\eqref{negative-suf} we take
$N_1(1+\delta_1)=N_2(1+\delta_2)=N$ (that is $\delta_i=NN_i-1\geq 1$),
we obtain 
$$ 
D_i\leq -c^2N\quad\forall i=1,2.
$$
\end{remark}

\newpage 
\section{Roots of polynoms systems}

\begin{lemma}\label{lemma-Tech-zeri-finiti}
Assume $f$ is a real polynom of degree~$d \geq 2$ such that
 $F=(F_1,\hdots,F_N) : [0,1]^N \mapsto [0,1]^N$ ($N \geq 1$) 
 with   $$F_i(z) =\alpha \bar{z} +\beta q +(1-\alpha -\beta) f(z_i)-z_i$$
 where $\alpha,\beta \in [0,1]^2$ such that $1-\alpha-\beta \neq 0$.
Then the set $\mathcal{Z}(\mathbf{F})$ of roots of the system $(F_1,\hdots,F_N)$ is finite.
\end{lemma}
\begin{proof}
 Let $\mathcal{Z}_{\mathbb C}(\mathbf{F})$ be the algebraic set of the solutions $z\in \mathbb C$ such that $\forall i, 1\leq i \leq N$, $F_i(z)=0$.
 Let $I$ be the ideal generated by the polynomials $(F_1,\hdots,F_N)$ in $\mathbb C[z_1,\hdots,z_N]$.
 Let $I(\mathcal{Z}_{\mathbb C}(\mathbf{F}))$ be the ideal generated by all polynomials
from  $\mathbb C[z_1,\hdots,z_N]$ vanishing in $\mathcal{Z}_{\mathbb C}(\mathbf{F})$.
It holds $I \subset I(\mathcal{Z}_{\mathbb C}(\mathbf{F}))$.
Using Corollary~2.15 
in~\cite{eisenbud2013commutative}, we get
$ \mathcal{Z}_{\mathbb C}(\mathbf{F})$ is finite if and only if 
the dimension of $\faktor{\mathbb C[z_1,\hdots,z_N] }{I(\mathcal{Z}_{\mathbb C}(\mathbf{F}))}$ as $\mathbb C$-vector space is finite. 
 Since $I \subset I(\mathcal{Z}_{\mathbb C}(\mathbf{F}))$ there is a surjective morphism
 from  $\faktor{\mathbb C[z_1,\hdots,z_N] }{I}$ 
 to  $\faktor{\mathbb C[z_1,\hdots,z_N] }{I(\mathcal{Z}_{\mathbb C}(\mathbf{F}))}$.
 Thus, it is enough to state that
  the dimension of $\faktor{\mathbb C[z_1,\hdots,z_N] }{I}$  as $\mathbb C$ vector space is finite.
  Since $f$ is a polynom of degree~$d \geq 2$, it remains in 
  $\faktor{\mathbb C[z_1,\hdots,z_N] }{I}$  
  the monomes $z_1^{a_1} z_2^{a_2} z_N^{a_N}$ 
  (with $\forall i, 1\leq i \leq N$, $1 \leq a_i \leq (d-1)$) 
  whose cardinal is $d^N$. 
  Thus the cardinal of $ \mathcal{Z}_{\mathbb C}(\mathbf{F})$ is bounded from above by $d^N$. Since it holds on the field~$\mathbb C$, it holds on the field~$\mathbb R$.
\end{proof}

\bibliographystyle{abbrv} 
\bibliography{BibGamesLucca.bib}

\begin{thebibliography}{10}

\bibitem{AleMinelli2020}
M.~Aleandri and I.~G. Minelli.
\newblock Delay-induced periodic behavior in competitive populations.
\newblock Arxiv 2008.00257 August 2020.

\bibitem{aleandri2019opinion}
M.~Aleandri and I.~G. Minelli.
\newblock Opinion dynamics with lotka-volterra type interactions.
\newblock {\em Electronic Journal of Probability}, 24, 2019.

\bibitem{ale-cri-rescaled}
G.~Aletti and I.~Crimaldi.
\newblock An urn model with local reinforcement: a theoretical framework for a
  chi-squared goodness of fit test with a big sample.
\newblock {\em arXiv 1906.10951, submitted}, 2019.

\bibitem{ale-cri-ghi}
G.~Aletti, I.~Crimaldi, and A.~Ghiglietti.
\newblock Synchronization of reinforced stochastic processes with a
  network-based interaction.
\newblock {\em Annals of Applied Probability}, 27:3787--3844, 2017.

\bibitem{ale-cri-ghi-MEAN}
G.~Aletti, I.~Crimaldi, and A.~Ghiglietti.
\newblock Networks of reinforced stochastic processes: {A}symptotics for the
  empirical means.
\newblock {\em Bernoulli}, 25(4B):3339--3378, 2019.

\bibitem{ale-cri-ghi-WEIGHT-MEAN}
G.~Aletti, I.~Crimaldi, and A.~Ghiglietti.
\newblock {Interacting Reinforced Stochastic Processes: {S}tatistical Inference
  based on the Weighted Empirical Means}.
\newblock {\em Bernoulli}, 26(2):1098--1138, 2020.

\bibitem{ale-ghi}
G.~Aletti and A.~Ghiglietti.
\newblock Interacting generalized {F}riedman's urn systems.
\newblock {\em Stochastic Process. Appl.}, 127:2650--2678, 2017.

\bibitem{ale-ghi-ros}
G.~Aletti, A.~Ghiglietti, and W.~F. Rosenberger.
\newblock Nonparametric covariate-adjusted response-adaptive design based on a
  functional urn model.
\newblock {\em Annals of Statistics}, 46(6B):3838--3866, Dec. 2018.

\bibitem{ale-ghi-vid}
G.~Aletti, A.~Ghiglietti, and A.~N. Vidyashankar.
\newblock Dynamics of an adaptive randomly reinforced urn.
\newblock {\em Bernoulli}, 24(3):2204--2255, 2018.

\bibitem{alos}
C.~Alos~Ferrer and S.~Weidenholzer.
\newblock Contagion and efficiency.
\newblock {\em Journal of Economic Theory}, 143(1):251--274, 2008.

\bibitem{arenas}
A.~Arenas, A.~D{\'{\i}}az-Guilera, J.~Kurths, Y.~Moreno, and C.~Zhou.
\newblock Synchronization in complex networks.
\newblock {\em Physics Reports}, 469(3):93--153, 2008.

\bibitem{arthur2015process}
B.~Arthur, S.~Durlauf, and D.~A. Lane.
\newblock Process and the emergence in the economy.
\newblock {\em The economy as an evolving complex system II}, 2015.

\bibitem{Arthur107}
W.~B. Arthur.
\newblock {C}omplexity and the {E}conomy.
\newblock {\em Science}, 284(5411):107--109, 1999.

\bibitem{arthur-1983}
W.~B. Arthur, Y.~Ermoliev, and Y.~Kaniovski.
\newblock A generalized urn problem and its applications.
\newblock {\em Kibernetika}, 1:49--56, 1983.

\bibitem{Arthur-urn}
W.~B. Arthur, Y.~M. Ermoliev, and Y.~M. Kaniovski.
\newblock {Non-linear urn processes: Asymptotic behavior and applications}.
\newblock {\em IIASA Working Paper}, 1987.

\bibitem{axelrod}
R.~Axelrod.
\newblock {\em {The Complexity of Cooperation: Agent-Based Models of
  Competition and Collaboration}}.
\newblock Princeton University Press, 1997.

\bibitem{benaim}
M.~Bena{\"{i}}m.
\newblock {Dynamics of Stochastic Approximation Algorithms}.
\newblock {\em S{\'e}minaire de Probabilit{\'e}}, 33:1--68, 1999.

\bibitem{ben}
M.~Bena{\"{\i}}m, I.~Benjamini, J.~Chen, and Y.~Lima.
\newblock A generalized {P}\'olya's urn with graph based interactions.
\newblock {\em Random Struct. Algor.}, 46(4):614--634, 2015.

\bibitem{ber-cri-pra-rig-barriere}
P.~Berti, I.~Crimaldi, L.~Pratelli, and P.~Rigo.
\newblock Asymptotics for randomly reinforced urns with random barriers.
\newblock {\em J. Appl. Probab.}, 53(4):1206--1220, 2016.

\bibitem{bilancini}
E.~Bilancini and L.~Boncinelli.
\newblock The evolution of conventions under condition-dependent mistakes.
\newblock {\em Economic Theory}, 2019.

\bibitem{blume}
L.~E. Blume.
\newblock {T}he {S}tatistical {M}echanics of {S}trategic {I}nteraction.
\newblock {\em Games and Economic Behavior}, 5(3):387--424, 1993.

\bibitem{Bon02}
E.~Bonabeau.
\newblock {A}gent-based modeling: {M}ethods and techniques for simulating human
  systems.
\newblock {\em Proceedings of the national academy of sciences}, 99(suppl
  3):7280--7287, 2002.

\bibitem{BL03}
P.~Bonacich and T.~M. Liggett.
\newblock {A}symptotics of a matrix valued {M}arkov chain arising in sociology.
\newblock {\em Stochastic Processes and their Applications}, 104(1):155--171,
  2003.

\bibitem{borkar}
V.~S. Borkar.
\newblock {\em Stochastic approximation: a dynamical systems viewpoint},
  volume~48.
\newblock Springer, 2009.

\bibitem{bottazzi2007modeling}
G.~Bottazzi, G.~Dosi, G.~Fagiolo, and A.~Secchi.
\newblock Modeling industrial evolution in geographical space.
\newblock {\em Journal of Economic Geography}, 7(5):651--672, 2007.

\bibitem{challet2000statistical}
D.~Challet, M.~Marsili, and R.~Zecchina.
\newblock {Statistical mechanics of systems with heterogeneous agents: Minority
  games}.
\newblock {\em Physical Review Letters}, 84(8):1824, 2000.

\bibitem{che-luc}
J.~Chen and C.~Lucas.
\newblock A generalized {P}\'olya's urn with graph based interactions:
  convergence at linearity.
\newblock {\em Electron. Commun. Probab.}, 19, 2014.

\bibitem{chen-kuba}
M.-R. Chen and M.~Kuba.
\newblock On generalized {P}\'olya urn models.
\newblock {\em J. Appl. Prob.}, 50:1169--1186, 2013.

\bibitem{cir}
P.~Cirillo, M.~Gallegati, and J.~H{\"u}sler.
\newblock A {P}\'olya lattice model to study leverage dynamics and contagious
  financial fragility.
\newblock {\em Advances in Complex Systems}, 15(suppl. 2), 2012.

\bibitem{collet2015collective}
F.~Collet, P.~Dai~Pra, and M.~Formentin.
\newblock Collective periodicity in mean-field models of cooperative behavior.
\newblock {\em Nonlinear Differential Equations and Applications NoDEA},
  22(5):1461--1482, 2015.

\bibitem{collet2016rhythmic}
F.~Collet, M.~Formentin, and D.~Tovazzi.
\newblock Rhythmic behavior in a two-population mean-field ising model.
\newblock {\em Physical Review E}, 94(4):042139, 2016.

\bibitem{collevecchio}
A.~Collevecchio, C.~Cotar, and M.~LiCalzi.
\newblock On a preferential attachment and generalized {P}\'olya's urn model.
\newblock {\em Ann. Appl. Prob.}, 23:1219--1253, 2013.

\bibitem{cri-ipergeom}
I.~Crimaldi.
\newblock Central limit theorems for a hypergeometric randomly reinforced urn.
\newblock {\em J. Appl. Prob.}, 53(3):899--913, 2016.

\bibitem{cri-dai-lou-min}
I.~Crimaldi, P.~Dai~Pra, P.-Y. Louis, and I.~G. Minelli.
\newblock Synchronization and functional central limit theorems for interacting
  reinforced random walks.
\newblock {\em Stochastic Processes and their Applications}, 129(1):70--101,
  2019.

\bibitem{cri-dai-min}
I.~Crimaldi, P.~Dai~Pra, and I.~G. Minelli.
\newblock Fluctuation theorems for synchronization of interacting {P}\'olya's
  urns.
\newblock {\em Stochastic Process. Appl.}, 126(3):930--947, 2016.

\bibitem{dai-lou-min}
P.~Dai~Pra, P.-Y. Louis, and I.~G. Minelli.
\newblock Synchronization via interacting reinforcement.
\newblock {\em J. Appl. Probab.}, 51(2):556--568, 2014.

\bibitem{del}
B.~Delyon.
\newblock {S}tochastic approximation with decreasing gain: {C}onvergence and
  asymptotic theory.
\newblock {\em Technical report}, 2000.

\bibitem{dosi}
G.~Dosi, Y.~Ermoliev, and Y.~Kaniovski.
\newblock {Generalized urn schemes and technological dynamics}.
\newblock {\em Journal of Mathematical Economics}, 23(1):1--19, 1994.

\bibitem{dosi2006evolutionary}
G.~Dosi, G.~Fagiolo, and A.~Roventini.
\newblock An evolutionary model of endogenous business cycles.
\newblock {\em Computational Economics}, 27(1):3--34, 2006.

\bibitem{Duffy}
J.~Duffy and E.~Hopkins.
\newblock {Learning, information, and sorting in market entry games: theory and
  evidence. Games and Economic behavior}.
\newblock {\em XXX}, 51(1):31--62, 2005.

\bibitem{duflo-fr}
M.~Duflo.
\newblock {\em {M{\'e}thodes r{\'e}cursives al{\'e}atoires}}.
\newblock Masson, Paris, 1990.

\bibitem{duflo}
M.~Duflo.
\newblock {\em {Random Iterative Models}}, volume~34 of {\em {Stochastic
  Modelling and Applied Probability}}.
\newblock Springer Berlin Heidelberg, Berlin, Heidelberg, 1997.

\bibitem{egg-pol}
F.~Eggenberger and G.~P\'{o}lya.
\newblock {\"{U}}ber die {S}tatistik verketteter {V}org\"{a}nge.
\newblock {\em Z. Angewandte Math. Mech.}, 3:279--289, 1923.

\bibitem{EL04}
B.~M. Eidelson and I.~Lustick.
\newblock {V}ir-pox: {A}n agent-based analysis of smallpox preparedness and
  response policy.
\newblock {\em Journal of Artificial Societies and Social Simulation}, 7(3),
  2004.

\bibitem{eisenbud2013commutative}
D.~Eisenbud.
\newblock {\em {C}ommutative {A}lgebra: with a view toward algebraic geometry},
  volume 150.
\newblock Springer Science \& Business Media, 2013.

\bibitem{Ell93}
G.~Ellison.
\newblock Learning, local interaction, and coordination.
\newblock {\em Econometrica: Journal of the Econometric Society}, pages
  1047--1071, 1993.

\bibitem{Fagiolo}
G.~Fagiolo.
\newblock {A Note on Equilibrium Selection in Polya's Urn Coordination Games}.
\newblock {\em Economics Bullettin}, 3(45):1--14, 2005.

\bibitem{fagiolo2005endogenous}
G.~Fagiolo.
\newblock Endogenous neighborhood formation in a local coordination model with
  negative network externalities.
\newblock {\em Journal of Economic Dynamics and Control}, 29(1-2):297--319,
  2005.

\bibitem{fagiolo2003exploitation}
G.~Fagiolo and G.~Dosi.
\newblock Exploitation, exploration and innovation in a model of endogenous
  growth with locally interacting agents.
\newblock {\em Structural Change and Economic Dynamics}, 14(3):237--273, 2003.

\bibitem{fagiolo2004matching}
G.~Fagiolo, G.~Dosi, and R.~Gabriele.
\newblock Matching, bargaining, and wage setting in an evolutionary model of
  labor market and output dynamics.
\newblock {\em Advances in Complex Systems}, 7(02):157--186, 2004.

\bibitem{fort}
{Fort, Gersende}.
\newblock Central limit theorems for stochastic approximation with controlled
  {M}arkov chain dynamics.
\newblock {\em ESAIM: PS}, 19:60--80, 2015.

\bibitem{fortini}
S.~Fortini, S.~Petrone, and P.~Sporysheva.
\newblock {On a notion of partially conditionally identically distributed
  sequences}.
\newblock {\em Stoch. Process. Appl.}, 128:819--846, 2018.

\bibitem{galam1986majority}
S.~Galam.
\newblock Majority rule, hierarchical structures, and democratic
  totalitarianism: A statistical approach.
\newblock {\em Journal of Mathematical Psychology}, 30(4):426--434, 1986.

\bibitem{ghi-vid-ros}
A.~Ghiglietti, A.~N. Vidyashankar, and W.~F. Rosenberger.
\newblock Central limit theorem for an adaptive randomly reinforced urn model.
\newblock {\em Ann. Appl. Probab.}, 27(5):2956--3003, 2017.

\bibitem{ieee-paper}
M.~Hayhoe, F.~Alajaji, and B.~Gharesifard.
\newblock {A} {P}olya urn-based model for epidemics on networks.
\newblock In {\em 2017 American Control Conference (ACC)}, pages 358--363,
  2017.

\bibitem{KMR93}
M.~Kandori, G.~J. Mailath, and R.~Rob.
\newblock Learning, mutation, and long run equilibria in games.
\newblock {\em Econometrica: Journal of the Econometric Society}, pages 29--56,
  1993.

\bibitem{KusYin03}
H.~J. Kushner and G.~G. Yin.
\newblock {\em Stochastic approximation and recursive algorithms and
  applications}, volume~35 of {\em Applications of Mathematics (New York)}.
\newblock Springer-Verlag, New York, second edition, 2003.
\newblock Stochastic Modelling and Applied Probability.

\bibitem{laru-page}
S.~Laruelle and G.~Pag\`es.
\newblock Randomized urn models revisited using stochastic approximation.
\newblock {\em Ann. Appl. Prob.}, 23:1409--1436, 2013.

\bibitem{mailler}
N.~Lasmar, C.~Mailler, and O.~Selmi.
\newblock {Multiple drawing multi-colour urns by stochastic approximation}.
\newblock {\em J. Appl. Probab.}, 55(1):254--281, 2018.

\bibitem{lewis}
D.~K. Lewis.
\newblock {\em {C}onvention: {A} {P}hilosophical {S}tudy}.
\newblock Wiley-Blackwell, 1969.

\bibitem{lima}
Y.~Lima.
\newblock Graph-based {P}\'olya's urn: completion of the linear case.
\newblock {\em Stoch. Dyn.}, 16(2), 2016.

\bibitem{LouisMinelli}
P.-Y. Louis and I.~G. Minelli.
\newblock {S}ynchronization in {I}nteracting {R}einforced {S}tochastic
  {P}rocesses.
\newblock In {\em Emergence, Complexity and Computation}, pages 105--118.
  Springer International Publishing, 2018.

\bibitem{LouisMirebrahimi}
P.-Y. Louis and M.~Mirebrahimi.
\newblock Synchronization and fluctuations for interacting stochastic systems
  with individual and collective reinforcement.
\newblock Preprint hal-01856584v3, 2018.

\bibitem{mah}
H.~M. Mahmoud.
\newblock {\em {P}\'olya urn models}.
\newblock Texts in Statistical Science Series. CRC Press, Boca Raton, FL, 2009.

\bibitem{martins2008continuous}
A.~C. Martins.
\newblock Continuous opinions and discrete actions in opinion dynamics
  problems.
\newblock {\em International Journal of Modern Physics C}, 19(04):617--624,
  2008.

\bibitem{martins2013trust}
A.~C. Martins.
\newblock Trust in the coda model: Opinion dynamics and the reliability of
  other agents.
\newblock {\em Physics Letters A}, 377(37):2333--2339, 2013.

\bibitem{mas2016behavioral}
M.~M{\"a}s and H.~H. Nax.
\newblock {A behavioral study of ``noise'' in coordination games}.
\newblock {\em Journal of Economic Theory}, 162:195--208, 2016.

\bibitem{OMH}
J.~Orbell, T.~Morikawa, J.~Hartwig, J.~Hanley, and N.~Allen.
\newblock "{M}achiavellian" intelligence as a basis for the evolution of
  cooperative dispositions.
\newblock {\em American Political Science Review}, pages 1--15, 2004.

\bibitem{pag-sec}
A.~M. Paganoni and P.~Secchi.
\newblock Interacting reinforced-urn systems.
\newblock {\em Adv. in Appl. Probab.}, 36(3):791--804, 2004.

\bibitem{Pemantle1990}
R.~Pemantle.
\newblock {Nonconvergence to Unstable Points in Urn Models and Stochastic
  Approximations}.
\newblock {\em Annals of Probability}, 18(2):698--712, 1990.

\bibitem{pem}
R.~Pemantle.
\newblock A survey of random processes with reinforcement.
\newblock {\em Probab. Surv.}, 4:1--79, 2007.

\bibitem{sah}
N.~Sahasrabudhe.
\newblock {Synchronization and fluctuation theorems for interacting Friedman
  urns}.
\newblock {\em J. Appl. Probab.}, 53(4):1221--1239, 2016.

\bibitem{silverman2018methodological}
E.~Silverman.
\newblock {\em Methodological investigations in agent-based modelling: with
  applications for the social sciences}.
\newblock Springer Nature, 2018.

\bibitem{SP00}
B.~Skyrms and R.~Pemantle.
\newblock A dynamic model of social network formation.
\newblock In {\em Adaptive networks}, pages 231--251. Springer, 2009.

\bibitem{TRACY1992278}
N.~D. Tracy and J.~W. Seaman.
\newblock Urn models for evolutionary learning games.
\newblock {\em Journal of Mathematical Psychology}, 36(2):278--282, 1992.

\bibitem{Zhang2016}
L.-X. Zhang.
\newblock Central limit theorems of a recursive stochastic algorithm with
  applications to adaptive designs.
\newblock {\em Ann. Appl. Probab.}, 26(6):3630--3658, 2016.

\end{thebibliography}

\end{document}